\documentclass[reqno,10pt]{amsart}
\usepackage{amsfonts}
\usepackage{a4wide}

\usepackage{amsmath,amssymb,amsthm,amsfonts}
\usepackage{mathrsfs}
\usepackage{bbm}
\usepackage{hyperref}

\usepackage{color}

\newtheorem{lemma}{Lemma}[section]
\newtheorem{theorem}{Theorem}[section]

\newtheorem{remark}{Remark}[section]

\numberwithin{equation}{section}

\newcommand{\dis}{\displaystyle}

\newcommand{\R}{\mathbb{R}}

\newcommand{\Z}{\mathbb{Z}}

\renewcommand{\S}{\mathbb{S}}
\newcommand{\T}{\mathbb{T}}

\newcommand{\FH}{\mathbf{H}}

\newcommand{\FP}{\mathbf{P}}

\newcommand{\FX}{\mathbf{X}}
\newcommand{\FY}{\mathbf{Y}}

\newcommand{\Fb}{\mathbf{b}}
\newcommand{\Fk}{\mathbf{k}}

\newcommand{\SL}{\mathscr{L}}

\newcommand{\CA}{\mathcal{A}}

\newcommand{\CC}{\mathcal{C}}
\newcommand{\CE}{\mathcal{E}}
\newcommand{\CF}{\mathcal{F}}
\newcommand{\CG}{\mathcal{G}}
\newcommand{\CH}{\mathcal{H}}
\newcommand{\CI}{\mathcal{I}}

\newcommand{\CK}{\mathcal{K}}
\newcommand{\CL}{\mathcal{L}}
\newcommand{\CN}{\mathcal{N}}
\newcommand{\CM}{\mathcal{M}}

\newcommand{\CS}{\mathcal{S}}

\newcommand{\RA}{\mathscr{A}}
\newcommand{\RB}{\mathscr{B}}

\newcommand{\RH}{\mathfrak{H}}
\newcommand{\RI}{\mathfrak{I}}

\newcommand{\RR}{\mathfrak{R}}

\newcommand{\RT}{\mathfrak{T}}

\newcommand{\na}{\nabla}

\newcommand{\al}{\alpha}
\newcommand{\bet}{\beta}
\newcommand{\ga}{\gamma}
\newcommand{\om}{\omega}
\newcommand{\Om}{\Omega}
\newcommand{\la}{\lambda}
\newcommand{\de}{\delta}
\newcommand{\si}{\sigma}
\newcommand{\pa}{\partial}
\newcommand{\ka}{\kappa}
\newcommand{\eps}{\epsilon}

\newcommand{\ta}{\theta}
\newcommand{\vth}{\vartheta}

\newcommand{\vps}{\varepsilon}

\newcommand{\ze}{\zeta}

\newcommand{\Ga}{\Gamma}

\newcommand{\lag}{\langle}
\newcommand{\rag}{\rangle}

\newcommand{\eqdef}{\overset{\mbox{\tiny{def}}}{=}}

\usepackage{cite}

\makeatletter
\@namedef{subjclassname@2020}{%
  \textup{2020} Mathematics Subject Classification}
\makeatother
\begin{document}

\title[Solutions to the thermostated Boltzmann equation with deformation]{On smooth solutions to the thermostated Boltzmann equation with deformation}

\author[R.-J. Duan]{Renjun Duan}
%\thanks{}
\address[RJD]{Department of Mathematics, The Chinese University of Hong Kong,
Shatin, Hong Kong, P.R.~China}
\email{rjduan@math.cuhk.edu.hk}

\author[S.-Q. Liu]{Shuangqian Liu}
\address[SQL]{School of Mathematics and Statistics, Central China Normal University, Wuhan 430079, P.R.~China}
\email{tsqliu@jnu.edu.cn}

\begin{abstract}
This paper concerns a kinetic model of the thermostated Boltzmann equation with a linear deformation force described by a constant matrix. The collision kernel under consideration includes both the Maxwell molecule and general hard potentials with angular cutoff.  We construct the smooth steady solutions  via a perturbation approach when the deformation strength is sufficiently small. The steady solution is a spatially homogeneous non Maxwellian state  and may have the polynomial tail at large velocities. Moreover, we also establish the long time asymptotics toward steady states for the Cauchy problem on the corresponding spatially inhomogeneous equation in torus, which in turn gives the non-negativity of steady solutions.
\end{abstract}

%\date{\today}


%35Q20  	Boltzmann equations
%35B40  	Asymptotic behavior of solutions to PDEs


\maketitle

%\thanks{}
%\maketitle
%\begin{abstract}
%\Red{To be added.}
%
%\end{abstract}

%\setcounter{tocdepth}{1}
%\tableofcontents

%\thispagestyle{empty}

%\newpage
\section{Intoduction}

The homoenergetic solutions to the Boltzmann equation were first introduced by Galkin \cite{G1} and Truesdell \cite{T} independently at almost the same time. These prototypical solutions not only indicate the existence of invariant manifolds of molecular dynamics but also give a new insight into the relation between atomic forces and nonequilibrium behavior of the gas. Recently, James-Nota-Vel\'azquez \cite{JNV-ARMA,JNV-JNS,JNV19} and  Bobylev-Nota-Vel\'azquez \cite{BNV-2019} provided the systematic mathematical study of the subject. Motivated by those works, the authors of this paper \cite{DL-2020} also considered the smoothness and asymptotic stability of self-similar solutions to the Boltzmann equation for the uniform shear flow in case of the Maxwell molecule. In the non Maxwell molecule case, for instance, for the hard potentials, the problem is more subtle to treat and still remains largely open, because the temperature of system increases only in a polynomial rate depending on the collision kernel and the shear rate in the rescaled equation is no longer a constant but a time-dependent function, see the conjecture in \cite{JNV-ARMA} for details.

On the other hand, instead of studying the uniform shear flow as a time-dependent state due to the viscous heating, it is also usual to introduce non-conservative external forces to compensate exactly for the viscous increase of temperature and achieve a steady state. This kind of force is referred to as thermostats and a typical choice of the thermostat force is the friction $-\beta v$ with a constant $\beta\in \R$, see \cite[Chapter 3.4]{GaSa}. Inspired by this, we are concerned in this paper with the spatially homogeneous steady problem on the thermostated Boltzmann equation with a  deformation force:
\begin{align}\label{G-st}
-\beta\na_v\cdot(vG_{st})-\al\na_v\cdot(AvG_{st})=Q(G_{st},G_{st}).
\end{align}
Here, the unknown $G_{st}=G_{st}(v)$ denotes the non-negative velocity distribution function of particles with velocity $v\in \R^3$. The matrix $A=(a_{ij})\in M_{3\times 3}(\R)$ induce a deformation force $-\al A v$ with the strength given by the parameter $\al>0$ and the constant $\beta\in \R$ is a parameter standing for the strength of the thermostated force. The nonlinear term  $Q(\cdot,\cdot)$ is the collision operator defined as
\begin{equation}\label{Q-op}
Q(F_1,F_2)=\int_{\R^3}\int_{\S^2}B(\om,v-v_\ast)[F_1(v_\ast')F_2(v')-F_1(v_\ast)F_2(v)]\,d\omega dv_\ast,
%=&Q_{\textrm{gain}}(F_1,F_2)-Q_{\textrm{loss}}(F_1,F_2).
\end{equation}
where we have denoted $v'=v+[(v_\ast-v)\cdot\omega]\om$ and $v_\ast'=v_\ast-[(v_\ast-v)\cdot\omega]\om$ with $\omega\in \S^2$ in terms of the conservation laws $v_\ast+v=v_\ast'+v'$ and $|v_\ast|^2+|v|^2=|v_\ast'|^2+|v'|^2$. Throughout this paper, we let %assume that
%moreover, $B(\om,v-v_\ast)$ is the collision kernel given by
\begin{eqnarray}
&\dis B(\om,v-v_\ast)=|v-v_\ast|^\ga B_0(\cos\ta),\ \cos\ta=\om\cdot\frac{v-v_\ast}{|v-v_\ast|},\ \om\in\S^2,\label{hard-sp}\\
%\end{align}
%\begin{equation}
&\dis 0\leq\ga\leq1,\ 0\leq B_0(\ta)\leq C|\cos\ta|.\label{hd-po}
\end{eqnarray}
This includes the cases of the Maxwell molecule $\gamma=0$ and general hard potentials $0<\gamma\leq 1$ under the Grad's angular cutoff assumption.

To consider \eqref{G-st}, we supplement it with the restriction condition that
\begin{equation}
\label{con.ge}
\int_{\R^3}[1,v,|v|^2]G_{st}(v)\,dv=[1,0,3].
\end{equation}
The steady problem \eqref{G-st} is solvable only if its left hand term is microscopic, namely,
\begin{equation*}
%\label{ }
\int_{\R^3}[1,v,|v|^2]\{-\bet \na_v\cdot(v G_{st})-\al \na_v\cdot(AvG_{st})\}\,dv=0. %\ i=1,2,3.
\end{equation*}
This together with \eqref{con.ge} implies that
%we get from the conditional energy conservation law \eqref{con.ge} that
\begin{equation}
\label{def.b.in}
\beta=-\al \frac{\int_{\R^3} v\cdot (Av)G_{st}\,dv}{\int_{\R^3} |v|^2G_{st}\,dv}=-\frac{\al}{3}\int_{\R^3} v\cdot (Av)G_{st}\,dv.
\end{equation}
Plugging this back to \eqref{G-st} gives
\begin{align}\label{G-st.f}
\frac{1}{3}\int_{\R^3} v\cdot (Av)G_{st}\,dv\na_v\cdot(vG_{st})-\na_v\cdot(AvG_{st})=\frac{1}{\al}Q(G_{st},G_{st}).
\end{align}
From \eqref{G-st.f}, the deformation strength $\al>0$ plays the same role as the Knudsen number, and we then expect to adopt the perturbation approach as in \cite{DL-2020} to construct smooth solutions for any small $\al>0$.

To present the main results of this paper, we first introduce some notations. To the end, associated with the condition \eqref{con.ge}, we define the reference global Maxwellian $\mu$ by
\begin{equation}
\label{def.mu}
\mu=(2\pi)^{-\frac{3}{2}}e^{-|v|^2/2},
\end{equation}
and use
a velocity weight function
\begin{equation}
\label{def.vwwq}
w_l=w_l(v):=(1+|v|^2)^l
\end{equation}
with an integer $l>0$. Let $\ze=(\ze_1,\ze_2,\ze_3)$ and $\vth=(\vth_1,\vth_2,\vth_3)$ be multi-indices with length $|\ze|$ and $|\vth|$, respectively, and denote $\pa_\ze=\pa_v^\ze=\pa_{v_1}^{\ze_1}\pa_{v_2}^{\ze_2}\pa_{v_3}^{\ze_3}$, $\pa^\vth=\pa_x^\vth=\pa_{x_1}^{\vth_1}\pa_{x_2}^{\vth_2}\pa_{x_3}^{\vth_3}$ and  $\pa_\ze^\vth=\pa_v^\ze\pa_x^\vth$ for simplicity. We define $\ze'\leq\ze$ if each component of $\ze'$ is not greater than the one of $\ze$, and  write $\ze'<\ze$ in case of $\ze'\leq\ze$ and $|\ze'|<|\ze|.$
We also let $C_{\vth}^{\vth'}$ be the usual binomial coefficient for two  multi-indices $\vth$ and $\vth'$ with $\vth'\leq \vth$.
%$\left(\begin{array} {cc}
%&\vth \\
%&\vth'
%\end{array}\right)$.
%and denote the multinomial coefficient $\left(\begin{array} {cc}
%&|\vth| \\
%&\vth
%\end{array}\right)$ by $C_{|\vth|}^{\vth}$.

The first result in the paper is to establish the existence of smooth solutions to the steady problem \eqref{G-st} and \eqref{con.ge} for the steady state $G_{st}$ with $\beta$ given by \eqref{def.b.in} .
% of \eqref{G-st}.

\begin{theorem}\label{st.sol}
Assume \eqref{hard-sp} and \eqref{hd-po} for the collision kernel. Let $A=(a_{ij})\in M_{3\times 3}(\R)$ be a non scalar matrix and
\begin{equation}
\label{def.G1}
G_1=-\sum_{i,j=1}^3a_{ij}L^{-1}\left\{(v_iv_j-\frac{1}{3}\delta_{ij}|v|^2)\mu^{\frac{1}{2}}\right\},
\end{equation}
where $L$ is the linearized collision operator as in \eqref{def.L}. Then, there is an integer $l_0>0$ such that for any integer $l\geq l_0$,
there is $\al_0=\al_0(l)>0$ depending on $l$ such that for any $\al\in (0,\al_0)$, the steady problem \eqref{G-st} and \eqref{def.b.in} under the condition \eqref{con.ge} admits a unique non-negative smooth solution $G_{st}=G_{st}(v)\in C^\infty (\R^3)$ of the form
\begin{equation}
\label{def.spert}
G_{st}=\mu+\al \mu^{\frac{1}{2}} G_1+\al^2 \widetilde{G}_R,
\end{equation}
satisfying that $\int_{\R^3}[1,v,|v|^2]\widetilde{G}_R\,dv=0$ and for any integer $m\geq 0$,
\begin{equation}
\label{st.sole2}
\sum\limits_{|\ze|\leq m}\|w_l\pa_{\ze}\widetilde{G}_R\|_{L^\infty}\leq \tilde{C}_{m,l},
\end{equation}
where $\tilde{C}_{m,l}>0$ is a constant depending only on $m$ and $l$ but not on $\al$ and $w_l$ is given in \eqref{def.vwwq}.
\end{theorem}

%Below are a few remarks in order about Theorem \ref{st.sol} above.
%\begin{remark}
%It is straightforward to see that $G_1=-L^{-1}\{v\cdot Av \mu^{\frac{1}{2}}\}$ also enjoys the estimate \eqref{st.sole2}, so does the stationary solution $G_{st}$, therefore the self-similar profile constructed here is smooth enough.
%\end{remark}

Similar to \cite{DL-2020}, we point out that the obtained steady solution is a non Maxwellian state and may have the polynomial tail at large velocities, which is the main feature of the problem.  In order to justify the non-negativity of the steady solution $G_{st}$ constructed in Theorem \ref{st.sol}, we introduce a spatially inhomogeneous model in torus $\T^3=[0,2\pi]^3$:
\begin{eqnarray}
&\dis \pa_t G+v\cdot\na_xG-\beta\na_v\cdot(vG)-\al\na_v\cdot(AvG)=Q(G,G),\ t>0,\ x\in\T^3,\ v\in\R^3,\label{G-ust}\\
&\dis G(0,x,v)= G_0(x,v),\ x\in\T^3,\ v\in\R^3.
\label{G-id}
\end{eqnarray}
It turns out that the steady solution $G_{st}$ can be used to describe the large time asymptotics of the unsteady problem \eqref{G-ust} and \eqref{G-id}. We state this result as follows.
%For this, the second result in this paper is concerned with the dynamic stability of the stationary solution $G_{st}$, which, in particular, guarantees the nonnegativity of $G_{st}$.

\begin{theorem}\label{ge.th}
Let $G_{st}(v)$ be the steady profile obtained in Theorem \ref{st.sol}  and the constant $\beta$ be defined in \eqref{def.b.in}.
Assume further that $\beta I+\al A$ is invertible. Then, there are constants $\la>0$ and $C>0$ independent of $\al$ such that if it holds that $G_0(x,v)\geq0$,
\begin{align}\label{id.cons}
\int_{\T^3}\int_{\R^3}[G_0(x,v)-G_{st}]\,dvdx=0,\quad \int_{\T^3}\int_{\R^3}vG_0(x,v)\,dvdx=0,
\end{align}
and
\begin{equation}\label{th.ids.p}
\sum\limits_{|\ze|+|\vth|\leq m}\left\|w_l\pa_\ze^{\vth}\left[G_0(x,v)-G_{st}(v)\right]\right\|_{L^\infty}\leq\al^2,%\notag
\end{equation}
for an integer {$m\geq 1$}, then the Cauchy problem \eqref{G-ust} and \eqref{G-id} admits a unique  global solution $G(t,x,v)\geq0$ satisfying that
\begin{equation}\label{decay}
\sum\limits_{|\ze|+|\vth|\leq m}\left\|w_l\pa_\ze^{\vth}\left[G(t,x,v)-G_{st}(v)\right]\right\|_{L^\infty}\\
\leq C \al e^{-\la \beta_1\al^2 t},
%\frac{C}{\al} e^{-\la \beta_1\al^2 t}\sum\limits_{|\ze|+|\vth|\leq m}\left\|w_l\pa_\ze^{\vth}\left[G_0(x,v)-G_{st}(v)\right]\right\|_{L^\infty}
\end{equation}
for any $t\geq 0$,
where $\beta_1>0$ is a positive constant given by
\begin{equation}
\label{def.b1}
\beta_1=-\frac{1}{3}\int_{\R^3} v\cdot (Av)\left\{\mu^{\frac{1}{2}} G_1+\al \widetilde{G}_R\right\}\,dv.
\end{equation}
\end{theorem}

%
%\begin{remark}
%Since $\al>0$ and can be sufficiently small, it is obvious to see that $\beta_1\geq0$, and
%if $A$ is not a a scalar matrix, then actually $\beta_1>0$, which means $G(t,x,v)$ tends to $G_{st}$ exponentially as $t\rightarrow+\infty,$ this coincides with the previous work by Duan and the author \cite{DL-2020}, where the shear flow problem was considered.
%\end{remark}
%\begin{remark}
%It is conjectured in \cite[Conjecture 4.1, pp.1954]{JNV-JNS} that if $\ga>0$, the weak solution of the spatially homogeneous Boltzmann equation \eqref{dbe} behaves like a Maxwellian distribution but not like a self-similar solution as $t\rightarrow+\infty$. Theorem \ref{ge.th} demonstrates that the long time asymptotics of the spatially inhomogeneous Cauchy prolem \eqref{G-ust} and \eqref{G-id} is still given by the self-similar solution even for general hard potentials i.e. the homogeneity parameter $\ga\in[0,1]$.
%\end{remark}

%\subsection{Existing works}
In what follows we mention some existing works that are most related to the background and motivations of the current topic; readers may refer to \cite{DL-2020} for a more detailed review. Based on the Fourier transform method in \cite{Bo75, Bo88}, Bobylev-Cercignani \cite{BC02a,BC02b,BC03} discussed the self-similar asymptotics for the spatially homogeneous Boltzmann equation. As in the original work by Galkin \cite{G1} and Truesdell \cite{T},
by solving the ODE system consisting of velocity moments, particularly the second order moments, Cercignani \cite{Cer00} investigated the shear flow problem on a granular flow between parallel plates which is modeled by the Boltzmann equation, and
Bobylev-Cercignani \cite{BC-Mo} later obtained the well-posedness and large time behavior of the granular system described by Boltzmann-like
equations. We also mention that Cercignani \cite{Cer-1989}
proved the global existence of homoenergetic affine flows for the Boltzmann equation in the case of
simple shear for a large class of interaction potentials which include hard potentials, and these solutions in general may not be self-similar.
It seems that \cite{Cer-1989} is the first mathematical result on the homoenergetic solution of the Boltzmann equation for the non Maxwell molecule.

Recently, in a significant progress by James-Nota-Vel\'azquez \cite{JNV-ARMA}, the existence of homoenergetic mild solutions as non-negative Radon measures was studied in a systematic way for a large class of initial data, and the problem on the asymptotics of homoenergetic solutions in the case of non Maxwell molecules was also proposed. In the meantime, it is discussed in \cite{JNV-JNS,JNV19} that there is a balance between the hyperbolic term and collision term for the Boltzmann equation describing homogenergetic flow and the corresponding long time asymptotic behavior
depends on which term is dominated in large time. By combining the Fourier transform method and moments argument, a more recent progress has been achieved by Bobylev-Nota-Vel\'azquez \cite{BNV-2019}, where the authors proved the self-similar asymptotics of solutions in large time for the Boltzmann equation with a general deformation force
%\begin{equation*}
%%\label{ }
%\pa_t F-\na_v\cdot (AvF)=Q(F,F)
%\end{equation*}
under a smallness condition on the matrix $A$, and they also showed that the self-similar profile can have the finite polynomial moments of higher order as long as the norm of $A$ is getting smaller.
To the best of our knowledge, \cite{BNV-2019} seems the only known result on the large time asymptotics to the self-similar profile in weak topology, see also \cite{Bo21} for a further study to provide explicit estimates of the smallness of the matrix $A$. Following \cite{JNV-ARMA} and \cite{BNV-2019}, in the case of Maxwell molecule, the authors of this paper \cite{DL-2020} constructed smooth self-similar profiles for the shear flow problem on the Boltzmann equation and proved the dynamical stability of the stationary solution via a perturbation approach.

As mentioned at the beginning, different from the uniform shear flow where the temperature increases in time and the self-similar asymptotic has to be involved, we expect the extra thermostated term to compensate the viscous heating energy and drive the system to converge to the steady state. We remark that a similar situation may occur to the bounded domain case with diffuse boundaries that also can absorb the shearing energy such that the system tends asymptotically to the steady motion instead of the self-similar solution. In particular, a boundary value problem on the Boltzmann equation for the plane Couette flow was studied in \cite{DLY-2021}, where they established the existence of spatially inhomogeneous non-equilibrium stationary solutions to the steady problem for small shear rate and proved dynamical stability of the stationary solution.

%It should be pointed out that most of the existing works are focused on
%the Maxwell molecule case, for which the interaction potential takes the form $U(r)=r^{-4}$, while the general interaction potential is in the form of $r^{-s+1}$, and the homogeneity is measured by $\ga=\frac{s-5}{s-1}$. In the case of Maxwell molecule collision, the significant advantage is the following identity
%\begin{equation}
%%\label{g1usf}
%L(v_iv_j\mu^{1/2})=2b_0(v_iv_j-\frac{\de_{ij}}{3}|v|^2)\mu^{1/2},\ i,j=1,2,3\notag
%\end{equation}
%with the positive constant
%%\begin{equation}
%%\label{ }
%$b_0:=3\pi \int_{-1}^1 B_0(z)z^2(1-z^2)\,dz$, which cannot be true for $\ga\neq0$. As mentioned at beginning of this paragraph, with this identity, the ODE system consisting of the second order moments for the Boltzmann equation describing homogenergetic flow can be solved, cf. \cite{JNV-ARMA,BNV-2019}.
%In this paper, motivated by \cite{DL-2020,DLY-2021}, we aim at the existence and dynamical stability of the smooth self-similar profiles of the Boltzmann equation with hard potentials. .........

%\subsection{Strategies and ideas}
Compared to our previous work \cite{DL-2020} about the self-similar steady problem in case of the simple shear force and Maxwell molecules, we treat in this paper the more general deformation force described by the matrix $A$ and also include the case of hard potentials $0<\gamma\leq 1$ for the molecular interaction. In what follows we outline the key strategies in the proof of main results and point out the main differences with \cite{DL-2020}.   %Basically, the main idea employed in the paper is Caflisch's decomposition and Guo's $L^\infty\cap L^2$ theory under the perturbation regime.
First of all, for the steady problem \eqref{G-st} or \eqref{G-st.f}, we look for solutions by setting the perturbation
$
G_{st}=\mu+\al\sqrt{\mu}G_1+\al^2 \widetilde{G}_R$ with $\widetilde{G}_R=\sqrt{\mu}G_R$ as in \eqref{def.spert}. Here, $G_1$ as in \eqref{def.G1} is introduced  to remove the zero-order inhomogeneous term in terms of \eqref{G1} and $G_R$ is the remainder satisfying  \eqref{GR}. Note that $G_1$ involves the general deformation matrix $A$ and it is non-zero for any non scalar matrix $A$. The usual energy approach fails to be used to treat  \eqref{GR} due to the second order velocity growth of the term $\frac{\al}{2}v\cdot(Av) G_R$ since the linearized collision operator only provides the dissipation term $\int \nu(v) |G_R|^2dv$ with $\nu(v)\sim |v|^\gamma$ $(0\leq \gamma\leq 1)$ for large velocities. As in \cite{DL-2020}, we employ the Caflisch's decomposition (cf.\cite{Ca-1980})
\begin{align}
\widetilde{G}_R=\sqrt{\mu}G_R=G_{R,1}+\sqrt{\mu}G_{R,2},\notag
\end{align}
where $G_{R,1}$ and $G_{R,2}$ satisfy the coupled system
\begin{align}\label{GR1-in}
-\al^2&\beta_1\na_v\cdot(vG_{R,1})-\al\na_v\cdot(AvG_{R,1})
+\nu G_{R,1}\notag\\=&\chi_M\CK G_{R,1}-\frac{1}{2}\al^2\beta_1|v|^2\sqrt{\mu}G_{R,2}-\frac{\al}{2}v\cdot(Av)\sqrt{\mu}G_{R,2}
+\cdots,
\end{align}
and
\begin{align}\label{GR2-in}
-\al^2&\beta_1\na_v\cdot(vG_{R,2})-\al \na_v\cdot(Av G_{R,2})
+LG_{R,2}=\mu^{-\frac{1}{2}}(1-\chi_M)\CK G_{R,1},
\end{align}
respectively. The benefit of this splitting is that the term $\frac{\al}{2}v\cdot(Av) \sqrt{\mu}G_{R,2}$ is no longer a trouble since it contains
$\sqrt{\mu}$ which can absorb any order polynomial velocity growth. The price to pay is that one cannot make a direct energy estimate on $G_{R,1}$ because $\chi_M\CK G_{R,1}$ may not be small in the $L^2$ setting. However, this can be resolved in terms of the $L^2$-$L^\infty$ interplay since the smallness for $\chi_M\CK G_{R,1}$ can be recovered via the velocity weighted $L^\infty$ norm. Indeed, in the case of Maxwell molecule, the following decay mechanism of $\CK$ has been found in \cite{DL-2020}:
\begin{align}%\label{CK1}
\sup_{|v|\geq M(l)} w_{l}|\pa_{\ze}(\CK f)|\leq \frac{C}{l} \sum\limits_{0\leq \ze'\leq \ze}\|w_{l}\pa_{\ze'}f\|_{L^\infty},\notag
\end{align}
where $C$ is independent of $l$ and $M(l)\to \infty$ as $l\to \infty$. Thus, the smallness in $L^\infty$ holds whenever $l$ is suitably large.
Note that the above estimate seems hard to be true for the non Maxwell molecule case. To treat this difficulty, motivated by \cite{AEP-87}, in case of hard potentials $0<\gamma\leq 1$,  we instead make use of
the following  estimate
\begin{align}
\sup_{|v|\geq M(l)}(1+|v|)^{-\ga} w_{l}|\CK f|\leq C\{(1+M(l))^{-\ga/2}+\varsigma(l)\}\|w_{l}f\|_{L^\infty},\notag
\end{align}
for $C$ independent of $l$, where it holds that $M(l)\to \infty$ and $\varsigma(l)\to 0$ as $l\to \infty$. Then, the smallness in $L^\infty$ still holds when $l$ is chosen to be large enough. Therefore, in both cases $\gamma=0$ and $0<\gamma\leq 1$, the $L^\infty$ estimates combined with the $L^2$ estimates can be closed.

In addition, the coupled equations \eqref{GR1-in} and \eqref{GR2-in} will be solved by an iteration method in which the conservation laws
$\lag G_{R,1}+\mu^{\frac{1}{2}}G_{R,2},[1,v_i,|v|^2]\rag
%+\lag ,[1,v_i,|v|^2]\rag
=0$ $(i=1,2,3)$
play a crucial role.
To ensure that the macroscopic moments of the iteration system are conserved, we design the following delicate approximation equations
\begin{eqnarray}%\label{gr12-ls}
\left\{\begin{array}{rll}
\begin{split}
\eps G^{n+1}_{R,1}&-\beta^{n}\na_v\cdot (vG^{n+1}_{R,1})-\al \na_v\cdot(AvG^{n+1}_{R,1})
+\nu G^{n+1}_{R,1}
-\chi_{M}\CK G^{n+1}_{R,1}\\&+\frac{\beta^n}{2}|v|^2\mu^{\frac{1}{2}}G^{n+1}_{R,2}+\frac{1}{2}\al v\cdot(Av)\mu^{\frac{1}{2}}G^{n+1}_{R,2}-\left(\beta_1^{n+1}-\frac{1}{3}\lag G_1,LG_1\rag \right)\na_v\cdot(v\mu)
\\=&\frac{1}{3}\lag G_1,LG_1\rag \na_v\cdot(v\mu)+\frac{\beta^n}{\al}\na_v\cdot(v\mu^{\frac{1}{2}}G_1)+\na_v\cdot(Av\sqrt{\mu}G_1)+
Q(\mu^{\frac{1}{2}}G_1,\mu^{\frac{1}{2}}G_1)\\&+\al\{Q(\mu^{\frac{1}{2}}G^{n}_R,\mu^{\frac{1}{2}}G_1)
+Q(\mu^{\frac{1}{2}}G_1,\mu^{\frac{1}{2}}G^{n}_R)\}+\al^2 Q(\mu^{\frac{1}{2}}G^n_R,\mu^{\frac{1}{2}}G^n_R),\\
\eps G^{n+1}_{R,2}&-\beta^n\na_v\cdot (v G^{n+1}_{R,2})-\al \na_v\cdot(AvG^{n+1}_{R,2})
+LG^{n+1}_{R,2}-(1-\chi_{M})\mu^{-\frac{1}{2}}\CK G^{n+1}_{R,1}=0,
\end{split}
\end{array}\right.\label{it-GR-in}
\end{eqnarray}
where two penalty terms with the  parameter $\eps>0$ have been added  and
\begin{align}%\label{bt1n}
 \beta^n=-\frac{\al}{3}{\rm tr}A+\al^2\beta_1^n,\notag
 \end{align}
 with
 \begin{align}%\label{bt1n}
\beta_1^n=\frac{1}{3}\int_{\R^3} G_1LG_1\,dv
-\frac{\al}{3}\int_{\R^3} \FP_1\{ v\cdot (Av)\sqrt{\mu}\}G_R^n\,dv,\
 \mu^{\frac{1}{2}}G^{n}_R=G^{n}_{R,1}+\mu^{\frac{1}{2}}G^{n}_{R,2}.\notag
\end{align}
System \eqref{it-GR-in} provides us the following cancellations
\begin{align}
&\left\lag\left(\beta_1^{n+1}-\frac{1}{3}\lag G_1,LG_1\rag \right)\na_v\cdot(v\mu),\frac{1}{2}|v|^2\right\rag
-\frac{1}{2}\al \left\lag v\cdot(Av)\mu^{\frac{1}{2}}G^{n+1}_{R,2},\frac{1}{2}|v|^2\right\rag
\notag\\&\qquad-\al \left\lag\na_v\cdot(AvG^{n+1}_{R,1}),\frac{1}{2}|v|^2\right\rag
-\al \left\lag\na_v\cdot(AvG^{n+1}_{R,2}),\frac{1}{2}|v|^2\sqrt{\mu}\right\rag=0,\notag
\end{align}
and
\begin{align}
&\frac{1}{3}\left\lag\lag G_1,LG_1\rag \na_v\cdot(v\mu),\frac{1}{2}|v|^2\right\rag
+ \left\lag \na_v\cdot(Av\sqrt{\mu}G_1),\frac{1}{2}|v|^2\right\rag=0,\notag
\end{align}
which indeed give the energy conservation $\lag G^{n+1}_{R,1}+\mu^{\frac{1}{2}} G^{n+1}_{R,2},|v|^2\rag=0$. Moreover, as in \cite{DHWZ-19} for treating the nonlocal collision term, we introduce a $\si$-parametrized procedure to ensure the construction of solutions to the linear inhomogeneous system with $0\leq \si\leq 1$; see Lemma \ref{lifpri} and Lemma \ref{ex.pals} for details. However, this induces the loss of conservation laws for the system with $0\leq \si<1$ in the hard potential case $0<\gamma\leq 1$, which is quite different from the situation treated in \cite{DL-2020}.

 %In addition,
%to obtain the extra weight $(1+|v|)^{-\ga}$ in left hand of the above estimate, we used the following observation
%\begin{align}
%\int_{-\infty}^{t}e^{-\int_{s}^t\nu (V(\tau))d\tau}\nu (V(s))ds\leq1.\notag
%\end{align}

The second point is concerned with the non-negativity of the steady profiles. For the purpose, we introduce a spatially inhomogeneous model \eqref{G-ust} and prove the asymptotic stability of the stationary solution under small perturbation. We remark that although it is a spatially inhomogeneous problem, the proof with slight modifications can still be carried over to treat the spatially homogeneous case. {One difficulty part is to obtain the macroscopic dissipation in a more delicate way than that in the steady case. In particular, we re-design the Caflisch's decomposition $\sqrt{\mu} f=f_1+\sqrt{\mu}f_2$ with $f_1$ and $f_2$ satisfying \eqref{f1-eq}  and \eqref{f2-eq}, respectively. The key point is to add a microscopic fourth-order moment function $\al Av\cdot(\na_v\sqrt{\mu})(|v|^2-3)c_{f_2}$ to the left-hand of \eqref{f2-eq}  in order to cancel such trouble term coming from $-\al \nabla_v\cdot (Avf_2)$. Correspondingly,  \eqref{f1-eq} has been modified with $\al Av\cdot(\na_v\sqrt{\mu})(|v|^2-3)\sqrt{\mu}c_{f_2}$ added to the right-hand side. Under such decomposition, the macroscopic energy $\|c\|^2$ and the microscopic energy $\|\mathbf{P}_1g_2\|^2$ can be combined for estimates, see the result \eqref{g2-l2-plus}. Thus, the corresponding energy dissipation rates $\al^2\|c\|^2$ and $\|\mathbf{P}_1g_2\|_\nu^2$ are obtained. This $L^2$ estimate \eqref{g2-l2-plus} is crucial for obtaining the macroscopic dissipation and further deducing the exponential decay rate with the size proportional to $\al^2$.} This size is the same as that in \cite{DL-2020} for simple shear flow where ${\rm tr} (A)=0$ and the lowest order of $\beta$ is $\al^2$. In the current case for a general deformation matrix $A$, the lowest order of $\beta$ is $\al$ if ${\rm tr} (A)\neq 0$. We remark that it is unclear for us whether the degenerate order $\al^2$ for the size of decay rate is optimal. Moreover, similar to \cite{BMV-16} for the study at the fluid level,  it would be interesting to further consider  possible enhanced decay rates with respect to any small $\al$ by using the deformation effect in case of the hard potentials $0<\gamma\leq 1$ and we will explore this issue in the future.
%which plays a crucial role in $L^2$ estimates. The above cancellation indicates a balance between the spatially inhomogeneous transport and the deformation motion in the sense that \eqref{c-dap} cannot be captured in the case of the spatially homogeneous model \eqref{dbe}.

The third point is related to an application of the Guo's $L^\infty-L^2$ method (cf.~\cite{Guo-2010}). The key idea of this approach in the $L^\infty$ estimate is to convert an integration with respect to $v$ variable along characteristics into an integration with respect to $x$ variable. In the process, one need obtain a proper control for the Jacobian
\begin{align}
\left|\frac{\pa X(s')}{\pa v_\ast}\right|
=&\left|(\beta I+\al A)^{-1}\left[e^{-(s'-s)(\beta I+\al A)}-I\right]\right|\notag
\end{align}
along the following characteristic line
\begin{align*}%\label{H1CLp1}
\left\{\begin{array}{rll}
&V(s')=V(s';s,X(s),v_\ast)= e^{-(s'-s)(\beta I+\al A)}v_\ast,\\[2mm]
&X(s')=X(s';s,X(s),v_\ast)=X(s)-(\beta I+\al A)^{-1}\left[e^{-(s'-s)(\beta I+\al A)}-I\right]v_\ast.
\end{array}\right.
\end{align*}
%doesn't vanish.
For this, as described in Lemma \ref{expm-lem}, we prove a lower bound of the determinant of a matrix exponential, and moreover, we also give an upper bound of the region of the integration after the change of variable $X(s')\rightarrow y$.

%\subsection{Organization of the paper}

The rest  of this paper is arranged as follows. The existence of the steady profile $G_{st}(v)$ for \eqref{G-st} is established in Section \ref{st-pro}. Section \ref{ust-pro}
is devoted to the unsteady problem  \eqref{G-ust} and \eqref{G-id}. In Section \ref{pre-sec} as an appendix, we give the basic estimates on the linearized operator $L$ as well as the nonlinear operators $\Ga$ and $Q$, further present a key estimate for the operator $\CK$ in the case of hard potentials, and finally derive a lower bound for a matrix exponential.

\medskip
%\subsection{
\noindent{\it Notations.}
We give more notations to be used throughout the paper.
% \begin{itemize}
% \item
Let $C$ denote some generic positive (generally large) constant and $\la$ denote some generic positive (generally small) constants, where $C$ and $\la$  may take different values in different places. Let ${\bf 1}_\CS$ be the characteristic function on the set $\CS$.
%\item
For simplicity, we use $\|\cdot\| $ to denote the norms of either $L^{2}(\T_x^3
\times \R_v^{3})$ or $L^{2}(\T_x^3)$ or $L^{2}(\R_v^3)$. We also use $\|\cdot \|_{L^\infty }$ to denote the norms of either $L^{\infty }(\T_x^3
\times \R_v^{3})$ or $L^{\infty }(\R_v^3)$. Moreover, %we denote $\|\cdot \|_{\nu}\equiv \|\nu^{1/2}\cdot\|_2$, and
$(\cdot,\cdot)$ denotes the inner product of
$L^2(\T_x^3\times {\R}_v^{3})$  and $\langle\cdot\rangle$ denotes the inner product of $L^2(\R^3_v)$.

%\end{itemize}

\section{Steady problem}\label{st-pro}
This section is devoted to the existence of the non-equilibrium smooth steady solution of \eqref{G-st}. We begin with some usual notations in the framework of perturbations around the global Maxwellian $\mu$ in \eqref{def.mu}. First of all, we introduce the linearized collision operator $L$ and the nonlinear collision operator $\Ga$, defined by
\begin{align}\label{def.L}
Lg=-\mu^{-1/2}\left\{Q(\mu,\sqrt{\mu}g)+Q(\sqrt{\mu}g,\mu)\right\},
%\notag
\end{align}
and
\begin{align}\label{def.Ga}
\Gamma (f,g)=\mu^{-1/2} Q(\sqrt{\mu}f,\sqrt{\mu}g)=\int_{\R^3}\int_{\S^2}B(\om,v-v_\ast)\mu^{1/2}(v_\ast)[f(v_\ast')g(v')-f(v_\ast)g(v)]\,d\om dv_\ast,%\notag
\end{align}
respectively. Note that
$$
Lf=\nu f-Kf
$$
with
\begin{align}\label{sp.L}
\left\{\begin{array}{rll}
&\nu=\dis{\int_{\R^3}\int_{\S^2}}B(\om,v-v_\ast)\mu(v_\ast)\, d\om dv_\ast\backsim (1+|v|)^{\ga},\\[2mm]
&Kf=\mu^{-\frac{1}{2}}\left\{Q(\mu^{\frac{1}{2}}f,\mu)+Q_{\textrm{gain}}(\mu,\mu^{\frac{1}{2}}f)\right\},
%=\Gamma (\mu^{\frac{1}{2}},f)+\Gamma (f,\mu^{\frac{1}{2}}),
\end{array}\right.
\end{align}
where $Q_{\textrm{gain}}$ denotes the positive part of $Q$ in \eqref{Q-op}. Moreover, it holds that
\begin{align}
Kf=\int_{\R^3}\Fk(v,v_\ast)f(v_\ast)\,dv_\ast=\int_{\R^3}(\Fk_2-\Fk_1)(v,v_\ast)f(v_\ast)\,dv_\ast,\notag
\end{align}
with
\begin{equation}
0\leq \Fk_1(v,v_\ast)\leq \tilde{c}_1|v-v_\ast|^\ga e^{-\frac{1}{4}(|v|^2+|v_\ast|^2)},\ 0\leq \Fk_2(v,v_\ast)\leq \tilde{c}_2|v-v_\ast|^{-2+\ga}e^{-
\frac{1}{8}|v-v_\ast|^{2}-\frac{1}{8}\frac{\left||v|^{2}-|v_\ast|^{2}\right|^{2}}{|v-v_\ast|^{2}}},\label{fk-12}
\end{equation}
where both $\tilde{c}_1$ and $\tilde{c}_2$ are positive constants. Note that in the case of $\ga=1$, \eqref{fk-12} has been derived in \cite[pp.45-46]{Gla-book}, and the remaining cases that $\ga\in[0,1)$ can be treated similarly. The upper bound in \eqref{fk-12}  may not be optimal.

The kernel of $L$, denoted as $\ker L$, is a five-dimensional space spanned by
$$
\{1,v,|v|^2-3\}\sqrt{\mu}:= \{\phi_i\}_{i=1}^5.
$$ We further define a projection from $L^2$ to $\ker(L)$ by
\begin{align*}%\label{P0}
\FP_0 g=\left\{a_g+\Fb_g\cdot v+(|v|^2-3)c_g\right\}\sqrt{\mu}
\end{align*}
for $g\in L^2$, and correspondingly denote the operator $\FP_1$ by $\FP_1g=g-\FP_0 g$, which is orthogonal to $\FP_0$ in $L^2$.
Traditionally, $\FP_0 g$ is also called the macroscopic part, while $\FP_1 g$ stands for the microscopic component.

It is also convenient to define
\begin{align}\notag
\CL f=-\left\{Q(f,\mu)+Q(\mu,f)\right\}=\nu f-\CK f,
\end{align}
with
\begin{equation}\label{sp.cL}
\CK f=Q(f,\mu)+Q_{\textrm{gain}}(\mu,f)=\sqrt{\mu}K(\frac{f}{\sqrt{\mu}}).
\end{equation}
%according to \eqref{sp.L}.

\subsection{Hilbert expansion and Caflisch's decomposition}
As derived before, we will study the steady problem
\begin{equation}
\label{sp}
-\beta \na_v\cdot (vG_{st})-\al \na_v\cdot(Av G_{st})=Q(G_{st},G_{st})
\end{equation}
with
\begin{equation}
\label{def-beta}
\beta=-\frac{\al}{3}\int_{\R^3} v\cdot (Av)G_{st}\,dv.
\end{equation}
Our goal is to look for a unique smooth solution $G_{st}(v)$ satisfying
\begin{equation}
\label{con-G}
\int_{\R^3} G_{st} \,dv=1,\quad
\int_{\R^3}v_i G_{st} \,dv=0,i=1,2,3,\quad
\int_{\R^3} |v|^2 G_{st} \,dv=3.
\end{equation}
Note that through the paper we have omitted the dependence of $G_{st}$ on the parameter $\al$.
It can be expected that $G_{st}\to \mu$ if $\al\to 0$. As such, we set
\begin{equation}\label{G-exp}
G_{st}=\mu+\al\sqrt{\mu}\left\{G_1+\al G_R\right\}
\end{equation}
with $\FP_0G_1=\FP_0G_R=0 $ such that \eqref{con-G} is valid, and hence we impose that
\begin{eqnarray}\label{con.g1r}
\left\{\begin{array}{rll}
\begin{split}
&\int_{\R^3}G_1\sqrt{\mu}\,dv=\int_{\R^3}G_R\sqrt{\mu}\,dv=0,\\
&\int_{\R^3}G_1v_i\sqrt{\mu}\,dv=\int_{\R^3}G_Rv_i\sqrt{\mu}\,dv=0,\ i=1,2,3,\\
&\int_{\R^3}G_1|v|^2\sqrt{\mu}\,dv=\int_{\R^3}G_R|v|^2\sqrt{\mu}\,dv=0,
\end{split}
\end{array}\right.
\end{eqnarray}
where $G_1$ is the first order correction and $G_R$ denotes the higher order remainder. Plugging \eqref{G-exp} into \eqref{def-beta}, we get
\begin{align}\label{bet-exp}
\beta=-\frac{\al}{3}\int_{\R^3} v\cdot (Av)[\mu+\al\sqrt{\mu}\left\{G_1+\al G_R\right\}]\,dv=
\al\beta_0+\al^2\beta_1
\end{align}
with
\begin{align}\label{beta01}
\beta_0=-\frac{1}{3}{\rm tr}A,\ \beta_1=-\frac{1}{3}\int_{\R^3} v\cdot (Av)[\sqrt{\mu}\left\{G_1+\al G_R\right\}]\,dv.
\end{align}
Furthermore, substituting \eqref{G-exp} and \eqref{bet-exp} into \eqref{sp} and comparing the coefficients in front of the different orders of $\al$, one has
\begin{align}\label{G1}
-\beta_0\mu^{-\frac{1}{2}}\na_v\cdot(v\mu)-\mu^{-\frac{1}{2}}\na_v\cdot(Av\mu)+ LG_1=0,
\end{align}
and
\begin{align}\label{GR}
-\beta&\mu^{-\frac{1}{2}}\na_v\cdot(v\sqrt{\mu}G_R)-\al\mu^{-\frac{1}{2}}\na_v\cdot(Av\sqrt{\mu}G_R)
+LG_R\notag\\=&\beta_1\mu^{-\frac{1}{2}}\na_v\cdot(v\mu)+\frac{\beta}{\al}\mu^{-\frac{1}{2}}\na_v\cdot(v\sqrt{\mu}G_1)
+\mu^{-\frac{1}{2}}\na_v\cdot(Av\sqrt{\mu}G_1)
\notag\\&+\al\{\Ga(G_1,G_R)+\Ga(G_R,G_1)\}+\al^2\Ga(G_R,G_R).
\end{align}
In light of expression for $\beta_0$ in \eqref{beta01}, one gets from \eqref{G1} that
\begin{align}\label{G1-exp}
G_1=-\sum_{i,j=1}^3a_{ij}L^{-1}\left\{(v_iv_j-\frac{1}{3}\delta_{ij}|v|^2)\mu^{\frac{1}{2}}\right\},
\end{align}
which in turn gives
\begin{align}%\label{beta01}
\beta_1=&\frac{1}{3}\int_{\R^3}\FP_1\{ v\cdot (Av)\sqrt{\mu}\}L^{-1}\{\FP_1(v\cdot Av \mu^{\frac{1}{2}})\}\,dv
-\frac{\al}{3}\int_{\R^3} \FP_1\{ v\cdot (Av)\sqrt{\mu}\}G_R\,dv
\notag\\
=&\frac{1}{3}\int_{\R^3} G_1LG_1\,dv-\frac{\al}{3}\int_{\R^3} \FP_1\{ v\cdot (Av)\sqrt{\mu}\}G_R\,dv.\notag
\end{align}
Note that one has $\beta_1>0$ provided that $A$ is not a scalar matrix and $\al$ is suitably small.

The remainder $G_R$ is determined by \eqref{GR}. There is a severe growth term $\frac{\al}{2}v\cdot(Av)G_R$ caused by the deformation force. To overcome this difficulty, as \cite{DL-2020,DLY-2021}, we resort to the following Caflisch's decomposition
\begin{align}
\sqrt{\mu}G_R=G_{R,1}+\sqrt{\mu}G_{R,2},\notag
\end{align}
where $G_{R,1}$ and $G_{R,2}$ satisfy
\begin{align}\label{GR1}
-\beta&\na_v\cdot(vG_{R,1})-\al\na_v\cdot(AvG_{R,1})
+\nu G_{R,1}\notag\\=&\chi_M\CK G_{R,1}-\frac{\beta}{2}|v|^2\sqrt{\mu}G_{R,2}-\frac{\al}{2}v\cdot(Av)\sqrt{\mu}G_{R,2}
+\beta_1\na_v\cdot(v\mu)+\frac{\beta}{\al}\na_v\cdot(v\sqrt{\mu}G_1)\notag\\&+\na_v\cdot(Av\sqrt{\mu}G_1)
+Q(\sqrt{\mu}G_1,\sqrt{\mu}G_1)+\al\{Q(\sqrt{\mu}G_1,\sqrt{\mu}G_R)+Q(\sqrt{\mu}G_R,\sqrt{\mu}G_1)\}
\notag\\&+\al^2Q(\sqrt{\mu}G_R,\sqrt{\mu}G_R),
\end{align}
and
\begin{align}\label{GR2}
-&\beta\na_v\cdot(vG_{R,2})-\al \na_v\cdot(Av G_{R,2})
+LG_{R,2}=\mu^{-\frac{1}{2}}(1-\chi_M)\CK G_{R,1},
\end{align}
respectively. Here, $\chi_{M}(v)$ is a non-negative smooth cutoff function defined by
\begin{align}%\label{def.chiM}
\chi_{M}(v)=\left\{\begin{array}{rll}
1,&\ |v|\geq M+1,\\[2mm]
0,&\ |v|\leq M,
\end{array}\right.\notag
\end{align}
with $M>0$ sufficiently large.

We will prove the unique existence of \eqref{GR1} and \eqref{GR2} in the Banach  space
\begin{align}
\FX_m=\big\{\CG=&[\CG_1,\CG_2]\big|\sum\limits_{k\leq m}\{\|w_l\na_v^k\CG_1\|_{L^\infty}+\|w_l\na_v^k\CG_2\|_{L^\infty}\}<+\infty,\
k,m\in\Z^+,\notag\\
 &\lag \CG_1,[1,v_i,|v|^2]\rag+\lag \CG_2,[1,v_i,|v|^2]\mu^{\frac{1}{2}}\rag=0,\ i=1,2,3\big\}\notag
\end{align}
associated with the norm
$$
\|[\CG_1,\CG_2]\|_{\FX,m}=\sum\limits_{k\leq m}\{\|w_l\na_v^k\CG_1\|_{L^\infty}+\|w_l\na_v^k\CG_2\|_{L^\infty}\}.
$$
To do so,
we design the following iteration equations
\begin{eqnarray}\label{gr12-ls}
\left\{\begin{array}{rll}
\begin{split}
\eps G^{n+1}_{R,1}&-\beta^{n}\na_v\cdot (vG^{n+1}_{R,1})-\al \na_v\cdot(AvG^{n+1}_{R,1})
+\nu G^{n+1}_{R,1}
-\chi_{M}\CK G^{n+1}_{R,1}\\&+\frac{\beta^n}{2}|v|^2\mu^{\frac{1}{2}}G^{n+1}_{R,2}+\frac{1}{2}\al v\cdot(Av)\mu^{\frac{1}{2}}G^{n+1}_{R,2}-\left(\beta_1^{n+1}-\frac{1}{3}\lag G_1,LG_1\rag \right)\na_v\cdot(v\mu)
\\=&\frac{1}{3}\lag G_1,LG_1\rag \na_v\cdot(v\mu)+\frac{\beta^n}{\al}\na_v\cdot(v\mu^{\frac{1}{2}}G_1)+\na_v\cdot(Av\sqrt{\mu}G_1)+
Q(\mu^{\frac{1}{2}}G_1,\mu^{\frac{1}{2}}G_1)\\&+\al\{Q(\mu^{\frac{1}{2}}G^{n}_R,\mu^{\frac{1}{2}}G_1)
+Q(\mu^{\frac{1}{2}}G_1,\mu^{\frac{1}{2}}G^{n}_R)\}+\al^2 Q(\mu^{\frac{1}{2}}G^n_R,\mu^{\frac{1}{2}}G^n_R),\\
\eps G^{n+1}_{R,2}&-\beta^n\na_v\cdot (v G^{n+1}_{R,2})-\al \na_v\cdot(AvG^{n+1}_{R,2})
+LG^{n+1}_{R,2}-(1-\chi_{M})\mu^{-\frac{1}{2}}\CK G^{n+1}_{R,1}=0.
\end{split}
\end{array}\right.
\end{eqnarray}
Here the parameter $\eps>0$ is introduced such that all the conservation laws for $G_R^{n+1}$ as in \eqref{con.g1r} can be satisfied. Moreover
we have denoted
\begin{align}%\label{n-sq}
\mu^{\frac{1}{2}}G^{n}_R=G^{n}_{R,1}+\mu^{\frac{1}{2}}G^{n}_{R,2},\quad n\geq0,\notag
\end{align}
and
\begin{align}\label{bt1n}
 \beta^n=\al\beta_0+\al^2\beta_1^n
 \end{align}
 with
 \begin{align}%\label{bt1n}
\beta_1^n=\frac{1}{3}\int_{\R^3} G_1LG_1\,dv
-\frac{\al}{3}\int_{\R^3} \FP_1\{ v\cdot (Av)\sqrt{\mu}\}G_R^n\,dv,\quad n\geq0,\notag
\end{align}
as well as
$$
[G^{0}_{R,1},G^{0}_{R,2}]=[0,0].
$$
Note that the approximation solutions are constructed to satisfy \eqref{gr12-ls}, by which the following identities hold true
\begin{align}
&\left\lag\left(\beta_1^{n+1}-\frac{1}{3}\lag G_1,LG_1\rag \right)\na_v\cdot(v\mu),\frac{1}{2}|v|^2\right\rag
-\frac{1}{2}\al \left\lag v\cdot(Av)\mu^{\frac{1}{2}}G^{n+1}_{R,2},\frac{1}{2}|v|^2\right\rag
\notag\\&\qquad-\al \left\lag\na_v\cdot(AvG^{n+1}_{R,1}),\frac{1}{2}|v|^2\right\rag
-\al \left\lag\na_v\cdot(AvG^{n+1}_{R,2}),\frac{1}{2}|v|^2\sqrt{\mu}\right\rag=0,\notag
\end{align}
and
\begin{align}
&\frac{1}{3}\left\lag\lag G_1,LG_1\rag \na_v\cdot(v\mu),\frac{1}{2}|v|^2\right\rag
+ \left\lag \na_v\cdot(Av\sqrt{\mu}G_1),\frac{1}{2}|v|^2\right\rag=0,\notag
\end{align}
so that one can show the conservation laws  \eqref{con.g1r} for $G_R^{n+1}$.

%\frac{1}{3}\left\lag\lag G_1,LG_1\rag \na_v\cdot(v\mu),

The proof of Theorem \ref{st.sol}  follows by three steps. First, we show the well-posedness of the system \eqref{gr12-ls}
for given $[G^{n}_{R,1},G^{n}_{R,2}]$ and $\eps>0$. Second, we establish the limit process $n\rightarrow+\infty$ for any fixed parameter $\eps>0$. Third, we pass the limits $\eps\rightarrow0^+$ to obtain the unique smooth solution of the system \eqref{GR1} and \eqref{GR2}.

\subsection{A uniform $L^\infty$ estimate with respect to the parameter $\si$}
Since both $K$ and $\CK$ are nonlocal and don't possess the property of smallness, it is convenient to
introduce the following linear vector operator parameterized by $\si\in[0,1]$ (cf.~\cite{DL-2020}):
$$
\mathscr{L}_\si[\CG_1,\CG_2]=[\mathscr{L}^1_\si,\mathscr{L}^2_\si][\CG_1,\CG_2],
$$
\begin{eqnarray*}%\label{laop}
\left\{\begin{array}{rll}
\begin{split}
\mathscr{L}^1_{\si}[\CG_1,\CG_2]=&\eps\CG_1-\beta'\na_v\cdot (v\CG_1)-\al \na_v\cdot(Av\CG_1)
+\nu\CG_1
-\si\chi_{M}\CK\CG_1\\
&+\frac{\beta'}{2}|v|^2\sqrt{\mu}\CG_2
+\al\frac{v\cdot (Av)}{2}\sqrt{\mu}\CG_2-\bet''(\CG)\na_v\cdot(v\mu),\\
\mathscr{L}^2_\si[\CG_1,\CG_2]=&\eps\CG_2-\beta'\na_v\cdot (v\CG_2)-\al \na_v\cdot(Av\CG_2)
+\nu\CG_2-\si K\CG_2-\si(1-\chi_{M})\mu^{-\frac{1}{2}}\CK \CG_1,
\end{split}
\end{array}\right.
\end{eqnarray*}
where $\beta'$ is a given constant of order $\al$, and
\begin{equation}\label{def.betaG}
\beta''(\CG)=-\frac{\al}{3}\int_{\R^3} \FP_1\{ v\cdot (Av)\}(\CG_1+\sqrt{\mu}\CG_2)\,dv.
\end{equation}
We then consider the solvability of the general coupled linear system
\begin{align}\label{pals}
\left\{\begin{array}{rll}
&\mathscr{L}^1_{\si}[\CG_1,\CG_2]=\CF_1,\\[2mm]
&\mathscr{L}^2_\si[\CG_1,\CG_2]=\CF_2,
 \end{array}\right.
\end{align}
where $[\CF_1,\CF_2]$ is given.

\begin{remark}
Note that in the case of $0<\ga\leq1$ (hard potentials) and $\si\neq1$,
the approximation system \eqref{pals} does not imply $[\CG_1,\CG_2]\in\FX_m$ even if $[\CF_1,\CF_2]\in\FX_m$, because the
structural
damage of the linear operators $\CL$ and $L$ violates the following laws of conservation
$$\lag \CG_1,[1,v_i,|v|^2]\rag+\lag \CG_2,[1,v_i,|v|^2]\mu^{\frac{1}{2}}\rag=0,\ i=1,2,3.
$$
\end{remark}

Due to the above remark, different from \cite{DL-2020} in the pure Maxwell molecule case, a convenient functional space to be considered is the following
\begin{align}
\bar{\FX}_m=\big\{\CG=&[\CG_1,\CG_2]\big|\sum\limits_{k\leq m}\{\|w_l\na_v^k\CG_1\|_{L^\infty}+\|w_l\na_v^k\CG_2\|_{L^\infty}\}<+\infty,\
k,m\in\Z^+\big\}\notag
\end{align}
equipped with the norm
$$
\|[\CG_1,\CG_2]\|_{\bar{\FX}_m}=\sum\limits_{k\leq m}\{\|w_l\na_v^k\CG_1\|_{L^\infty}+\|w_l\na_v^k\CG_2\|_{L^\infty}\}.
$$
The main idea showing the well-posedness of \eqref{pals} is to adopt the bootstrap argument based on the following
{\it a priori} $L^\infty$ estimates.

\begin{lemma}[{\it a priori} estimate]\label{lifpri}
Assume that $\beta'\neq0$ is of order $\al$. Let $[\CG_1,\CG_2]\in\bar{\FX}_{m}$ with $m\geq 0$ be a solution to \eqref{pals} with $\eps>0$ and suitably small, $\si\in[0,1)$ and $[\CF_1,\CF_2]\in\bar{\FX}_m$. There is $l_0>0$ such that for any $l\geq l_0$ arbitrarily large, there are $\al_0=\al_0(l)>0$ and large $M=M(l)>0$ such that for any $0<\al<\al_0$ with $C\al <1-\si$ for a generic large constant $C>0$,  the solution $[\CG_1,\CG_2]$ of the system \eqref{pals}  satisfies the following estimate
\begin{align}\label{Lif.es1}
\|[\CG_1,\CG_2]\|_{\bar{\FX}_{m}}=\|\mathscr{L}_\si^{-1}[\CF_1,\CF_2]\|_{\bar{\FX}_{m}}\leq C_\mathscr{L}\sum\limits_{0\leq k\leq m}\left\{\|w_{l}\na_v^k\CF_1\|_{L^\infty}+\|w_{l}\na_v^k\CF_2\|_{L^\infty}\right\},
\end{align}
where the constant $C_\mathscr{L}>0$ depends on $\eps$ but not on $\si$  and $\al$.
\end{lemma}

\begin{proof}
The proof is divided into two steps.

\medskip
\noindent\underline{Step 1. $L^\infty$ estimates.}  Taking $0\leq k\leq m$ and $l>0$, we set $H_{1,k}=w_{l}\na_v^k\CG_1$ and $H_{2,k}=w_{l}\na_v^k\CG_2$. Then, $\FH_k=[H_{1,k},H_{2,k}]$ satisfies the following equations:
\begin{align}
\eps H_{1,k}&-\bet' \na_v\cdot(v H_{1,k})+2l \bet' \frac{|v|^2}{1+|v|^2} H_{1,k}
-\al \na_v\cdot(AvH_{1,k})+2l \al \frac{v\cdot(Av)}{{1+|v|^2}}H_{1,k}
+\nu H_{1,k}\notag\\
&-\si\chi_{M}w_l\CK \left(\frac{H_{1,k}}{w_{l}}\right)
-w_l\bet''(\frac{\FH_{0}}{w_l})\na_v^k\na_v\cdot(v\mu)\notag\\
= &{\bf1}_{k'=1}w_{l}\bet' C_k^{k'}\na_v\cdot(\na_v^{k'}v\na_v^{k-k'}\CG_1)
+{\bf1}_{k'=1}\al C_k^{k'}w_l\na_v\cdot(\na_{v}^{k'}(Av)\na_v^{k-k'}\CG_1)
\notag\\&-\frac{\beta'}{2}w_l\sum\limits_{k'\leq k}C_k^{k'}\na_v^{k'}(|v|^2\mu^{\frac{1}{2}})\na_v^{k-k'}\CG_2
-\frac{\al}{2}\sum\limits_{k'\leq k}w_{l} C_k^{k'}\na_v^{k'}\left(v\cdot (Av)\mu^{\frac{1}{2}}\right)\na_v^{k-k'}\CG_2
\notag\\&-{\bf1}_{k>0}w_l\sum\limits_{0<k'\leq k}C_k^{k'}\na_v^{k'}\nu\na^{k-k'}_v\CG_1+{\bf1}_{k>0}\si\sum\limits_{0<k'\leq k}C_k^{k'}w_{l}\na_v^{k'}(\chi_{M}\CK) \na_v^{k-k'}\CG_1\notag\\&+w_{l}\na_v^k\CF_1,\label{H10}
\end{align}
and
\begin{align}
\eps H_{2,k}&-\bet' \na_v\cdot(v H_{2,k})+2l \bet' \frac{|v|^2}{1+|v|^2} H_{2,k}
-\al \na_v\cdot (AvH_{2,k})+2l \al \frac{v\cdot (Av)}{{1+|v|^2}}H_{2,k}
\notag\\
&+\nu H_{2,k}-\si w_{l}K\left(\frac{H_{2,k}}{w_{l}}\right)
\notag\\=&{\bf1}_{k'=1}w_{l}\bet' C_k^{k'}\na_v\cdot(\na_v^{k'}v\na_v^{k-k'}\CG_2)
+{\bf1}_{k'=1}\al C_k^{k'}w_{l}\na_v\cdot(\na^{k'}_{v}(Av)\na_v^{k-k'}\CG_2)
\notag\\&-{\bf1}_{k>0}w_l\sum\limits_{0<k'\leq k}C_k^{k'}\na_v^{k'}\nu \na^{k-k'}_v\CG_2 +{\bf1}_{k>0}\si w_{l}\sum\limits_{0<k'\leq k}C_k^{k'}\na_v^{k'}K \na^{k-k'}_v\CG_2 \notag\\
&+\si\sum\limits_{k'\leq k}C_k^{k'}w_{l}\na_v^{k'}((1-\chi_{M})\mu^{-\frac{1}{2}}\CK) \na_v^{k-k'}\CG_1+w_{l}\na_v^k\CF_2,\label{H20}
\end{align}
where $\FH_0:=[H_1,H_2]=[H_{1,0},H_{2,0}]=w_{l}[\CG_1,\CG_2].$
Notice that \eqref{H10} and \eqref{H20} are linear PDEs of first order, it is convenient to
apply the method of characteristics to obtain $L^\infty$ estimate (cf.~\cite{EGKM-13,EGKM-18}). To do this, we first introduce a uniform parameter $t\in\R$, and regard $H_{i,k}(v)=H_{i,k}(t,v)$$(i=1,2)$, then define the characteristic line $[s,V(s;t,v)]$ for equations \eqref{H10} and \eqref{H20} going through $(t,v)$ such that
\begin{align}\label{H1CL}
\left\{\begin{array}{rll}
&\frac{d V}{ds}=-\beta' V(s;t,v)-\al A V(s;t,v),\\[2mm]
%&\frac{d V_i}{ds}=-\beta' V_i(s;t,v),\ i=2,3,\\[2mm]
&V(t;t,v)=v,
\end{array}\right.
\end{align}
which is equivalent to
\begin{equation*}%\label{H1CLp1}
\begin{split}
V(s)=V(s;t,v)=e^{-(s-t)(\beta' I+\al A)}v.
\end{split}
\end{equation*}
Since $\beta'\neq0$, it is natural to expect that %$\lim\limits_{s\rightarrow-\infty}
$|V(s)|\to+\infty$ as $s\rightarrow-\infty$ and
%$\lim\limits_{|v|\rightarrow+\infty}
$G_R(v)\to 0$ as $|v|\rightarrow+\infty$.
 Due to this, integrating along the backward trajectory \eqref{H1CL} with respect to $s\in(-\infty,t]$, one can write the solutions of \eqref{H10} and \eqref{H20} as the mild form of
\begin{align}%\label{H10m}
H_{1,k}(v(t))=\sum\limits_{i=1}^6\CI_{i},\notag
\end{align}
with
\begin{align}
\CI_{1}=\si\int_{-\infty}^{t}e^{-\int_{s}^t\CA^\eps(\tau,V(\tau))d\tau}\left\{\chi_{M}w_l\CK
\left(\frac{H_{1,k}}{w_{l}}\right)\right\}(V(s))\,ds,\notag
\end{align}
\begin{align}
\CI_{2}=
\int_{-\infty}^{t}e^{-\int_{s}^t\CA^\eps(\tau,V(\tau))d\tau}
\left\{w_l\bet''(\frac{\FH_0}{w_l})\na_v^k\na_v\cdot(v\mu)\right\}(V(s))\,ds,
\notag
\end{align}
\begin{align}
\CI_{3}=&\int_{-\infty}^{t}e^{-\int_{s}^t\CA^\eps(\tau,V(\tau))d\tau}
\bigg\{{\bf1}_{k'=1}w_{l}\bet' C_k^{k'}\na_v\cdot(\na_v^{k'}v\na_v^{k-k'}\CG_1)
\notag\\&\qquad\qquad\qquad\qquad\qquad\qquad+{\bf1}_{k'=1}\al C_k^{k'}w_l\na_v\cdot(\na_{v}^{k'}(Av)\na_v^{k-k'}\CG_1)\bigg\}(V(s))\,ds,\notag
\end{align}
\begin{align}
\CI_{4}=&-\int_{-\infty}^{t}e^{-\int_{s}^t\CA^\eps(\tau,V(\tau))d\tau}\sum\limits_{k'\leq k}C_k^{k'}\bigg\{\frac{\beta'}{2}w_l\na_v^{k'}(|v|^2\mu^{\frac{1}{2}})\na_v^{k-k'}\CG_2
\notag\\&\qquad\qquad\qquad\qquad\qquad\qquad\qquad\qquad
+\frac{\al}{2}w_{l} \na_v^{k'}\left(v\cdot (Av)\mu^{\frac{1}{2}}\right)\na_v^{k-k'}\CG_2\bigg\}(V(s))\,ds,\notag
\end{align}
\begin{align}
\CI_{5}=&{\bf1}_{k>0}\si\int_{-\infty}^{t}e^{-\int_{s}^t\CA^\eps(\tau,V(\tau))d\tau}\sum\limits_{0<k'\leq k}C_k^{k'}
\bigg\{-w_l\na_v^{k'}\nu \na^{k-k'}_v\CG_1
%\notag\\&\qquad\qquad\quad
+w_{l}\na_v^{k'}(\chi_{M}\CK) \na_v^{k-k'}\CG_1 \bigg\}(V(s))\,ds,\notag
\end{align}
\begin{align}
\CI_{6}=\int_{-\infty}^{t}e^{-\int_{s}^t\CA^\eps(\tau,V(\tau))d\tau}\left(w_{l}\na_v^k\CF_1\right)(V(s))\,ds,\notag
\end{align}
and
\begin{align}%\label{H20m}
H_{2,k}=\sum\limits_{i=7}^{11}\CI_{i},\notag
\end{align}
with
\begin{align}
\CI_{7}=&\si\int_{-\infty}^{t}e^{-\int_{s}^t\CA^\eps(\tau,V(\tau))d\tau}\left[w_{l}K\left(\frac{H_{2,k}}{w_{l}}\right)\right](V(s))\,ds,
\notag
\end{align}
\begin{align}
\CI_{8}=&\int_{-\infty}^{t}e^{-\int_{s}^t\CA^\eps(\tau,V(\tau))d\tau}\bigg\{{\bf1}_{k'=1}w_{l}\bet' C_k^{k'}\na_v\cdot(\na_v^{k'}v\na_v^{k-k'}\CG_2)
\notag\\&\qquad\qquad\qquad\qquad\qquad\qquad+{\bf1}_{k'=1}\al C_k^{k'}w_l\na_v\cdot(\na_{v}^{k'}(Av)\na_v^{k-k'}\CG_2)\bigg\}(V(s))\,ds,\notag
\end{align}
\begin{align}
\CI_{9}=&\si{\bf1}_{k>0}\int_{-\infty}^{t}e^{-\int_{s}^t\CA^\eps(\tau,V(\tau))d\tau}\sum\limits_{0<k'\leq k}C_k^{k'}
\left\{-w_{l}\na_v^{k'}\nu \na^{k-k'}_v\CG_2+w_l\na_v^{k'}K \na^{k-k'}_v\CG_2\right\}(V(s))\,ds,
\notag
\end{align}
\begin{align}
\CI_{10}=&\si\int_{-\infty}^{t}e^{-\int_{s}^t\CA^\eps(\tau,V(\tau))d\tau}
\left\{\sum\limits_{k'\leq k}C_k^{k'}w_{l}\na_v^{k'}((1-\chi_{M})\mu^{-\frac{1}{2}}\CK) \left(\na_v^{k-k'}\CG_1\right)\right\}(V(s))\,ds,\notag \end{align}
\begin{align}
\CI_{11}=&\int_{-\infty}^{t}e^{-\int_{s}^t\CA^\eps(\tau,V(\tau))d\tau}\left(w_{l}\na_v^k\CF_2\right)(V(s))\,ds,\notag
\end{align}
where
\begin{align}%\label{sglw}
\CA^\eps(\tau,V(\tau))=\nu(V(\tau))+\eps-3\beta'+2l \bet' \frac{|V(\tau)|^2}{1+|V(\tau)|^2} +2l \al \frac{V(\tau)\cdot (AV(\tau))}{{1+|V(\tau)|^2}}-\al{\rm tr}A\geq \frac{1}{2}\nu(V(\tau)),\notag
\end{align}
provided that $\eps>0$, $\al>0$, $l\al$ and $|\al{\rm tr}A|$ are suitably small.
Note that $\nu(V(\tau))$ is independent of $V(\tau)$ in the Maxwell molecule case.
Here and in the sequel, the velocity derivatives $\na^{k'}_v$ acting on the nonlocal operators such as $\CK$, $K$ etc. are understood in the way as \eqref{der-nop}.

In what follows, we will compute $\CI_i$ $(1\leq i\leq11)$, separately.
%The estimations will be divided into two case.
%\noindent{\it Case 1. $k=0$.} In this case, the terms $I_4, I_5, I_6$, $I_{10}, I_{11}$ and $I_{12}$ vanish.
The estimates for $\CI_1$ is divided into two cases. If $\ga=0$ i.e. the Maxwell molecule case, we apply \eqref{CK1} in Lemma \ref{CK} to obtain that
\begin{align}\notag
\CI_1\leq \frac{C}{l }\|H_{1,k}(v)\|_{L^\infty}\int_{-\infty}^{t}e^{-\frac{\nu_0}{2}(t-s)}\,ds\leq \frac{C}{l }\|H_{1,k}(v)\|_{L^\infty},
\end{align}
where
\begin{align}%\label{sp.L}
\nu_0=\int_{\R^3}\int_{\S^2}B_0(\cos \ta)\mu(v_\ast)\, d\om dv_\ast>0.\notag
\end{align}
If $0<\ga\leq1$, Lemma \ref{g-ck-lem} leads us to
\begin{align}\notag
\CI_1\leq& \int_{-\infty}^{t}e^{-\int_{s}^t\nu (V(\tau))d\tau}\nu (V(s))[\nu (V(s))]^{-1}\left\{\chi_{M}w_l\CK
\left(\frac{H_{1,k}}{w_{l}}\right)\right\}(V(s))\,ds\notag\\
\leq& \int_{-\infty}^{t}e^{-\int_{s}^t\nu (V(\tau))d\tau}\nu (V(s))
\left(\frac{C}{(1+M)^{\ga/2}}+\varsigma\right)\|H_{1,k}(v)\|_{L^\infty}
\,ds\notag\\
\leq& \left(\frac{C}{(1+M)^{\ga/2}}+\varsigma\right)\|H_{1,k}(v)\|_{L^\infty},\notag
\end{align}
where the following estimate has been used:
\begin{align}
\int_{-\infty}^{t}e^{-\int_{s}^t\nu (V(\tau))d\tau}\nu (V(s))ds\leq1.\notag
\end{align}
By virtue of \eqref{def.betaG}, one has
\begin{align}\notag
\CI_2\leq C\al\|H_{1,0}\|_{L^\infty}+C\al\|H_{2,0}\|_{L^\infty}.
\end{align}
It is straightforward to see that
\begin{align}\notag
 \CI_3\leq C\al\sum\limits_{k'\leq k}\|H_{1,k'}\|_{L^\infty},\ \CI_4,\CI_{8}\leq C\al\sum\limits_{k'\leq k}\|H_{2,k'}\|_{L^\infty}.
\end{align}
For $\CI_5$, we first rewrite $\na_v^{k'}(\chi_{M}\CK)(\na_v^{k-k'}\CG_1)$ as
\begin{align*}
\na_v^{k'}(\chi_{M}\CK) (\na_v^{k-k'}\CG_1)=&\sum\limits_{k''\leq k'}C_{k'}^{k''}\na_v^{k'-k''}\chi_{M}\na_v^{k''}\CK (\na_v^{k-k'}\CG_1)\notag\\
=&\sum\limits_{k''\leq k'}C_{k'}^{k''}\na_v^{k'-k''}\chi_{M}\na_v^{k''}\left\{Q (\mu,\na_v^{k-k'}\CG_1)
+Q (\na_v^{k-k'}\CG_1,\mu)\right\}.
\end{align*}
Then it follows that
\begin{align}\notag
\CI_5{\bf1}_{k\geq1}\leq C\sum\limits_{k'<k}\|H_{1,k'}\|_{L^\infty},
\end{align}
according to Lemma \ref{op.es.lem}. And likewise, we also have
\begin{align}\notag
\CI_{10}\leq C\sum\limits_{k'\leq k}\|H_{1,k'}\|_{L^\infty}.
\end{align}
Next, Lemma \ref{Ga} leads us to have
\begin{align}\notag
\CI_{9}{\bf1}_{k\geq1}\leq C\sum\limits_{k'<k}\|H_{2,k'}\|_{L^\infty}.
\end{align}
For $\CI_6$ and $\CI_{11}$, one directly has
%in view of Lemma \ref{op.es.lem}, \eqref{g1g2.es} and \eqref{sglw}, we have
\begin{align}\notag
\CI_6
\leq C\|w_{l}\na_v^k\CF_1\|_{L^\infty}, \ \CI_{11}\leq C\|w_{l}\na_v^k\CF_2\|_{L^\infty}.
\end{align}
Finally, for the delicate term $\CI_{7}$, we divide our computations into the following three cases.

\medskip
\noindent{\it Case 1.} $|V|\geq M$ with $M$ suitably large.
From Lemma \ref{Kop}, it follows that
$$
\int\mathbf{k}_w(V,v_\ast)\,dv_\ast\leq \frac{C}{1+|V|}\leq \frac{C}{M}.
$$
Applying this, one has
\begin{equation}\label{I81}
\CI_{7}\leq \sup\limits_{-\infty<s\leq t}\int_{\R^3}\mathbf{k}_w(V,v_\ast)\,dv_\ast\|H_{2,k}\|_{L^\infty}\leq \frac{C}{M}\|H_{2,k}\|_{L^\infty}.
\end{equation}

\medskip
\noindent{\it Case 2.} $|V|\leq M$ and $|v_\ast|\geq 2M$. In this situation, we have
$|V-v_\ast|\geq M$, then
\begin{equation*}
\mathbf{k}_w(V,v_\ast)
\leq Ce^{-\frac{\vps M^2}{8}}\mathbf{k}_w(V,v_\ast)e^{\frac{\vps |V-v_\ast|^2}{8}}.
\end{equation*}
Using Lemma \ref{Kop}, one sees that
$\int\mathbf{k}_w(V,v_\ast)e^{\frac{\vps |V-v_\ast|^2}{8}}\,dv_\ast$ is still bounded. Therefore,
by a similar argument as for obtaining \eqref{I81}, it follows that
\begin{equation*}%\label{J32}
\begin{split}
\CI_{7}\leq Ce^{-\frac{\vps M^2}{8}}\|H_{2,k}\|_{L^\infty}.
\end{split}
\end{equation*}
To complete our estimates for $\CI_{7}$, we are now in a position to handle the last case:

\noindent{\it Case 3.} $|V|\leq M$ and $|v_\ast|\leq 2M$. In this case, the key point is to convert the bound in $L^\infty$-norm to the one in $L^2$-norm which will be established later on. To do so, for any large $M>0$,
we choose a number $p=p(M)$ to define
\begin{equation}\label{km}
\mathbf{k}_{w,p}(V,v_\ast)\equiv \mathbf{1}_{|V-v_\ast|\geq\frac{1}{p},|v_\ast|\leq p}\mathbf{k}_{w}(V,v_\ast),
\end{equation}
such that $\sup\limits_{V}\int_{\mathbf{R}^{3}}|\mathbf{k}_{w,p}(V,v_\ast)
-\mathbf{k}_{w}(V,v_\ast)|\,dv_\ast\leq
\frac{1}{M}.$ One then has
\begin{align*}
%\begin{split}
\CI_{7}&\leq C\sup\limits_{s}\int_{|v_\ast|\leq 2M}\mathbf{k}_{w,p}(V,v_\ast)|\na_v^k\CG_2(v_\ast)|dv_\ast+\frac{1}{M}\|H_{2,k}\|_{L^\infty}\\
&\leq C(p)\sup\limits_{s}\|\na_v^k\CG_2\|+\frac{1}{M}\|H_{2,k}\|_{L^\infty},
%\end{split}
\end{align*}
according to H\"older's inequality and the fact that $\int_{\R^3}\mathbf{k}^2_{w,p}(V,v_\ast)dv_\ast<\infty.$

Therefore, it follows that for any large $M>0$,
\begin{equation}%\label{CI8}
\begin{split}
\CI_{7}\leq C\left(e^{-\frac{\vps M^2}{8}}+\frac{1}{M}\right)\|H_{2,k}\|_{L^\infty}+C\|\na_v^k\CG_2\|.\notag
\end{split}
\end{equation}
Combing all the estimates above together, we now arrive at
\begin{align}\label{lifn}
\left\{\begin{array}{rll}
&\|H_{1,k}\|_{L^\infty}\leq
\left({\bf1}_{0<\ga\leq 1}\frac{C}{(1+M)^{\ga/2}}+\varsigma+\frac{C}{l}+C\al\right)\|H_{1,k}\|_{L^\infty}+C\al\|H_{1,0}\|_{L^\infty}
\\[4mm]&\qquad\qquad\quad+C\al\sum\limits_{k'\leq k}\|H_{2,k'}\|_{L^\infty}
+{\bf1}_{k\geq1}C\sum\limits_{k'<k}\|H_{1,k'}\|_{L^\infty}
+C\|w_{l}\na_v^k\CF_1\|_{L^\infty},\\[4mm]
&\|H_{2,k}\|_{L^\infty}\leq \left(e^{-\frac{\vps M^2}{8}}+\frac{C}{M}+C\al\right)\|H_{2,k}\|_{L^\infty}+{\bf1}_{k\geq1}C\sum\limits_{k'<k}\|H_{2,k'}\|_{L^\infty}
\\[4mm]&\qquad\qquad\quad+C\sum\limits_{k'\leq k}\|H_{1,k'}\|_{L^\infty}+C\|\na_v^k\CG_2\|+C\|w_{l}\na_v^k\CF_2\|_{L^\infty}.\end{array}\right.
\end{align}
It should be pointed out that the constant $C$ in \eqref{lifn} is independent of $\si$ and $\eps$.

\medskip
\noindent\underline{Step 2. $L^2$ estimates.}
To close our estimates, we now turn to deduce
the $H^k$ estimate on $\CG_2$. To do this, we start from the basic $L^2$ estimate of $\CG_2$.
By the inner product $\lag \eqref{pals}_2, \CG_2\rag$, one has
\begin{align}\label{CG2-l2}
\eps\lag&\CG_2,\CG_2\rag
-\beta'\lag\na_v\cdot (v\CG_2),\CG_2\rag
-\al \lag\na_v\cdot(Av\CG_2),\CG_2\rag
\notag\\&+(1-\si)\lag\nu \CG_2,\CG_2\rag
+\si \lag L\CG_2,\CG_2\rag-\si\lag(1-\chi_{M})\mu^{-\frac{1}{2}}\CK \CG_1,\CG_2\rag=\lag\CF_2,\CG_2\rag,
\end{align}
where we have used the identity
\begin{align}%\label{keyob}
\nu f-\si Kf=(1-\si)\nu f+\si Lf.\notag
\end{align}
Applying Lemma \ref{es-L} and Cauchy-Schwarz's inequality, we get from \eqref{CG2-l2} that for $l>\frac{3}{2}$
\begin{align}
\eps\|\CG_2\|^2&+(1-\si)\|\CG_2\|^2_\nu+\de_0\si\|\FP_1\CG_2\|^2_\nu\notag\\
\leq& C\al\|\CG_2\|^2+\frac{\eps}{4}\|\CG_2\|^2+\frac{C}{\eps}\|w_l\CG_1\|_{L^\infty}^2+\frac{C}{\eps}\|w_l\CF_2\|_{L^\infty}^2,\notag
\end{align}
which further implies
\begin{align}\label{Cg2-l2z}
\frac{\eps}{2}\|\CG_2\|^2+(1-\si)\|\CG_2\|^2_\nu+\de_0\si\|\FP_1\CG_2\|^2_\nu\leq \frac{C}{\eps}\|w_l\CG_1\|_{L^\infty}^2+\frac{C}{\eps}\|w_l\CF_2\|_{L^\infty}^2,
\end{align}
provided that $0<\al\ll 1-\si$ with $0\leq\si<1.$

To deduce the higher order $L^2$ estimate on $\CG_2$, one gets from
$\lag \na_v^k\eqref{pals}_2, \na_v^k\FP_1\CG_2\rag$ that for $k\geq1$
\begin{align}
\eps\lag&\na_v^k(\FP_1\CG_2+\FP_0\CG_2),\na_v^k\FP_1\CG_2\rag
-\beta'\lag\na_v^k\na_v\cdot (v\FP_1\CG_2),\na_v^k\FP_1\CG_2\rag
\notag\\&-\beta'\lag\na_v^k\na_v\cdot (v\FP_0\CG_2),\na_v^k\FP_1\CG_2\rag
-\al \lag\na_v^k\na_v\cdot(Av\FP_1\CG_2),\na_v^k\FP_1\CG_2\rag
\notag\\&-\al \lag\na_v^k\na_v\cdot(Av\FP_0\CG_2),\na_v^k\FP_1\CG_2\rag
+(1-\si)\lag\nu\na_v^k\FP_1\CG_2,\na_v^k\FP_1\CG_2\rag
\notag\\&+(1-\si)\sum\limits_{1\leq k'\leq k}C_k^{k'}{\lag\na_v^{k'}\nu\na_v^{k-k'}\FP_1\CG_2,\na_v^k\FP_1\CG_2\rag}
+(1-\si)\sum\limits_{k'\leq k}C_k^{k'}{\lag\na_v^{k'}\nu\na_v^{k-k'}\FP_0\CG_2,\na_v^k\FP_1\CG_2\rag}
\notag\\&+\si \lag\na_v^k(L\CG_2),\na_v^k\FP_1\CG_2\rag
-\si\lag\na_v^k[(1-\chi_{M})\mu^{-\frac{1}{2}}\CK \CG_1],\na_v^k\FP_1\CG_2\rag=\lag\na_v^k\CF_2,\na_v^k\FP_1\CG_2\rag,\notag
\end{align}
from which, by using Lemma \ref{es-L} and Cauchy-Schwarz's inequality again, we further obtain
\begin{align}\label{Cg2-l2h}
\eps\|\na_v^k\FP_1\CG_2\|^2&+(1-\si)\|\na_v^k\FP_1\CG_2\|^2_\nu+\de_1\si\|\na_v^k\FP_1\CG_2\|^2_\nu-C\|\FP_1\CG_2\|^2\notag\\
\leq& C(\al+\eta)\|\na_v^k\FP_1\CG_2\|^2+C_\eta\sum\limits_{k'< k}\|\na_v^{k'}\FP_1\CG_2\|^2+C\sum\limits_{k'\leq k}\|w_l\na_v^{k'}\CG_1\|_{L^\infty}^2\notag\\&+C\|w_l\na_v^k\CF_2\|_{L^\infty}^2
+C\|\FP_0\CG_2\|^2,
\end{align}
where $\eta>0$ is suitably small.

As a consequence, a linear combination of \eqref{Cg2-l2z} and \eqref{Cg2-l2h} with $k=1,2,\cdots,m$ yields
\begin{align}\label{Cg2-l2-fin}
\eps\sum\limits_{1\leq k\leq m}&\|\na_v^k\FP_1\CG_2\|^2+\eps\|\FP_0\CG_2\|^2+\la\sum\limits_{k\leq m}\|\na_v^k\FP_1\CG_2\|^2_\nu\notag\\
\leq& C(\eps)\sum\limits_{k\leq m}\|w_l\na_v^{k}\CG_1\|_{L^\infty}^2+C(\eps)\sum\limits_{k\leq m}\|w_l\na_v^k\CF_2\|_{L^\infty}^2.
\end{align}
Finally, taking the linear combination of \eqref{lifn} and \eqref{Cg2-l2-fin} for $0\leq k\leq m$ and adjusting constants, we get
\begin{align*}
\sum\limits_{0\leq k\leq m}\left\{\|H_{1,k}\|_{L^\infty}+\|H_{2,k}\|_{L^\infty}\right\}\leq
C(\eps)\sum\limits_{0\leq k\leq m}\|w_{l}\na_v^k[\CF_1,\CF_2]\|_{L^\infty}.
\end{align*}
%which further gives rise to \eqref{Lif.es1}.
This shows the desired estimate \eqref{Lif.es1} and ends the proof of Lemma \ref{lifpri}.
\end{proof}

\subsection{Existence for the linear problem with fixed $\epsilon>0$}
With Lemma \ref{lifpri} in hand, we now turn to prove the existence of solutions to \eqref{pals} with fixed $\eps>0$ in $L^\infty$ framework by the contraction mapping method.

\begin{lemma}\label{ex.pals}
%For any $m\geq0$, Let  $[\CF_1,\CF_2]$ satisfy \eqref{F12} with $0\leq k\leq m$, then
Let all the assumptions of Lemma \ref{lifpri} be satisfied. There is $l_0>0$ such that for any $l\geq l_0$ arbitrarily large, there are $\al_0=\al_0(l)>0$ and large $M=M(l)>0$ such that for any $0<\al<\al_0$,
there exists a unique solution $[\CG_1,\CG_2]\in\bar{\FX}_{m}$
to \eqref{pals} with $\si=1$ satisfying
\begin{align}\label{Lif.es2}
\sum\limits_{0\leq k\leq m}&\left\{\|w_{l}\na_{v}^{k}\CG_1\|_{L^\infty}+\|w_{l}\na_{v}^{k}\CG_2\|_{L^\infty}\right\}
\leq
C\sum\limits_{0\leq k\leq m}\left\{\|w_{l}\na_{v}^{k}\CF_1\|_{L^\infty}+\|w_{l}\na_{v}^{k}\CF_2\|_{L^\infty}\right\}.
\end{align}
\end{lemma}

\begin{proof} Our proof relies on the {\it a priori} estimate \eqref{Lif.es1} established in Lemma \ref{lifpri} and the bootstrap argument, cf. \cite{DL-2020,DLY-2021,DHWZ-19}.
%As the proof of Lemma \ref{ex-G1-rf}, the proof here is also divided into following three steps.

\noindent\underline{{\it Step 1. Existence for $\si=0$.}}
If $\si=0$, then \eqref{pals} becomes
\begin{align*}%\label{laop}
\eps\CG_1&-\beta'\na_v\cdot (v\CG_1)-\al \na_v\cdot(Av\CG_1)
+\nu\CG_1\\
&+\frac{\beta'}{2}|v|^2\sqrt{\mu}\CG_2
+\al\frac{v\cdot (Av)}{2}\sqrt{\mu}\CG_2-\bet''(\CG)\na_v\cdot(v\mu)=\CF_1,
\end{align*}
and
\begin{align*}%\label{laop}
\eps\CG_2-\beta'\na_v\cdot (v\CG_2)-\al \na_v\cdot(Av\CG_2)
+\nu\CG_2=\CF_2.
\end{align*}
Then, in this simple case of $\si=0$, since there is no trouble term involving $K$ and
$\CK$, the existence of $L^\infty$-solutions can be easily proved by the characteristic method and the contraction mapping theorem. That is, it follows immediately that
\begin{align}\label{L0}
\|\mathscr{L}_0^{-1}[\CG_1,\CG_2]\|_{\bar{\FX}_{m}}\leq C_\mathscr{L}\|[\CF_1,\CF_2]\|_{\bar{\FX}_{m}}.
\end{align}

\noindent\underline{{\it Step 2. Existence for $\si\in[0,\si_\ast]$ for some $\si_\ast>0$.}}
Letting $\si\in(0,1)$, we now consider
\begin{align}\label{CG1st}
\eps\CG_1&-\beta'\na_v\cdot (v\CG_1)-\al \na_v\cdot(Av\CG_1)
+\nu\CG_1\notag\\
&+\frac{\beta'}{2}|v|^2\sqrt{\mu}\CG_2
+\al\frac{v\cdot (Av)}{2}\sqrt{\mu}\CG_2-\bet''(\CG)\na_v\cdot(v\mu)=\si \chi_{M}\CK \CG_{1}+\CF_1.
\end{align}
and
\begin{align}\label{CG2st}
\eps\CG_2-\beta'\na_v\cdot (v\CG_2)-\al \na_v\cdot(Av\CG_2)
+\nu\CG_2=\si  K\CG_2+\si (1-\chi_{M})\mu^{-\frac{1}{2}}\CK\CG_1+\CF_2,
\end{align}
To verify the well-posedness of the above system, we further design the following approximation equations
\begin{align}\label{CG1stn}
\eps\CG^{n+1}_1&-\beta'\na_v\cdot (v\CG^{n+1}_1)-\al \na_v\cdot(Av\CG^{n+1}_1)
+\nu\CG^{n+1}_1\notag\\
&+\frac{\beta'}{2}|v|^2\sqrt{\mu}\CG^{n+1}_2
+\al\frac{v\cdot (Av)}{2}\sqrt{\mu}\CG^{n+1}_2-\bet''(\CG^{n+1})\na_v\cdot(v\mu)=\si \chi_{M}\CK \CG^{n}_{1}+\CF_{1}
:=\CF_1^{(1)},
\end{align}
and
\begin{align}\label{CG2stn}
\eps\CG^{n+1}_2&-\beta'\na_v\cdot (v\CG^{n+1}_1)-\al \na_v\cdot(Av\CG^{n+1}_1)
+\nu\CG^{n+1}_2\notag\\&=\si  K\CG^{n}_2+\si (1-\chi_{M})\mu^{-\frac{1}{2}}\CK\CG^{n}_1+\CF_2:=\CF^{(1)}_2,
\end{align}
with $[\CG^{0}_1,\CG^{0}_{2}]=[0,0]$. Our goal next is to prove: (i) $[\CG_1^n,\CG_2^n]_{n=0}^\infty$ is uniformly bounded in $\bar{\FX}_{m}$, (ii)
$[\CG_1^n,\CG_2^n]_{n=0}^\infty$ is a Cauchy sequence in $\bar{\FX}_{m}$. Thanks to \eqref{L0}, it follows
\begin{align}\label{siast.re}
\|[\CG_1^{n+1},\CG_2^{n+1}]\|_{\bar{\FX}_{m}}\leq& C_\SL
\|{[\CF_1^{(1)},\CF_2^{(1)}]}\|_{\bar{\FX}_{m}}\\
\leq& C_\SL \si \bar{C}_1\|[\CG_1^{n},\CG_2^{n}]\|_{\bar{\FX}_{m}}
+\underbrace{C_\mathscr{L}\sum\limits_{0\leq k\leq m}\left\{\|w_{q}\na_{v}^k\CF_1\|_{L^\infty}
+\|w_{q}\na_{v}^k\CF_2\|_{L^\infty}\right\}}_{\CM_0},\notag
\end{align}
where $\bar{C}_1>0$ is independent of $\si$ and $n$. Choosing $0<\si_\ast\ll1$ such that
\begin{align}\label{siast}
C_\SL \si_\ast \bar{C}_1\leq\frac{1}{2},
\end{align}
and moreover there exists a positive integer $N$ such that
\begin{align}\label{Nsiast}
N\si_\ast=1.
\end{align}
Then we get from \eqref{siast.re} that
\begin{align}\label{CGumbd}
\|[\CG_1^{n},\CG_2^{n}]\|_{\bar{\FX}_{m}}\leq 2\CM_0,
\end{align}
for all $n\geq0$. Furthermore, by \eqref{CG1stn}, \eqref{CG2stn} and \eqref{siast} and using \eqref{L0} once more,  one has \begin{align}\label{CGcase}
\|[\CG_1^{n+1},\CG_2^{n+1}]-[\CG_1^{n},\CG_2^{n}]\|_{\bar{\FX}_{m}}\leq& C_\SL \si \bar{C}_1\|[\CG_1^{n},\CG_2^{n}]-[\CG_1^{n-1},\CG_2^{n-1}]\|_{\bar{\FX}_{m}}\notag\\
\leq&\frac{1}{2}\|[\CG_1^{n},\CG_2^{n}]-[\CG_1^{n-1},\CG_2^{n-1}]\|_{\bar{\FX}_{m}}.
\end{align}
As a consequence, \eqref{CGcase} and \eqref{CGumbd} imply that the system \eqref{CG1st} and \eqref{CG2st} admits a unique solution $[\CG_1,\CG_2]\in\bar{\FX}_{m}$ for all $\si\in[0,\si_\ast].$
Moreover, utilizing Lemma \ref{lifpri}, for such a solution, we actually have the following uniform estimate
\begin{align}%\label{umbd1st}
\|[\CG_1,\CG_2]\|_{\bar{\FX}_{m}}\leq
C_\mathscr{L}\sum\limits_{0\leq k\leq m}\left\{\|w_{l}\na_{v}^k\CF_1\|_{L^\infty}
+\|w_{l}\na_{v}^k\CF_2\|_{L^\infty}\right\},\notag
\end{align}
which is also equivalent to
\begin{align}\label{Last}
\|\mathscr{L}_{\si_\ast}^{-1}[\CF_1,\CF_2]\|_{\bar{\FX}_{m}}\leq C_\mathscr{L}\|[\CF_1,\CF_2]\|_{\bar{\FX}_{m}}.
\end{align}

\noindent\underline{{\it Step 3. Existence for $\si\in[0,2\si_\ast]$ for some $\si_\ast>0$.}} By using \eqref{Last} and performing the similar calculations as for obtaining \eqref{CGumbd} and \eqref{CGcase}, for $\si'\in[0,\si_\ast],$ one can see that there exists a unique solution $[\CG_1,\CG_2]\in\bar{\FX}_{m}$ to the lifting system
\begin{align}\notag
\eps\CG_1&-\beta'\na_v\cdot (v\CG_1)-\al \na_v\cdot(Av\CG_1)
+\nu\CG_1+\frac{\beta'}{2}|v|^2\sqrt{\mu}\CG_2
+\al\frac{v\cdot (Av)}{2}\sqrt{\mu}\CG_2\notag\\
&-\bet''(\CG)\na_v\cdot(v\mu)-\si_\ast \chi_{M}\CK \CG_{1}=\si' \chi_{M}\CK \CG_{1}+\CF_1,\notag
\end{align}
and
\begin{align}
\eps\CG_2&-\beta'\na_v\cdot (v\CG_2)-\al \na_v\cdot(Av\CG_2)
+\nu\CG_2-\si_\ast  K\CG_2-\si_\ast (1-\chi_{M})\mu^{-\frac{1}{2}}\CK\CG_1\notag\\&=\si'  K\CG_2+\si' (1-\chi_{M})\mu^{-\frac{1}{2}}\CK\CG_1+\CF_2.\notag
\end{align}
In other words, we have proved the existence of $\mathscr{L}^{-1}_{2\si_\ast}$ on $\bar{\FX}_{m}$ and \eqref{Lif.es1} holds true for $\si=2\si_\ast$.

\noindent\underline{{\it Step 4. Existence for $\si=1$.}} In this final step, we shall show how to extend the existence of $\mathscr{L}^{-1}_{2\si_\ast}$ to the one of $\mathscr{L}^{-1}_1$ by the above procedure. As a matter of fact, using \eqref{Nsiast} and repeating Step 3 $N-2$ times, we can prove that $\mathscr{L}^{-1}_{(N-1)\si_\ast}$ is well-defined. With this, we then consider the following
lifting system
\begin{align}
\eps\CG_1&-\beta'\na_v\cdot (v\CG_1)-\al \na_v\cdot(Av\CG_1)
+\nu\CG_1+\frac{\beta'}{2}|v|^2\sqrt{\mu}\CG_2
+\al\frac{v\cdot (Av)}{2}\sqrt{\mu}\CG_2\notag\\
&-\bet''(\CG)\na_v\cdot(v\mu)-(N-1)\si_\ast \chi_{M}\CK \CG_{1}=\si' \chi_{M}\CK \CG_{1}+\CF_1,\label{N-1G1}
\end{align}
and
\begin{align}
\eps\CG_2&-\beta'\na_v\cdot (v\CG_2)-\al \na_v\cdot(Av\CG_2)
+\nu\CG_2-(N-1)\si_\ast  K\CG_2-(N-1)\si_\ast (1-\chi_{M})\mu^{-\frac{1}{2}}\CK\CG_1\notag\\&=\si'  K\CG_2+\si' (1-\chi_{M})\mu^{-\frac{1}{2}}\CK\CG_1+\CF_2,\label{N-1G2}
\end{align}
where $\si'\in[0,\si_\ast].$
Notice that we still have $(N-1)\si_\ast<1$ in the above system and as in \eqref{Cg2-l2z} we may let $0<\al\ll 1-(N-1)\si_\ast$. Then as Step 2, we can further verify that \eqref{N-1G1} and \eqref{N-1G2} possess a unique solution $[\CG_1,\CG_2]\in\bar{\FX}_{m}$ for $\si'=\si_\ast.$  Thus $\mathscr{L}^{-1}_1$ is also well-defined. We emphasize that
the solution we constructed here satisfies
\begin{align}
\eps\CG_1&-\beta'\na_v\cdot (v\CG_1)-\al \na_v\cdot(Av\CG_1)
+\nu\CG_1+\frac{\beta'}{2}|v|^2\sqrt{\mu}\CG_2
+\al\frac{v\cdot (Av)}{2}\sqrt{\mu}\CG_2\notag\\
&-\bet''(\CG)\na_v\cdot(v\mu)-\chi_{M}\CK \CG_{1}=\CF_1,\notag
\end{align}
and
\begin{align}
\eps\CG_2&-\beta'\na_v\cdot (v\CG_2)-\al \na_v\cdot(Av\CG_2)
+\nu\CG_2-K\CG_2- (1-\chi_{M})\mu^{-\frac{1}{2}}\CK\CG_1=\CF_2,\notag
\end{align}
which actually implies $[\CG_1,\CG_2]\in \FX_{m}$ if $[\CF_1,\CF_2]\in \FX_{m}$. Therefore, by performing the similar calculation as in the next subsection, we can still show the uniform bound as \eqref{Lif.es2}. This ends the proof of Lemma \ref{ex.pals}.
\end{proof}

\subsection{The remainder}

We are ready to complete the
\begin{proof}[Proof of Theorem \ref{st.sol}]
Since $G_1$ is given explicitly as \eqref{G1-exp}, to complete the proof of Theorem \ref{st.sol}, it suffices now to determine $G_R$ by proving the existence of the coupled system \eqref{GR1} and \eqref{GR2} under the constraint
\begin{align}
\lag G_{R,1},[1,v_i,|v|^2]\rag+\lag G_{R,2},[1,v_i,|v|^2]\mu^{\frac{1}{2}}\rag=0,\ i=1,2,3.\label{GR12-con}
\end{align}
To do this, let us first go back to the approximation system \eqref{gr12-ls}.
By applying Lemma \ref{ex.pals}, for fixed $\eps>0$, we see that $[G_{R,1}^{n+1},G_{R,2}^{n+1}]$ is well defined once $[G_{R,1}^{n},G_{R,2}^{n}]$ is given and belongs to $\bar{\FX}_{m}$ for any $m\geq0$. Furthermore, if $[G_{R,1}^{n},G_{R,2}^{n}]$ satisfies \eqref{GR12-con},
so does $[G_{R,1}^{n+1},G_{R,2}^{n+1}]$.
We now verify that $\{G_{R,1}^{n},G_{R,2}^{n}\}_{n=0}^\infty$ is a Cauchy sequence in $\FX_{m-1}$ with $m\geq1$, hence it is convergent and the limit denoted by $[G^{\eps}_{R,1},G^{\eps}_{R,2}]$ is the unique solution of the following system
\begin{align}
\eps G^\eps_{R,1}&-\beta^\eps\na_v\cdot(vG^\eps_{R,1})-\al\na_v\cdot(AvG^\eps_{R,1})
+\nu G^\eps_{R,1}\notag\\=&\chi_M\CK G^\eps_{R,1}-\frac{\beta^\eps}{2}|v|^2\sqrt{\mu}G^\eps_{R,2}-\frac{\al}{2}v\cdot(Av)\sqrt{\mu}G^\eps_{R,2}
+\frac{\beta^\eps}{\al}\na_v\cdot(v\sqrt{\mu}G_1)+\beta^\eps_1\na_v\cdot(v\mu)\notag\\&+\na_v\cdot(Av\sqrt{\mu}G_1)
+Q(\sqrt{\mu}G_1,\sqrt{\mu}G_1)+\al\{Q(\sqrt{\mu}G_1,\sqrt{\mu}G^\eps_R)
+Q(\sqrt{\mu}G^\eps_R,\sqrt{\mu}G_1)\}\notag\\&+\al^2Q(\sqrt{\mu}G^\eps_R,\sqrt{\mu}G^\eps_R)
:=\CN_\eps,\notag%\label{lmeGr1}
\end{align}
and
\begin{align}%\label{lmeGr2}
\eps G^\eps_{R,2}-&\beta^\eps\na_v\cdot(vG^\eps_{R,2})-\al \na_v\cdot(Av G^\eps_{R,2})
+LG^\eps_{R,2}=\mu^{-\frac{1}{2}}(1-\chi_M)\CK G^\eps_{R,1},\notag
\end{align}
where
$$
\beta^\eps=\al\beta_0+\al^2\beta^\eps_1,\
\beta^\eps_1=\frac{1}{3}\int_{\R^3} G_1LG_1\,dv
-\frac{\al}{3}\int_{\R^3} \FP_1\{ v\cdot (Av)\sqrt{\mu}\}G_R^{\eps}\,dv,
$$
and
$$
\sqrt{\mu}G_R^{\eps}=G^\eps_{R,1}+\sqrt{\mu}G^\eps_{R,2}.
$$
The key point here is that we can prove that the convergence of the sequence $\{G_{R,1}^{n},G_{R,2}^{n}\}_{n=0}^\infty$ is independent of $\eps$.
To see this, we first show the following uniform bound
\begin{align}\label{umbd}
\|[G^{n}_{R,1},G^{n}_{R,2}]\|_{\FX_{m}}\leq 2\CC_0,
\end{align}
where $\CC_0>0$ is independent of $\eps$, $n$ and $\al$. %for all $n\geq0$.
We give the proof by induction on $n\geq0$. Notice that $[G^{0}_{R,1},G^{0}_{R,2}]=[0,0]$, if $n=0$ the system \eqref{gr12-ls} reads
\begin{eqnarray}\label{gr12-0}
\left\{\begin{array}{rll}
\begin{split}
\eps G^{1}_{R,1}&-\beta^{0}\na_v\cdot (vG^{1}_{R,1})-\al \na_v\cdot(AvG^{1}_{R,1})
+\nu G^{1}_{R,1}
-\chi_{M}\CK G^{1}_{R,1}\\&+\frac{\beta^0}{2}|v|^2\mu^{\frac{1}{2}}G^{1}_{R,2}
+\frac{\al}{2} v\cdot(Av)\mu^{\frac{1}{2}}G^{1}_{R,2}
-\left(\beta_1^{1}-\frac{1}{3}\lag G_1,LG_1\rag \right)\na_v\cdot(v\mu)
\\=&\frac{1}{3}\lag G_1,LG_1\rag \na_v\cdot(v\mu)
+\frac{\beta^0}{\al}\na_v\cdot(v\mu^{\frac{1}{2}}G_1)+\na_v\cdot(Av\sqrt{\mu}G_1)+
Q(\mu^{\frac{1}{2}}G_1,\mu^{\frac{1}{2}}G_1),\\
\eps G^{1}_{R,2}&-\beta^0\na_v\cdot (v G^{1}_{R,2})-\al \na_v\cdot(AvG^{1}_{R,2})
+LG^{1}_{R,2}-(1-\chi_{M})\mu^{-\frac{1}{2}}\CK G^{1}_{R,1}=0,
\end{split}
\end{array}\right.
\end{eqnarray}
where $\beta^0$ and $\beta_1^1$ are defined as \eqref{bt1n}.
Performing the similar calculation as for obtaining \eqref{lifn}, one has
\begin{align}\label{GR1-1}
\sum\limits_{0\leq k\leq m}\|w_l\na_{v}^kG_{R,1}^1\|_{L^\infty}
\leq& C\al\sum\limits_{0\leq k\leq m}\|w_l\na_{v}^kG_{R,2}^1\|_{L^\infty}
+C,
\end{align}
and
\begin{align}\label{GR2-1}
\sum\limits_{0\leq k\leq m}\|w_l\na_{v}^kG_{R,2}^1\|_{L^\infty}\leq& C\sum\limits_{0\leq k\leq m}\|\na_{v}^kG^1_{R,2}\|
+C\sum\limits_{0\leq k\leq m}\|w_l\na_{v}^kG_{R,1}^{1}\|_{L^\infty},
\end{align}
where the constant $C>0$ is independent of $\eps$.

We now turn to deduce the $H^k$ estimate on $G^1_{R,2}$. To obtain the desired estimate which is independent of $\eps$, the conservation law \eqref{GR12-con}
plays a crucial role. As a matter of fact, by the iteration scheme \eqref{gr12-0}, it is not difficulty to check that
\begin{align}
\lag G^1_{R,1},[1,v_i,|v|^2]\rag+\lag G^1_{R,2},[1,v_i,|v|^2]\mu^{\frac{1}{2}}\rag=0,\ i=1,2,3,\label{GR12-1-con}
\end{align}
for $\eps>0.$ We emphasize that \eqref{GR12-1-con} may not be true in the framework of \eqref{pals} with $0<\ga\leq1$ and $\si\neq1$.

Next, we denote for $n\geq1$
\begin{align}\label{abc.def}
\bar{\FP}_0G^n_{R,1}=(a^n_{1}+\Fb^n_{1}\cdot v+c^n_{1}(|v|^2-3))\mu,\ \FP_0G^n_{R,2}=(a^n_{2}+\Fb^n_{2}\cdot v+c^n_{2}(|v|^2-3))\sqrt{\mu}.
\end{align}
Here and in the sequel, we use the notation
\begin{align}
\Fb^n_{i}=[b^n_{i,1},b^n_{i,2},b^n_{i,3}],\ i=1,2.\notag
\end{align}
From \eqref{GR12-1-con}, one has
\begin{align}%\label{m-m}
a^1_1+a^1_2=0,\ \Fb^1_1+\Fb^1_2=0,\ c^1_1+c^1_2=0.\notag
\end{align}
Consequently, it follows
\begin{align}
\|\FP_0G^1_{R,2}\|\lesssim |[a^1_2,\Fb^1_2,c^1_2]|\leq|[a^1_1,\Fb^1_1,c^1_1]|\lesssim\|w_lG^1_{R,1}\|_{L^\infty},\label{P0Gr2-1}
\end{align}
for $l>5/2.$
On the other hand, for the microscopic component of $G^1_{R,2}$, we get from the inner product
$\lag \na_v^k\eqref{gr12-0}_2,\na_v^k\FP_1G^1_{R,2}\rag$ with $k\geq0$ that
\begin{align}
\eps\lag&\na_v^k(\FP_1G^1_{R,2}+\FP_0G^1_{R,2}),\na_v^k\FP_1G^1_{R,2}\rag
-\beta^0\lag\na_v^k\na_v\cdot (v\FP_1G^1_{R,2}),\na_v^k\FP_1G^1_{R,2}\rag
\notag\\&-\beta^0\lag\na_v^k\na_v\cdot (v\FP_0G^1_{R,2}),\na_v^k\FP_1G^1_{R,2}\rag
-\al \lag\na_v^k\na_v\cdot(Av\FP_1G^1_{R,2}),\na_v^k\FP_1G^1_{R,2}\rag
\notag\\&-\al \lag\na_v^k\na_v\cdot(Av\FP_0G^1_{R,2}),\na_v^k\FP_1G^1_{R,2}\rag
+\lag\na_v^k L\FP_1G^1_{R,2},\na_v^k\FP_1G^1_{R,2}\rag
\notag\\&-\lag\na_v^k[(1-\chi_{M})\mu^{-\frac{1}{2}}\CK G^1_{R,1}],\na_v^k\FP_1G^1_{R,2}\rag=0.\notag
\end{align}
Using Lemma \ref{es-L} and Cauchy-Schwarz's inequality as well as \eqref{P0Gr2-1}, one gets
\begin{equation}\label{P1Gr2-1}
(\eps+\de_0)\|\na_v^k\FP_1G^1_{R,2}\|^2\leq
C(\eps+\al)\|w_lG^1_{R,1}\|^2_{L^\infty}+C\sum\limits_{k'\leq k}\|w_l\na_v^{k'}G^1_{R,1}\|^2_{L^\infty}
+C{\bf1}_{k>0}\|\FP_1G^1_{R,2}\|^2.
\end{equation}
Taking a linear combination of \eqref{P1Gr2-1} with respect to $k=0,1,\cdots,m$ and applying \eqref{P0Gr2-1}, we arrive at
\begin{align}\label{Gr2-1-1}
\|\FP_0G^1_{R,2}\|^2+\sum\limits_{k\leq m}\|\na_v^k\FP_1G^1_{R,2}\|^2\leq
C\sum\limits_{k\leq m}\|w_l\na_v^{k}G^1_{R,1}\|^2_{L^\infty}.
\end{align}
Therefore, by plugging this into \eqref{GR2-1} and using \eqref{GR1-1}, we finally obtain
\begin{align}%\label{GR12-1}
\sum\limits_{0\leq k\leq m}\|w_l\na_{v}^kG_{R,1}^1\|_{L^\infty}+\sum\limits_{0\leq k\leq m}\|w_l\na_{v}^kG_{R,2}^1\|_{L^\infty}
\leq\CC_0,\notag
\end{align}
for some suitably large $\CC_0>0$. This implies that \eqref{umbd} is true for $n=1$.

We now assume that \eqref{umbd} is valid for $n=N$ and then prove that \eqref{umbd}  holds for $n=N+1$.
In fact, applying the estimates \eqref{lifn} to the system \eqref{gr12-ls} with $n=N$, one has
\begin{align}\label{GRk1}
\sum\limits_{0\leq k\leq m}\|w_l\na_v^kG_{R,1}^{N+1}\|_{L^\infty}\leq C\al\sum\limits_{0\leq k\leq m}\|w_l\na_{v}^kG_{R,2}^{N+1}\|_{L^\infty}
+C\sum\limits_{0\leq k\leq m}\|w_{l}\na_{v}^k\CS^N\|_{L^\infty},
\end{align}
and
\begin{align}\label{GRk2}
\sum\limits_{0\leq k\leq m}\|w_l\na_{v}^kG_{R,2}^{N+1}\|_{L^\infty}\leq& C\sum\limits_{0\leq k\leq m}\|\na_{v}^kG_{R,2}^{N+1}\|
+C\sum\limits_{0\leq k\leq m}\|w_l\na_{v}^kG_{R,1}^{N+1}\|_{L^\infty},
\end{align}
where
\begin{align}
\CS^N=&\frac{1}{3}\lag G_1,LG_1\rag \na_v\cdot(v\mu)+\frac{\beta^N}{\al}\na_v\cdot(v\mu^{\frac{1}{2}}G_1)+\na_v\cdot(Av\sqrt{\mu}G_1)+
Q(\mu^{\frac{1}{2}}G_1,\mu^{\frac{1}{2}}G_1)\notag\\&+\al\{Q(\mu^{\frac{1}{2}}G^{N}_R,\mu^{\frac{1}{2}}G_1)
+Q(\mu^{\frac{1}{2}}G_1,\mu^{\frac{1}{2}}G^{N}_R)\}+\al^2 Q(\mu^{\frac{1}{2}}G^N_R,\mu^{\frac{1}{2}}G^N_R).\notag
\end{align}
Recall \eqref{G1-exp}. By employing Lemma \ref{op.es.lem} and the induction hypothesis, one has
\begin{align}
\sum\limits_{0\leq k\leq m}\|w_{l}\na_{v}^k\CS^N\|_{L^\infty}\leq C+C\al\CC_0+C\al^2\CC_0^2.\label{SN}
\end{align}
On the other hand, since $[G_{R,1}^{N},G_{R,2}^{N}]\in\FX_m$, from \eqref{gr12-ls}, it also follows
\begin{align}\notag
\lag G^{N+1}_{R,1},[1,v_i,|v|^2]\rag+\lag G^{N+1}_{R,2},[1,v_i,|v|^2]\mu^{\frac{1}{2}}\rag=0,\ i=1,2,3,%\label{GR12-1-con}
\end{align}
for $\eps>0.$ Based on this, as the estimate \eqref{Gr2-1-1}, one has
\begin{align}\label{Gr2-N-l2}
\|\FP_0G^{N+1}_{R,2}\|^2+\sum\limits_{k\leq m}\|\na_v^k\FP_1G^{N+1}_{R,2}\|^2\leq
C\sum\limits_{k\leq m}\|w_l\na_v^{k}G^{N+1}_{R,1}\|^2_{L^\infty}.
\end{align}
Substituting \eqref{SN} and \eqref{Gr2-N-l2} into \eqref{GRk1} and \eqref{GRk2}, we get
\begin{align}\label{GR12-N1}
\sum\limits_{0\leq k\leq m}\|w_l\na_v^kG_{R,1}^{N+1}\|_{L^\infty}+
\sum\limits_{0\leq k\leq m}\|w_l\na_v^kG_{R,2}^{N+1}\|_{L^\infty}
\leq C_0+C\al\CC_0+C\al^2\CC_0^2\leq 2\CC_0.
\end{align}
Hence \eqref{umbd} is valid for all $n\geq0.$

Having disposed of the above preliminary step, we now turn to prove that $[G_{R,1}^{n},G_{R,2}^{n}]|_{n=1}^{\infty}$
is a Cauchy sequence in the larger function space $\FX_{m-1}.$ For this purpose, we first denote
$$
[\tilde{G}_{R,1}^{n},\tilde{G}_{R,2}^{n}]=[G_{R,1}^{n}-G_{R,1}^{n-1},G_{R,2}^{n}-G_{R,2}^{n-1}],
\ \tilde{\beta}^n=\beta^n-\beta^{n-1},\ n\geq1,
$$
then by \eqref{gr12-ls}, we see that the triple $[\tilde{G}_{R,1}^{n},\tilde{G}_{R,2}^{n},\tilde{\beta}^n]$ satisfies
\begin{eqnarray}%\label{tgr12-ls}
\left\{\begin{array}{rll}
\begin{split}
\eps \tilde{G}^{n+1}_{R,1}&-\beta^{n}\na_v\cdot (v\tilde{G}^{n+1}_{R,1})-\al \na_v\cdot(Av\tilde{G}^{n+1}_{R,1})
+\nu \tilde{G}^{n+1}_{R,1}
-\chi_{M}\CK \tilde{G}^{n+1}_{R,1}\\&+\frac{\beta^n}{2}|v|^2\mu^{\frac{1}{2}}\tilde{G}^{n+1}_{R,2}+\frac{\al}{2} v\cdot(Av)\mu^{\frac{1}{2}}\tilde{G}^{n+1}_{R,2}-\tilde{\beta}_1^{n+1}\na_v\cdot(v\mu)
\\=&\tilde{\beta}^{n}\na_v\cdot (vG^{n}_{R,1})-\frac{\tilde{\beta}^n}{2}|v|^2\mu^{\frac{1}{2}}G^{n}_{R,2}
+\frac{\tilde{\beta}^n}{\al}\na_v\cdot(v\mu^{\frac{1}{2}}G_1)
\notag\\&+\al\{Q(\mu^{\frac{1}{2}}\tilde{G}^{n}_R,\mu^{\frac{1}{2}}G_1)
+Q(\mu^{\frac{1}{2}}G_1,\mu^{\frac{1}{2}}\tilde{G}^{n}_R)\}+\al^2 Q(\mu^{\frac{1}{2}}\tilde{G}^n_R,\mu^{\frac{1}{2}}\tilde{G}^n_R)
\notag\\&+\al^2 Q(\mu^{\frac{1}{2}}\tilde{G}^n_R,\mu^{\frac{1}{2}}G^n_R){+\al^2 Q(\mu^{\frac{1}{2}}G^n_R,\mu^{\frac{1}{2}}\bar{G}^n_R)},\\
\eps \tilde{G}^{n+1}_{R,2}&-\beta^n\na_v\cdot (v \tilde{G}^{n+1}_{R,2})-\al \na_v\cdot(Av\tilde{G}^{n+1}_{R,2})
+L\tilde{G}^{n+1}_{R,2}\notag\\&-(1-\chi_{M})\mu^{-\frac{1}{2}}\CK \tilde{G}^{n+1}_{R,1}=\tilde{\beta}^n\na_v\cdot (v G^{n}_{R,2}).
\end{split}
\end{array}\right.
\end{eqnarray}
Note that from \eqref{bt1n} one actually has $\tilde{\beta}^n=\al^2\tilde{\beta}_1^n=\al^2(\beta_1^n-\beta_1^{n-1}).$
Since both $[G_{R,1}^{n},G_{R,2}^{n}]$ and $[G_{R,1}^{n+1},G_{R,2}^{n+1}]$ satisfy \eqref{GR12-con}, so does their difference $[\tilde{G}^{n+1}_{R,1},\tilde{G}^{n+1}_{R,2}]$. With this,
we can proceed analogously to the deduction of \eqref{GR12-N1} to obtain that
\begin{align}
\sum\limits_{k\leq m-1}\{\|w_l\na_v^k\tilde{G}^{n+1}_{R,1}\|_{L^\infty}+\|w_l\na_v^k\tilde{G}^{n+1}_{R,2}\|_{L^\infty}\}\leq& C\al|\tilde{\beta}_1^{n}|+C\al\sum\limits_{k\leq m-1}\{\|w_l\tilde{G}^{n}_{R,1}\|_{L^\infty}+\|w_l\tilde{G}^{n}_{R,2}\|_{L^\infty}\}
\notag\\&+C\al\sum\limits_{k\leq m-1}\{\|w_l\tilde{G}^{n}_{R,1}\|_{L^\infty}+\|w_l\tilde{G}^{n}_{R,2}\|_{L^\infty}\}^2,\notag
\end{align}
which is equivalent to
\begin{align}\label{Gr12-ca}
\|[\tilde{G}^{n+1}_{R,1},\tilde{G}^{n+1}_{R,2}]\|_{\FX_{m-1}}\leq C\al\|[\tilde{G}^{n}_{R,1},\tilde{G}^{n}_{R,2}]\|_{\FX_{m-1}}.
\end{align}
Therefore $[G_{R,1}^{n},G_{R,2}^{n}]$ converges strongly to some function pair $[G_{R,1}^{\eps},G_{R,2}^{\eps}]\in\FX_{m-1}$. Moreover, from \eqref{umbd}, it also follows
\begin{align}\label{eps-umbd}
\|[G^\eps_{R,1},G^\eps_{R,2}]\|_{\FX_{m}}\leq 2\CC_0.
\end{align}
We shall have established the theorem if we prove that $[G_{R,1}^{\eps},G_{R,2}^{\eps}]\rightarrow [G_{R,1},G_{R,2}]$ as $\eps\rightarrow0^+.$
For this,  we
choose a positive sequence $\{\eps_n\}_{n=1}^\infty$ such that $|\eps_{n+1}-\eps_{n}|\leq 2^{-n}$,
then $\eps_n\rightarrow0^+$ as $n\to+\infty$. We consider the following approximation equations
\begin{align}
\eps_{n} G^{\eps_n}_{R,1}&-\beta^{\eps_n}\na_v\cdot(vG^{\eps_n}_{R,1})-\al\na_v\cdot(AvG^{\eps_n}_{R,1})
+\nu G^{\eps_n}_{R,1}\notag\\=&\chi_M\CK G^{\eps_n}_{R,1}-\frac{\beta^{\eps_n}}{2}|v|^2\sqrt{\mu}G^{\eps_n}_{R,2}-\frac{\al}{2}v\cdot(Av)\sqrt{\mu}G^{\eps_n}_{R,2}
+\frac{\beta^{\eps_n}}{\al}\na_v\cdot(v\sqrt{\mu}G_1)+\beta^{\eps_n}_1\na_v\cdot(v\mu)\notag\\&+\na_v\cdot(Av\sqrt{\mu}G_1)
+Q(\sqrt{\mu}G_1,\sqrt{\mu}G_1)+\al\{Q(\sqrt{\mu}G_1,\sqrt{\mu}G^\eps_R)
+Q(\sqrt{\mu}G^{\eps_n}_R,\sqrt{\mu}G_1)\}\notag\\&+\al^2Q(\sqrt{\mu}G^{\eps_n}_R,\sqrt{\mu}G^{\eps_n}_R),\notag%\label{lmeGr1-n}
\end{align}
and
\begin{align}%\label{lmeGr2-n}
\eps_n G^{\eps_n}_{R,2}-&\beta^{\eps_n}\na_v\cdot(vG^{\eps_n}_{R,2})-\al \na_v\cdot(Av G^{\eps_n}_{R,2})
+LG^{\eps_n}_{R,2}=\mu^{-\frac{1}{2}}(1-\chi_M)\CK G^{\eps_n}_{R,1}.\notag
\end{align}
Since each pair $[G^{\eps_n}_{R,1},G^{\eps_n}_{R,2}]$ is well-defined and satisfies \eqref{eps-umbd}, we have as the estimate \eqref{Gr12-ca} that
$$
\|[G^{\eps_n}_{R,1}-G^{\eps_{n-1}}_{R,1},G^{\eps_n}_{R,2}-G^{\eps_{n-1}}_{R,2}]\|_{\FX_{m-1}}\leq C|\eps_n-\eps_{n-1}|,\ n\geq1.
$$
Thus $[G_{R,1}^{\eps_n},G_{R,2}^{\eps_n}]\rightarrow [G_{R,1},G_{R,2}]$ as $\eps_n\rightarrow0^+$. Moreover it holds that
$[G_{R,1},G_{R,2}]\in\FX_m$ satisfies the same estimate as \eqref{eps-umbd}. This proves \eqref{st.sole2}. The non-negativity of the steady solution $G_{st}=\mu+\al\sqrt{\mu}(G_1+\al G_R)$ constructed here is a direct subsequence of the dynamical stability
of $G_{st}(v)$ verified in Theorem \ref{ge.th}. This ends the proof of Theorem \ref{st.sol}.
\end{proof}

\section{Unsteady problem}\label{ust-pro}
In this section, we turn to the time-dependent case. Our goal is to prove that the large time behavior of the Cauchy problem \eqref{G-ust} and \eqref{G-id} can be governed by the steady problem \eqref{G-st} which has been solved in Section \ref{st-pro}. The proof is based on the local-in-time existence and the {\it a priori} estimate as well as the continuum argument.

%the solution $G$ should depend on the parameter $\al$, in this paper, we omit the dependence for brevity.

%\subsection{Local existence}
The local-in-time existence of the Cauchy problem \eqref{G-ust} and \eqref{G-id} will be established by an iteration method and Duhamel's principle.
Set $G=G_{st}+\sqrt{\mu}f$, then we see that $f$ satisfies
\begin{align}\label{f-eq}
\pa_t f&+v\cdot\na_x f-\beta\mu^{-1/2}\na_v\cdot(v\sqrt{\mu}f)-\al\mu^{-1/2}\na_v\cdot(Av\sqrt{\mu}f)+Lf
\notag\\&=\Ga(f,f)+\al\{\Ga(G_1+\al G_R,f)+\Ga(f,G_1+\al G_R)\},\ t>0, \ x\in\T^3,\ v\in\R^3,
\end{align}
with
\begin{align}\label{f-id}
\sqrt{\mu}f(0,x,v)\eqdef f_{0}(x,v)=G_0(x,v)-G_{st}(v),\ x\in\T^3,\ v\in\R^3.
\end{align}
As it is pointed out in Section \ref{st-pro}, to eliminate the severe velocity growth in the left hand side of \eqref{f-eq}, it is necessary to use the following Caflisch's decomposition
$$
\sqrt{\mu}f=f_1+\sqrt{\mu}f_2,
$$
where $f_1$ and $f_2$ satisfy
\begin{align}\label{f1-eq}
\pa_t f_1&+v\cdot\na_x f_1-\beta\na_v\cdot(vf_1)-\al\na_v\cdot(Avf_1)+\nu f_1
\notag\\=&\chi_M\CK f_1-\frac{\beta}{2} |v|^2\sqrt{\mu}f_2-\frac{\al}{2} v\cdot(Av)\sqrt{\mu}f_2
+{\al Av\cdot(\na_v\sqrt{\mu})(|v|^2-3)\sqrt{\mu}c_{f_2}}
\notag\\&+Q(f_1,f_1)+Q(f_1,\sqrt{\mu}f_2)+Q(\sqrt{\mu}f_2,f_1)\notag\\&+\al\{Q(\sqrt{\mu}(G_1+\al G_R),\sqrt{\mu}f)
+Q(\sqrt{\mu}f,\sqrt{\mu}(G_1+\al G_R))\},\ t>0,\ x\in\T^3,\ v\in\R^3,
\end{align}
\begin{align}\label{f1-id}
f_1(0,x,v)=f_0(x,v)=G_0(x,v)-G_{st}(v),\ x\in\T^3,\ v\in\R^3,
\end{align}
\begin{align}\label{f2-eq}
\pa_t f_2&+v\cdot\na_x f_2-\beta\na_v\cdot(vf_2)-\al\na_v\cdot(Av f_2)
{+\al Av\cdot(\na_v\sqrt{\mu})(|v|^2-3)\sqrt{\mu}c_{f_2}}
\notag\\&+Lf_2
=(1-\chi_M)\mu^{-\frac{1}{2}}\CK f_1+\Ga(f_2,f_2),\ t>0,\ x\in\T^3,\ v\in\R^3,
\end{align}
and
\begin{align}\label{f2-id}
f_2(0,x,v)=0,\ x\in\T^3,\ v\in\R^3,
\end{align}
respectively. Here, $c_{f_2}$ is defined as
$$
\FP_0f_2=\{a_{f_2}(t,x)+\Fb_{f_2}(t,x)\cdot v+c_{f_2}(t,x)(|v|^2-3)\}\sqrt{\mu}.
$$
To determine $f$, we instead turn to solve $f_1$ and $f_2$ through the above system.

We shall look for solutions of \eqref{f1-eq}, \eqref{f1-id}, \eqref{f2-eq} and \eqref{f2-id} in the following function space
\begin{equation*}%\label{fsp}
\begin{split}
\FY^\al_{N,T}=&\Big\{(\CG_1,\CG_2)
\bigg|\sup\limits_{0\leq t\leq T}\sum\limits_{|\ze|+|\vth|\leq N}\left\{\|w_{l}\pa_\zeta^\vartheta\CG_1(t)\|_{L^\infty}
+\al\|w_{l}\pa_\zeta^\vartheta\CG_2(t)\|_{L^\infty}\right\}<+\infty%\\
%&\qquad\qquad\qquad(\CG_1,1)+(\CG_2,\mu^{\frac{1}{2}})=0,\ (\CG_1,v_i)+(\CG_2,v_i\mu^{\frac{1}{2}})=0,\ i=1,2,3
\Big\},
\end{split}
\end{equation*}
associated with the norm
$$
\|[\CG_1,\CG_2]\|_{\FY^\al_{N,T}}=\sup\limits_{0\leq t\leq T}\sum\limits_{|\ze|+|\vth|\leq N}
\left\{\|w_{l}\pa_\zeta^\vartheta\CG_1(t)\|_{L^\infty}+\al\|w_{l}\pa_\zeta^\vartheta\CG_2(t)\|_{L^\infty}\right\}.
$$
We then have the following result on local-in-time existence. For brevity, we omit its proof, cf.~\cite{DL-2020}.

\begin{theorem}[Local existence]\label{loc.th}
Under the conditions stated in Theorem \ref{ge.th}, there exits  $T_\ast>0$ which may depend on $\al$ such that the coupling problem \eqref{f1-eq}, \eqref{f1-id}, \eqref{f2-eq} and \eqref{f2-id} admits a unique local in time solution $[f_1(t,x,v),f_2(t,x,v)]$ satisfying
\begin{align*}
\|[f_1,f_2]\|_{\FY^\al_{N,T_\ast}}\leq C_0\al^2,
\end{align*}
for a constant $C_0>0$ independent of $\al$.
\end{theorem}

In what follows we focus on deducing the {\it a priori} $W^{N,\infty}$ estimates on the solution constructed in Theorem \ref{loc.th}. Namely, we assume that $[f_1,f_2]$ is a classical solution to the initial value problem \eqref{f1-eq}, \eqref{f1-id}, \eqref{f2-eq} and \eqref{f2-id}. The purpose is to prove
\begin{align}\label{ap-es}
\sup\limits_{0\leq s\leq t}\sum\limits_{|\ze|+|\vth|\leq N}e^{\la_0s}\|w_l\pa_\ze^{\vth}f_1(s)\|_{L^\infty}
&+\al\sup\limits_{0\leq s\leq t}\sum\limits_{|\ze|+|\vth|\leq N}e^{\la_0s}\|w_l\pa_\ze^{\vth}f_2(s)\|_{L^\infty}\notag\\
\leq& C\sum\limits_{|\ze|+|\vth|\leq N}\|w_l\pa_\ze^{\vth}f_0\|_{L^\infty},
\end{align}
for any $t\geq0$ and some constant $C>0$, under the {\it a priori} assumption that
\begin{align}\label{ap-as}
\sup\limits_{0\leq s\leq t}&\sum\limits_{|\ze|+|\vth|\leq N}e^{\la_0s}\|w_l\pa_\ze^{\vth}f_1(s)\|_{L^\infty}
+\al\sup\limits_{0\leq s\leq t}\sum\limits_{|\ze|+|\vth|\leq N}e^{\la_0s}\|w_l\pa_\ze^{\vth}f_2(s)\|_{L^\infty}\leq \al^2,
\end{align}
where $\la_0>0$ is independent of $\al$ to be determined later. Note that the initial condition \eqref{th.ids.p} is the consequence of \eqref{ap-as}.  The  {\it a priori} estimate together with the local existence established in Theorem \ref{loc.th}
and the continuum argument
enables us to construct the global existence for the Cauchy problem \eqref{f-eq} and \eqref{f-id}. Thus we are ready to complete the

\begin{proof}[Proof of Theorem \ref{ge.th}]

We first verify that \eqref{ap-as} holds true under the {\it a priori} assumption \eqref{ap-es}.
The proof is divided into two steps.

\noindent\underline{{\it Step 1. $W^{k,\infty}$ estimates.}}
Denoting
$$[g_{1},g_{2}](t)=e^{\la_0t}[f_1,f_2](t),$$
and defining
$$
\FP_0g_2=\{a_2(t,x)+\Fb_2(t,x)\cdot v+c_2(t,x)(|v|^2-3)\}\sqrt{\mu},
$$
one has by \eqref{f1-eq}, \eqref{f1-id}, \eqref{f2-eq} and \eqref{f2-id} that
\begin{align}\label{g1k.eq}
\pa_t [w_l\pa_\ze^{\vth}g_{1}]&+v\cdot\na_x[w_l\pa_\ze^{\vth}g_{1}]-\beta v\cdot\na_v [w_l\pa_\ze^{\vth}g_{1}]+2l\beta\frac{|v|^2}{1+|v|^2}w_l\pa_\ze^{\vth}g_{1}-3\beta w_l\pa_\ze^{\vth}g_{1}
\notag\\&-\al Av\cdot\na_v[w_l\pa_\ze^{\vth}g_{1}]+2l\al\frac{v\cdot Av}{1+|v|^2}w_l\pa_\ze^{\vth}g_{1}
-\al\textrm{tr}A w_l\pa_\ze^{\vth}g_{1}-\la_0w_l\pa_\ze^{\vth}g_{1}+\nu w_l\pa_\ze^{\vth}g_{1}
\notag\\=&-{\bf 1}_{|\ze|>0}\sum\limits_{|\ze'|=1}C_\ze^{\ze'}w_l\pa_{\ze'}v\cdot\na_x\pa^\vth_{\ze-\ze'}g_{1}
+\beta {\bf 1}_{|\ze|>0}\sum\limits_{|\ze'|=1}C_\ze^{\ze'}w_l\pa_{\ze'}v\cdot\na_v\pa^\vth_{\ze-\ze'}g_{1}
\notag\\&+\al{\bf 1}_{|\ze|>0}\sum\limits_{|\ze'|=1}C_\ze^{\ze'}w_l\pa_{\ze'}(Av)\cdot\na_v\pa^\vth_{\ze-\ze'}g_{1}
\notag\\&-{\bf 1}_{|\ze|>0}w_l\sum\limits_{0< \ze'\leq \ze}C_{\ze}^{\ze'}w_l\pa_{\ze-\ze'}\nu\pa^{\vth}_{\ze-\ze'}g_{1}
+w_l\pa_\ze^\vth(\chi_M\CK g_{1,0})-\frac{\beta}{2}w_l\pa_\ze^\vth(|v|^2\sqrt{\mu}g_{2})\notag\\&
-\frac{\al}{2} w_l\pa_\ze^\vth(v\cdot(Av)\sqrt{\mu}g_{2})
+{\al w_l\pa_\ze^\vth\left( Av\cdot(\na_v\sqrt{\mu})(|v|^2-3)\sqrt{\mu}c_{2}\right)}
\notag\\&+e^{\la_0t}w_l\pa_\ze^\vth \{Q(f_1,f_1)+Q(f_1,\sqrt{\mu}f_2)+Q(\sqrt{\mu}f_2,f_1)\}\notag\\&+\al w_le^{\la_0t}\pa_\ze^\vth\{Q(\sqrt{\mu}(G_1+\al G_R),\sqrt{\mu}f)
+Q(\sqrt{\mu}f,\sqrt{\mu}(G_1+\al G_R))\},
\end{align}
\begin{align}%\label{g1k-id}
\pa_\ze^\vth g_{1}(0,x,v)=\pa_\ze^\vth f_0(x,v),\notag
\end{align}
\begin{align}\label{g2k.eq}
\pa_t [w_l\pa_\ze^{\vth}g_{2}]&+v\cdot\na_x[w_l\pa_\ze^{\vth}g_{2}]-\beta v\cdot\na_v[w_l\pa_\ze^{\vth}g_{2}]-2l\beta\frac{|v|^2}{1+|v|^2}w_l\pa_\ze^{\vth}g_{2}-3\beta w_l\pa_\ze^{\vth}g_{2}
\notag\\&-\al Av\cdot\na_v[w_l\pa_\ze^{\vth}g_{2}]+2l\al\frac{v\cdot Av}{1+|v|^2}w_l\pa_\ze^{\vth}g_{2}
-\al\textrm{tr}Aw_l\pa_\ze^{\vth}g_{2}-\la_0w_l\pa_\ze^{\vth}g_{2}+\nu w_l\pa_\ze^{\vth}g_{2}\notag\\
=&-{\bf 1}_{|\ze|>0}\sum\limits_{|\ze'|=1}C_\ze^{\ze'}w_l\pa_{\ze'}v\cdot\na_x\pa^\vth_{\ze-\ze'}g_{2}
+\beta {\bf 1}_{|\ze|>0}\sum\limits_{|\ze'|=1}C_\ze^{\ze'}w_l\pa_{\ze'}v\cdot\na_v\pa_{\ze-\ze'}^{\vth}g_{2}
\notag\\&+\al {\bf 1}_{\ze>0}\sum\limits_{|\ze'|=1}C_\ze^{\ze'}w_l\pa_{\ze'}(Av)\cdot\na_v\pa_{\ze-\ze'}^{\vth}g_{2}
-{\al w_l\pa_\ze^\vth\left( Av\cdot(\na_v\sqrt{\mu})(|v|^2-3)c_{2}\right)}
\notag\\&-{\bf 1}_{|\ze|>0}\sum\limits_{0<\ze'\leq \ze}C_{\ze}^{\ze'}w_l\pa_{\ze'}\nu\pa_{\ze-\ze'}^{\vth}g_{2}
+w_l\pa_\ze^\vth(K g_{2,0})+w_l\pa_\ze^\vth((1-\chi_M)\mu^{-\frac{1}{2}}\CK g_{1})\notag\\&+e^{\la_0t}w_l\pa_\ze^\vth \Ga(f_2,f_2),
\end{align}
and
\begin{align}%\label{g2k-id}
\pa_\ze^\vth g_{2}(0,x,v)=0.\notag
\end{align}
As in \eqref{H1CL}, we recall that the characteristic line of the above system can be determined by
\begin{align}\label{CH2}
\left\{\begin{array}{rll}
&\frac{d X}{ds}=V(s;t,x,v),\\[2mm]
&\frac{d V}{ds}=-\beta V(s;t,x,v)-\al A V(s;t,x,v),\\[2mm]
%&\frac{d V_i}{ds}=-\beta' V_i(s;t,v),\ i=2,3,\\[2mm]
&X(t;t,x,v)=x,\ V(t;t,x,v)=v,
\end{array}\right.
\end{align}
which gives
\begin{align}\label{CH2-p2}
\left\{\begin{array}{rll}
&V(s)=V(s;t,x,v)=e^{-(s-t)(\beta I+\al A)}v,\\[2mm]
&X(s)=X(s;t,x,v)=x-(\beta I+\al A)^{-1}\left[e^{-(s-t)(\beta I+\al A)}-I\right]v.
\end{array}\right.
\end{align}
Along the characteristic line \eqref{CH2}, we write the solution of \eqref{g1k.eq} and \eqref{g2k.eq}
as the following mild form
%Next, \Red{similar to \eqref{CF1} and \eqref{CF2}, we have by using \eqref{CH2}} that
\begin{align}
w_l\pa_\ze^{\vth}g_{1}(t,x,v)=\sum\limits_{i=1}^9\CH_{i},\label{g1k}
\end{align}
with
\begin{align}
\CH_1=e^{-\int_0^t\CA^{\la}(s)ds}w_l\pa_\ze^\vth f_0(X(0),V(0)),\notag
\end{align}
\begin{align}
\CH_2=-{\bf 1}_{|\ze|>0}\sum\limits_{|\ze'|=1}C_\ze^{\ze'}\int_0^te^{-\int_s^t\CA^\la(\tau)d\tau}
\{w_l\pa_{\ze'}v\cdot\na_x\pa^\vth_{\ze-\ze'}g_{1}\}(s,X(s),V(s))ds,\notag
\end{align}
\begin{align}
\CH_3=\beta {\bf 1}_{|\ze|>0}\sum\limits_{|\ze'|=1}C_\ze^{\ze'}\int_0^te^{-\int_s^t\CA^\la(\tau)d\tau}
\{w_l\pa_{\ze'}v\cdot\na_v\pa^\vth_{\ze-\ze'}g_{1}\}(s,X(s),V(s))ds,\notag
\end{align}
\begin{align}
\CH_4=\al{\bf 1}_{|\ze|>0}\sum\limits_{|\ze'|=1}C_\ze^{\ze'}\int_0^te^{-\int_s^t\CA^\la(\tau)d\tau}
\{w_l\pa_{\ze'}(Av)\cdot\na_v\pa_{\ze-\ze'}^\vth g_{1}\}(s,X(s),V(s))ds,\notag
\end{align}
\begin{align}
\CH_5=-{\bf 1}_{|\ze|>0}\sum\limits_{0<\ze'\leq \ze}C_\ze^{\ze'}
\int_0^te^{-\int_s^t\CA^\la(\tau)d\tau}\{w_l\pa_{\ze'}\nu\pa_{\ze-\ze'}^{\vth}g_{1}\}(s,X(s),V(s))ds
,\notag
\end{align}
\begin{align}
\CH_6=\int_0^te^{-\int_s^t\CA^\la(\tau)d\tau}\{w_l\pa_{\ze}^\vth(\chi_M\CK g_{1})\}(s,X(s),V(s))ds
,\notag
\end{align}
\begin{align}
\CH_7=-\int_0^t&e^{\int_s^t\CA^\la(\tau)d\tau}\Big\{\frac{\beta}{2}w_l\pa_{\ze}^\vth(|v|^2\sqrt{\mu}g_{2})
+\frac{\al}{2} w_l\pa_{\ze}^\vth(v\cdot(Av)\sqrt{\mu}g_{2})\notag\\&\qquad-{\al w_l\pa_\ze^\vth\left( Av\cdot(\na_v\sqrt{\mu})(|v|^2-3)\sqrt{\mu}c_{2}\right)}\Big\}(s,X(s),V(s))ds,\notag
\end{align}
\begin{align}
\CH_8=\int_0^te^{-\int_s^t\CA^\la(\tau)d\tau}e^{\la_0s}\{w_l\pa_\ze^\vth \{Q(f_1,f_1)+Q(f_1,\sqrt{\mu}f_2)+Q(\sqrt{\mu}f_2,f_1)\}\}(s,X(s),V(s))ds,\notag
\end{align}
\begin{align}
\CH_9=&\al \int_0^te^{-\int_s^t\CA^\la(\tau)d\tau}e^{\la_0s}\Big\{w_l\pa_{\ze}^\vth\{Q(\sqrt{\mu}(G_1+\al G_R),\sqrt{\mu}f)
\notag\\&\qquad\qquad+Q(\sqrt{\mu}f,\sqrt{\mu}(G_1+\al G_R))\}\Big\}(s,X(s),V(s))ds,\notag
\end{align}
and
\begin{align}
w_l\pa_\ze^\vth g_{2}(t,x,v)
=\sum\limits_{i=10}^{16}\CH_{i},\label{g2k}
\end{align}
with
\begin{align}
\CH_{10}=-{\bf 1}_{|\ze|>0}\sum\limits_{|\ze'|=1}C_\ze^{\ze'}
\int_0^te^{-\int_s^t\CA^\la(\tau)d\tau}\{w_l\pa_{\ze'}v\cdot\na_x\pa^\vth_{\ze-\ze'}g_2\}(s,X(s),V(s))ds,\notag
\end{align}
\begin{align}
\CH_{11}=&{\bf 1}_{|\ze|>0}\sum\limits_{|\ze'|=1}C_\ze^{\ze'}
\int_0^te^{-\int_s^t\CA^\la(\tau)d\tau}\Big\{\beta w_l\pa_{\ze'}v\cdot\na_v\pa_{\ze-\ze'}^{\vth}g_2
\notag\\&\qquad+\al w_l\pa_{\ze'}(Av)\cdot\na_v\pa_{\ze-\ze'}^{\vth}g_2
-{\al w_l\pa_\ze^\vth\left( Av\cdot(\na_v\sqrt{\mu})(|v|^2-3)c_{2}\right)}\Big\}(s,X(s),V(s))ds,\notag
\end{align}
\begin{align}
\CH_{12}=-{\bf 1}_{|\ze|>0}\sum\limits_{0< \zeta_1\leq \zeta}C_\ze^{\ze'}
\int_0^te^{-\int_s^t\CA^\la(\tau)d\tau}\{w_l\pa_{\ze'}\nu\pa^\vth_{\zeta-\zeta_1}g_2\}(s,X(s),V(s))ds,\notag
\end{align}
\begin{align}
\CH_{13}=\int_0^te^{-\int_s^t\CA^\la(\tau)d\tau}
\{w_lK \pa_\ze^{\vth}g_{2}\}(s,X(s),V(s))ds,\notag
\end{align}
\begin{align}
\CH_{14}={\bf 1}_{|\ze|>0}\sum\limits_{0<\ze'\leq \ze}C_\ze^{\ze'}\int_0^te^{-\int_s^t\CA^\la(\tau)d\tau}
\{w_l(\pa_{\ze'}K)(\pa_{\ze-\ze'}^\vth g_2)\}(s,X(s),V(s))ds,\notag
\end{align}
\begin{align}
\CH_{15}=\int_0^te^{-\int_s^t\CA^\la(\tau)d\tau}
\{w_l\pa_{\ze}^\vth((1-\chi_M)\mu^{-\frac{1}{2}}\CK g_1)\}(s,X(s),V(s))ds,\notag
\end{align}
\begin{align}
\CH_{16}=\int_0^te^{-\int_s^t\CA^\la(\tau)d\tau}e^{\la_0s}
\{w_l\pa_\ze^\vth \Ga(f_2,f_2)\}(s,X(s),V(s))ds.\notag
\end{align}
Here, as before we have denoted
%as previously
\begin{align}%\label{A-lb}
\CA^\la(\tau,V(\tau))=&\nu(V(\tau))-3\beta+2l \beta  \frac{|V(\tau)|^2}{1+|V(\tau)|^2} +2l \al \frac{V(\tau)\cdot (AV(\tau))}{{1+|V(\tau)|^2}}-\al{\rm tr}A-\la_0\notag\\
=&\nu(V(\tau))-3\beta_1+2l \beta  \frac{|V(\tau)|^2}{1+|V(\tau)|^2} +2l \al \frac{V(\tau)\cdot (AV(\tau))}{{1+|V(\tau)|^2}}-\la_0,\notag
\end{align}
and moreover, as long as $l\al>0$, $|l\beta|$ and $\la_0$ are suitably small, one sees that
$\CA^\la(\tau,V(\tau))\geq \frac{1}{2}\nu(V(\tau))>\tilde{C}_0$ for some $\tilde{C}_0>0$, for which we also have
\begin{align}
\int_0^te^{-\int_s^t\nu(V(\tau))d\tau}\nu(V(s))ds<\infty.\label{Ala-ubd}
\end{align}
We now turn to estimate $\CH_i$ $(1\leq i\leq 16)$ individually.
We still start with the nonlocal terms $\CH_6$, $\CH_8$, $\CH_9$, $\CH_{13}$, $\CH_{14}$, $\CH_{15}$ and $\CH_{16}$, which turn out to be more intricate and be different from the corresponding estimates in the proof of Theorem \ref{loc.th}, because the estimates we want to obtain here must be uniform in time $t\in(0,\infty).$

For $\CH_6$, if $\ga=0$, one gets from Lemma \ref{CK} that
\begin{align}
|\CH_6|\leq& \frac{C}{M}\int_0^te^{-\int_s^t\frac{\nu(V(\tau))}{2}d\tau}ds
\sup\limits_{0\leq s\leq t} \sum\limits_{\ze'\leq \ze} \|w_l\pa_{\ze'}^{\vth}g_{1}(s)\|_{L^\infty}
\leq \frac{C}{M}
\sup\limits_{0\leq s\leq t} \sum\limits_{\ze'\leq \ze} \|w_l\pa_{\ze'}^{\vth}g_{1}(s)\|_{L^\infty}.\notag
\end{align}
If $0<\ga\leq1$, by \eqref{Ala-ubd} and using \eqref{CK2} in Lemma \ref{g-ck-lem}, we have
\begin{align}
|\CH_6|\leq& C\int_0^te^{-\int_s^t\frac{\nu(V(\tau))}{2}d\tau}\nu(V(s))ds
\sup\limits_{0\leq s\leq t} \left\|\left\{\nu^{-1}w_l\pa_\ze^\vth(\chi_M\CK g_{1})\right\}(s,X(s),V(s))\right\|_{L^\infty}
\notag\\
 \leq&\left({\bf 1}_{\ga>0}\frac{C}{M^{\ga/2}}+\varsigma\right)\sup\limits_{0\leq s\leq t}\sum\limits_{\ze'\leq \ze} \|w_l\pa_{\ze'}^{\vth}g_{1}(s)\|_{L^\infty}.\notag
\end{align}
Next, thanks to Lemma \ref{op.es.lem} and the {\it a priori} assumption \eqref{ap-es} as well as \eqref{Ala-ubd}, it follows
\begin{align}
|\CH_8|\leq& C\int_0^te^{-\int_s^t\frac{\nu(V(\tau))}{2}d\tau}\nu(V(s))ds
\notag\\&\times\sup\limits_{0\leq s\leq t}\left\|e^{\la_0s}\left\{\nu^{-1}w_l\pa_\ze^\vth \{Q(f_1,f_1)+Q(f_1,\sqrt{\mu}f_2)+Q(\sqrt{\mu}f_2,f_1)\}\right\}(s,X(s),V(s))\right\|_{L^\infty}
\notag\\
 \leq&C\sup\limits_{0\leq s\leq t}\sum\limits_{\ze'+\ze''\leq \ze\atop{\vth'\leq\vth}}
 \left\{\|w_l\pa_{\ze'}^{\vth'}g_1(s)\|_{L^\infty}\|w_l\pa_{\ze''}^{\vth-\vth'}g_1(s)\|_{L^\infty}
+\|w_l\pa_{\ze'}^{\vth'}g_1(s)\|_{L^\infty}\|w_l\pa_{\ze''}^{\vth-\vth'}g_2(s)\|_{L^\infty}\right\}\notag\\
 \leq& C\al^2\sup\limits_{0\leq s\leq t}\sum\limits_{\ze'\leq \ze, \vth'\leq\vth}
 \{\|w_l\pa_{\ze'}^{\vth'}g_1(s)\|_{L^\infty}+\|w_l\pa_{\ze'}^{\vth'}g_2(s)\|_{L^\infty}\},\notag
\end{align}
and similarly, in view of \eqref{Ala-ubd} and Theorem \ref{st.sol} and by Lemma \ref{op.es.lem}, one has
\begin{align}
|\CH_9|\leq& C\al\sup\limits_{0\leq s\leq t}
\left\|e^{\la_0s}\left\{\nu^{-1}w_l\pa_\ze^\vth Q(\sqrt{\mu}f,\sqrt{\mu}(G_1+\al G_R))\right\}(s,X(s),V(s))\right\|_{L^\infty}
\notag\\&+C\al\sup\limits_{0\leq s\leq t}\left\|e^{\la_0s}\left\{\nu^{-1}w_l\pa_\ze^\vth Q(\sqrt{\mu}(G_1+\al G_R)),\sqrt{\mu}f)\right\}(s,X(s),V(s))\right\|_{L^\infty}
\notag\\
 \leq&C\al\sup\limits_{0\leq s\leq t}\sum\limits_{\ze'\leq \ze} \|w_l\pa_{\ze'}^\vth[g_{1},g_{2}](s)\|_{L^\infty},\notag
\end{align}
\begin{align}
|\CH_{14}|\leq& {\bf 1}_{\ze>0}C\sup\limits_{0\leq s\leq t}
\big\|e^{\la_0s}\big\{\nu^{-1}w_l\{\pa_\ze^\vth [Q(\sqrt{\mu}f_2,\mu)+Q_{\textrm{gain}}(\mu,\sqrt{\mu}f_2)]
\notag\\&\qquad\qquad-[Q(\sqrt{\mu}\pa_\ze^\vth f_2,\mu)+Q_{\textrm{gain}}(\mu,\sqrt{\mu}\pa_\ze^\vth f_2)]\}
\big\}(s,V(s))\big\|_{L^\infty}
\notag\\
 \leq&{\bf 1}_{\ze>0}C\sup\limits_{0\leq s\leq t}\sum\limits_{\ze'< \ze} \|w_l\pa_{\ze'}^\vth g_{2}(s)\|_{L^\infty},\notag
\end{align}
\begin{align}
|\CH_{15}|\leq& C\sup\limits_{0\leq s\leq t} \left\|e^{\la_0s}\left\{\nu^{-1}w_l\pa_\ze^\vth \{(1-\chi_M)\mu^{-\frac{1}{2}} [Q(f_1,\mu)+Q_{\textrm{gain}}(\mu,f_1)]\}\right\}(s,V(s))\right\|_{L^\infty}
\notag\\
 \leq&C\sup\limits_{0\leq s\leq t}\sum\limits_{\ze'\leq \ze} \|w_l\pa_{\ze'}^\vth g_1(s)\|_{L^\infty}.\notag
\end{align}
For $\CH_{16}$, in light of Lemma \ref{Ga} and the {\it a priori} assumption \eqref{ap-as}, it follows
\begin{align}
|\CH_{16}|\leq& C\sup\limits_{0\leq s\leq t}\sum\limits_{\ze'+\ze''\leq \ze,\vth'\leq\vth} \|w_l\pa_{\ze'}^{\vth'}g_2(s)\|_{L^\infty}\|w_l\pa_{\ze''}^{\vth-\vth'}g_2(s)\|_{L^\infty}\notag\\
 \leq& C\al\sup\limits_{0\leq s\leq t}\sum\limits_{\ze'\leq \ze,\vth'\leq\vth} \|w_l\pa_{\ze'}^{\vth'}g_2(s)\|_{L^\infty}.\notag
\end{align}
For the delicate nonlocal term $\CH_{13}$, we first rewrite
\begin{align}\label{CH11-exp}
\CH_{13}=\int_0^te^{-\int_s^t\CA^\la(\tau)d\tau}\int_{\R^3}{\bf k}_w(V(s),v_\ast)(w_l\pa_\ze^{\vth}g_{2})(s,X(s),v_\ast)dv_\ast ds.
\end{align}
As in Section \ref{st-pro}, the computation for $\CH_{13}$ is then divided into the following three cases.

\noindent\underline{{\it Case 1. $|V(s)|>M$.}} In this case, we get from Lemma \ref{Kop} that
\begin{align}
|\CH_{13}|\leq\frac{C}{M}\sup\limits_{0\leq s\leq t}\|w_l\pa_\ze^{\vth}g_{2}(s)\|_{L^\infty}.\notag
\end{align}
\underline{{\it Case 2. $|V(s)|\leq M$ and $|v_\ast|>2M$.}} At this stage, one has $|V(s)-v_\ast|>M$, thus it follows
\begin{equation*}
\mathbf{k}_w(V,v_\ast)
\leq Ce^{-\frac{\vps M^2}{8}}\mathbf{k}_w(V,v_\ast)e^{\frac{\vps |V-v_\ast|^2}{8}},
\end{equation*}
which gives
\begin{align}
|\CH_{13}|\leq Ce^{-\frac{\vps M^2}{8}}\sup\limits_{0\leq s\leq t}\|w_l\pa_\ze^{\vth}g_{2}(s)\|_{L^\infty},\notag
\end{align}
according to Lemma \ref{Kop}.

\noindent\underline{{\it Case 3. $|V(s)|\leq M$ and $|v_\ast|\leq2M$.}} The key point in this case is to make use of the boundedness of the operator $K$
on the complement of a singular set, so that \eqref{CH11-exp} can be controlled by the $L^1$ norm of $g_{2}$, which further
can be converted to the $L^2$ norm. To see this, for any large $M>0$, we choose a number $p(M)$ to introduce ${\bf k}_{w,p}(V,v_\ast)$ as \eqref{km}, and then write
\begin{align}%\label{CH11-exp}
\CH_{13}=\int_0^te^{-\int_s^t\CA^\la(\tau)d\tau}\int_{\R^3}
[{\bf k}_w-{\bf k}_{w,p}+{\bf k}_{w,p}](V(s),v_\ast)(w_l\pa_\ze^{\vth}g_{2})(s,X(s),v_\ast)dv_\ast ds,\notag
\end{align}
which further gives the bound
\begin{align}\label{CH13-exp}
|\CH_{13}|\leq& \frac{C}{M}\sup\limits_{0\leq s\leq t}\|w_l\pa_\ze^{\vth}g_{2}(s)\|_{L^\infty}
\notag\\&+\underbrace{\int_0^te^{-\int_s^t\CA^\la(\tau)d\tau}{\bf 1}_{|V(s)|\leq M}\int_{|v_\ast|\leq 2M}\mathbf{k}_{w,p}(V(s),v_\ast)|(w_l\pa_\ze^{\vth}g_{2})(s,X(s),v_\ast)|dv_\ast ds}_{\RI}.
\end{align}
Putting the above estimate for $\CH_{13}$ together, we thus have
\begin{align}%\label{CH11-exp}
|\CH_{13}|
\leq& C\left(e^{-\frac{\vps M^2}{8}}+\frac{1}{M}\right)\sup\limits_{0\leq s\leq t}\|w_l\pa_\ze^{\vth}g_{2}(s)\|_{L^\infty}+\RI.\notag
\end{align}
Up to now, one cannot deduce the desired estimate for $\RI$, which in fact will be handled by iteration argument once all the other terms in the right hand side of \eqref{g2k} have been properly controlled.

Let us now turn to compute the other terms in the right hand side of \eqref{g1k}
and \eqref{g2k}.
It is straightforward to see
\begin{align}
|\CH_1|\leq\|w_l\pa_\ze^\vth f_0\|_{L^\infty},\notag
\end{align}
\begin{align}
|\CH_2|,\ |\CH_5|\leq {\bf 1}_{\ze>0}C\sup\limits_{0\leq s\leq t}\sum\limits_{\ze'< \ze\atop{\ze'+\vth'=\ze+\vth}} \|w_l\pa_{\ze'}^{\vth'}g_{1}(s)\|_{L^\infty},\notag
\end{align}
and
\begin{align}
|\CH_{10}|,\ |\CH_{12}|\leq {\bf 1}_{\ze>0}C\sup\limits_{0\leq s\leq t}
\sum\limits_{\ze'< \ze\atop{\ze'+\vth'=\ze+\vth}} \|w_l\pa_{\ze'}^{\vth'} g_{2}(s)\|_{L^\infty}.\notag
\end{align}
From \eqref{bet-exp}, it follows $|\beta|\leq C\al$. We then have
\begin{align}
|\CH_3|+|\CH_4|+|\CH_{11}|\leq C\al\sup\limits_{0\leq s\leq t}\|w_l\pa_\ze^{\vth}[g_{1},g_{2}](s)\|_{L^\infty},\notag
\end{align}
and
\begin{align}
|\CH_7|\leq C\al\sup\limits_{0\leq s\leq t}\sum\limits_{\ze'\leq\ze}\|w_l\pa_{\ze'}^{\vth}g_{2}(s)\|_{L^\infty}.\notag
\end{align}
Consequently, by plugging all the above estimates for $\CH_i$ $(1\leq i\leq 16)$ into \eqref{g1k} and \eqref{g2k}, respectively,
one gets
\begin{align}\label{g1k-es1}
|w_l\pa_\ze^{\vth}g_{1}(t,x,v)|\leq&\|w_l\pa_{\ze}^{\vth}f_0\|_{L^\infty}
+{\bf 1}_{\ze>0}C\sup\limits_{0\leq s\leq t}\sum\limits_{\ze'< \ze\atop{\ze'+\vth'=\ze+\vth}} \|w_l\pa_{\ze'}^{\vth'}g_{1}(s)\|_{L^\infty}
\notag\\
&+C\left(\al+\frac{1}{M}+{\bf 1}_{\ga>0}\frac{1}{M^{\ga/2}}+\varsigma\right)
\sup\limits_{0\leq s\leq t}\sum\limits_{\ze'+\vth'\leq \ze+\vth} \|w_l\pa_{\ze'}^{\vth'}g_{1}(s)\|_{L^\infty}
\notag\\&+C\al
\sup\limits_{0\leq s\leq t}\sum\limits_{\ze'+\vth'\leq \ze+\vth} \|w_l\pa_{\ze'}^{\vth'}g_2(s)\|_{L^\infty},
\end{align}
and
\begin{align}\label{g2k-es1}
|w_l\pa_\ze^{\vth}g_{2}(t,x,v)|\leq&
C\sup\limits_{0\leq s\leq t}\sum\limits_{\ze'\leq \ze}\|w_l\pa_{\ze'}^{\vth}g_{1}(s)\|_{L^\infty}
+{\bf 1}_{\ze>0}C\sup\limits_{0\leq s\leq t}\sum\limits_{\ze'< \ze\atop{|\ze'|+|\vth'|=\ze+\vth}} \|w_l\pa_{\ze'}^{\vth'}g_{2}(s)\|_{L^\infty}\notag\\
&+C\left(\al+e^{-\frac{\vps M^2}{8}}+\frac{1}{M}\right)\sup\limits_{0\leq s\leq t}\sum\limits_{\ze'\leq \ze} \|w_l\pa_{\ze'}^{\vth}g_{2}(s)\|_{L^\infty}
+\RI.
\end{align}
To continue, we have by substituting \eqref{g2k-es1} into $\RI$ defined in \eqref{CH13-exp} that
\begin{align}\label{RI-p1}
\RI\leq& C\int_0^te^{-\int_s^t\CA^\la(\tau)d\tau}{\bf 1}_{|V(s)|\leq M}\int_{|v_\ast|\leq 2M}\mathbf{k}_{w,p}(V(s),v_\ast)\bigg\{\sup\limits_{0\leq \tau\leq s}\sum\limits_{\ze'\leq \ze}\|w_l\pa_{\ze'}^{\vth}g_{1}(\tau)\|_{L^\infty}
\notag\\&+{\bf 1}_{\ze>0}\sup\limits_{0\leq \tau\leq s}\sum\limits_{\ze'< \ze\atop{|\ze'|+|\vth'|=\ze+\vth}} \|w_l\pa_{\ze'}^{\vth'}g_{2}(\tau)\|_{L^\infty}\bigg\}dv_\ast ds\notag\\
&+\int_0^te^{-\int_s^t\CA^\la(\tau)d\tau}{\bf 1}_{|V(s)|\leq M}\int_{|v_\ast|\leq 2M}\mathbf{k}_{w,p}(V(s),v_\ast)
\int_0^se^{-\int_{s'}^s\CA^\la(\tau)d\tau}{\bf 1}_{|V(s')|\leq M}\notag\\
&\qquad\times\int_{|v'_\ast|\leq 2M}\mathbf{k}_{w,p}(V(s'),v'_\ast)
{\bf 1}_{X(s')\in \T^3}|(w_l\pa_\ze^{\vth}g_{2})(s',X(s'),v'_\ast)|dv'_\ast dv_\ast ds'ds,
\end{align}
where we have denoted
\begin{align*}%\label{H1CLp1}
\left\{\begin{array}{rll}
&V(s')=V(s';s,X(s),v_\ast)= e^{-(s'-s)(\beta I+\al A)}v_\ast,\\[2mm]
&X(s')=X(s';s,X(s),v_\ast)=X(s)-(\beta I+\al A)^{-1}\left[e^{-(s'-s)(\beta I+\al A)}-I\right]v_\ast,
\end{array}\right.
\end{align*}
according to \eqref{CH2-p2}. As a consequence, \eqref{RI-p1} further implies
\begin{align}%\label{RI-p2}
\RI\leq C\sup\limits_{0\leq \tau\leq t}\sum\limits_{\ze'\leq \ze}\|w_l\pa_{\ze'}^{\vth}g_{1}(\tau)\|_{L^\infty}
+C{\bf 1}_{\ze>0}\sup\limits_{0\leq \tau\leq t}\sum\limits_{\ze'< \ze\atop{\ze'+\vth'=\ze+\vth}} \|w_l\pa_{\ze'}^{\vth'}g_{2}(\tau)\|_{L^\infty}+\RH,\notag
\end{align}
with $\RH$ denoting the second term in the right hand side of \eqref{RI-p1}. To compute $\RH$, we then split it into the following two integrals
\begin{align}
\RH=&\int_0^te^{-\int_s^t\CA^\la(\tau)d\tau}{\bf 1}_{|V(s)|\leq M}\int_{|v_\ast|\leq 2M}\mathbf{k}_{w,p}(V(s),v_\ast)
\left\{\int_0^{s-\eta_0}+\int_{s-\eta_0}^{s}\right\}e^{-\int_{s'}^s\CA^\la(\tau)d\tau}{\bf 1}_{|V(s')|\leq M}\notag\\
&\qquad\times\int_{|v'_\ast|\leq 2M}\mathbf{k}_{w,p}(V(s'),v'_\ast){\bf 1}_{X(s')\in \T^3}|(w_l\pa_\ze^{\vth}g_{2})(s',X(s'),v'_\ast)|
dv'_\ast dv_\ast ds'ds:=\RH_1+\RH_2,\notag
\end{align}
where $\eta_0>0$ is suitably small. It is straightforward to see that
$$
\RH_2\leq C\eta_0 \sup\limits_{0\leq s\leq t}\|w_l\pa_\ze^{\vth}g_{2}(s)\|_{L^\infty}.
$$
For $\RH_1$, since $s-s'\geq\eta_0$ in this integral, the Jacobian
\begin{align}
\tilde{\mathcal {J}}=:\Bigg|\left|\frac{\pa X(s')}{\pa v_\ast}\right|\Bigg|
=&\Bigg|\left|(\beta I+\al A)^{-1}\left[e^{-(s'-s)(\beta I+\al A)}-I\right]\right|\Bigg|\notag
\geq(s-s')^3/8\geq\eta_0^3/8,\notag
\end{align}
according to Lemma \ref{expm-lem}. Moreover, if we denote
$$
\Om_y=\left\{y\Big||y-X(s)|\leq \left|(\beta I+\al A)^{-1}\left[e^{-(s'-s)(\beta I+\al A)}-I\right]v_\ast\right|\right\},
$$
then, by applying \eqref{exma-ubd} of Lemma \ref{expm-lem}, we have
\begin{align}
|\Om_y|\leq C(s-s')e^{C\al (s-s')}.\notag
\end{align}
With these, one gets by a change of variable $X(s')\rightarrow y$
that
\begin{align}
\RH_2\leq& C\int_0^te^{-c_0(t-s)}
\int_0^{s-\eta_0}e^{-c_0(s-s')}
\int_{|v'_\ast|\leq 2M}\left(\int_{|v_\ast|\leq 2M}|(\pa_\ze^{\vth}g_{2})(s',X(s'),v'_\ast)|^2
dv_\ast\right)^{\frac{1}{2}} dv'_\ast ds'ds\notag\\
\leq& C\int_0^te^{-c_0(t-s)}
\int_0^{s-\eta_0}e^{-c_0(s-s')}
\int_{|v'_\ast|\leq 2M}\left(\int_{\Om_y}\tilde{\mathcal {J}}^{-3}|(\pa_\ze^{\vth}g_{2})(s',y,v'_\ast)|^2
dy\right)^{\frac{1}{2}} dv'_\ast ds'ds\notag\\
\leq& C\eta_0^{-3/2}\int_0^te^{-c_0(t-s)}
\int_0^{s-\eta_0}e^{-c_0(s-s')} \left(|\Om_y|^{\frac{1}{2}}+1\right)
\int_{|v'_\ast|\leq 2M}\left(\int_{\T^3}|(\pa_\ze^{\vth}g_{2})(s',y,v'_\ast)|^2
dy\right)^{\frac{1}{2}} dv'_\ast ds'ds
\notag\\
\leq& C\eta_0^{-3/2}\sup\limits_{0\leq s\leq t}\left(\int_{\R^3}\int_{\T^3}|(\pa_\ze^{\vth}g_{2})(s,y,v)|^2
dydv\right)^{\frac{1}{2}}.\notag
\end{align}
Thus, it follows
\begin{align}%\label{RI-p3}
\RI\leq& C\sup\limits_{0\leq s\leq t}\sum\limits_{\ze'\leq \ze}\|w_l\pa_{\ze'}^{\vth}g_{1}(s)\|_{L^\infty}
+C{\bf 1}_{\ze>0}\sup\limits_{0\leq s\leq t}\sum\limits_{\ze'< \ze\atop{\ze'+\vth'=\ze+\vth}} \|w_l\pa_{\ze'}^{\vth'}g_{2}(s)\|_{L^\infty}\notag\\
&+C\eta_0 \sup\limits_{0\leq s\leq t}\|w_l\pa_\ze^{\vth}g_{2}(s)\|_{L^\infty}
+C_{\eta_0}\sup\limits_{0\leq s\leq t}\|\pa_\ze^{\vth}g_{2}(s)\|.\notag
\end{align}
This together with \eqref{g2k-es1} further gives
\begin{align}\label{g2k-es2.v}
|w_l\pa_\ze^{\vth}g_{2}(t,x,v)|\leq&
C\sup\limits_{0\leq s\leq t}\sum\limits_{\ze'\leq \ze}\|w_l\pa_{\ze'}^{\vth}g_{1}(s)\|_{L^\infty}
+{\bf 1}_{\ze>0}C\sup\limits_{0\leq s\leq t}\sum\limits_{\ze'< \ze\atop{\ze'+\vth'=\ze+\vth}} \|w_l\pa_{\ze'}^{\vth'}g_{2}(s)\|_{L^\infty}\notag\\
&+C\left(\al+e^{-\frac{\vps M^2}{8}}+\frac{1}{M}+\eta_0\right)\sup\limits_{0\leq s\leq t}\sum\limits_{\ze'\leq \ze} \|w_l\pa_{\ze'}^{\vth}g_{2}(s)\|_{L^\infty}
+C_{\eta_0}\sup\limits_{0\leq s\leq t}\|\pa_\ze^{\vth}g_{2}(s)\|.
\end{align}
Finally, taking a linear combination of \eqref{g1k-es1} and \eqref{g2k-es2.v} with $|\ze|=0,1,\cdots, N$ and $|\ze|+|\vth|\leq N$, respectively, and adjusting constants, we conclude
\begin{align}\label{g1k-es2}
\sup\limits_{0\leq s\leq t}\sum\limits_{|\ze|+|\vth|\leq N}\|w_l\pa_\ze^{\vth}g_{1}(s)\|_{L^\infty}
\leq&\sum\limits_{|\ze|+|\vth|\leq N}\|w_l\pa_\ze^{\vth}f_0\|_{L^\infty}+
C\al\sup\limits_{0\leq s\leq t}\sum\limits_{|\ze|+|\vth|\leq N} \|w_l\pa_\ze^{\vth}g_{2}(s)\|_{L^\infty},
\end{align}
and
\begin{align}\label{g2k-es2}
\sup\limits_{0\leq s\leq t}\sum\limits_{|\ze|+|\vth|\leq N}\|w_l\pa_\ze^{\vth}g_{2}(s)\|_{L^\infty}\leq&
C\sup\limits_{0\leq s\leq t}\sum\limits_{|\ze|+|\vth|\leq N}\|w_l\pa_\ze^{\vth}g_{1}(s)\|_{L^\infty}
+C\sup\limits_{0\leq s\leq t}\sum\limits_{|\ze|+|\vth|\leq N}\|\pa_\ze^{\vth}g_2(s)\|.
\end{align}
Actually, from the proof presented above, we have the following refined estimates concerning the $L^\infty$ norm of $g_1$ and $g_2$ without velocity derivatives.
\begin{lemma}
Under the hypothesis \eqref{ap-as},
it holds that
\begin{align}\label{g1-g2-z}
\left\{\begin{array}{rll}
&\sup\limits_{0\leq s\leq t}\|w_lg_{1}(s)\|_{L^\infty}
\leq C\|w_lf_0\|_{L^\infty}+
C\al\sup\limits_{0\leq s\leq t} \|w_lg_{2}(s)\|_{L^\infty},\\
&\sup\limits_{0\leq s\leq t}\|w_lg_{2}(s)\|_{L^\infty}\leq
C\sup\limits_{0\leq s\leq t}\|w_lg_{1}(s)\|_{L^\infty}
+C\sup\limits_{0\leq s\leq t}\|g_2(s)\|,
\end{array}\right.
\end{align}
and
\begin{align}\label{g1-g2-hs}
\left\{\begin{array}{rll}
&\sup\limits_{0\leq s\leq t}\sum\limits_{1\leq |\vth|\leq N}|w_l\pa^\vth g_{1}(s)\|_{L^\infty}
\leq C\sum\limits_{1\leq |\vth|\leq N}\|w_l\pa^\vth f_0\|_{L^\infty}+
C\al\sup\limits_{0\leq s\leq t} \sum\limits_{1\leq |\vth|\leq N}\|w_l\pa^\vth g_{2}(s)\|_{L^\infty},\\
&\sup\limits_{0\leq s\leq t}\sum\limits_{1\leq |\vth|\leq N}\|w_l\pa^\vth g_{2}(s)\|_{L^\infty}\leq
C\sup\limits_{0\leq s\leq t}\sum\limits_{1\leq |\vth|\leq N}\|w_l\pa^\vth g_{1}(s)\|_{L^\infty}
+C\sup\limits_{0\leq s\leq t}\sum\limits_{1\leq |\vth|\leq N}\|\pa^\vth g_2(s)\|.
\end{array}\right.
\end{align}

\end{lemma}

\noindent\underline{{\it Step 2. $L^2$ estimates.}} To close our final estimate, it remains then to deduce the $H_{x,v}^N$ estimate of $e^{\la_0t}f_2(t,v)$ in \eqref{g2k-es2}. The computation is divided into the following three sub-steps.

\noindent\underline{{\it Step 2.1. The estimates for $c$.}}
In this sub-step, we consider the basic $L^2$ estimate for $c$, which is difficult to be obtained due to the exponential growth of the heat flux, cf. \cite{JNV-ARMA}.
Recall $[g_1,g_2](t,v)=e^{\la_0t}[f_1,f_2](t,v)$ and $\sqrt{\mu}g=g_1+\sqrt{\mu}g_2$. It is straightforward to see that
$g$ satisfies
\begin{align}\label{gb-eq}
\pa_t g&+v\cdot\na_xg-\beta\mu^{-1/2}\na_v\cdot(v\sqrt{\mu}g)-\al\mu^{-1/2}\na_v\cdot(Av\sqrt{\mu}g)-\la_0g+Lg
\notag\\&=\underbrace{e^{-\la_0t}\Ga(g,g)+\al\{\Ga(G_1+\al G_R,g)+\Ga(g,G_1+\al G_R)\}}_{\RR},
\end{align}
with
\begin{align}\label{gb-id}
\sqrt{\mu}g(0,x,v)=f_{0}(x,v).
\end{align}
Similar to \eqref{abc.def}, we define
\begin{align}
\FP_0 g&=\{a(t,x)+\Fb(t,x)\cdot v+c(t,x)(|v|^2-3)\}\sqrt{\mu},\notag \\
%\end{align}
%and
%\begin{align}
\bar{\FP}_0g_1&=\{a_1(t,x)+\Fb_1(t,x)\cdot v+c_1(t,x)(|v|^2-3)\}\mu,\notag
\end{align}
and recall the definition
$$
\FP_0g_2=\{a_2(t,x)+\Fb_2(t,x)\cdot v+c_2(t,x)(|v|^2-3)\}\sqrt{\mu}.\notag
$$
Then it follows
\begin{align}\label{abc-p1}
a(t,x)=a_1(t,x)+a_2(t,x),\ \Fb(t,x)=\Fb_1(t,x)+\Fb_2(t,x),\ c(t,x)=c_1(t,x)+c_2(t,x),
\end{align}
for any $t\geq0$ and $x\in\T^3$. In addition, \eqref{gb-eq} together with \eqref{id.cons} and \eqref{gb-id} implies
\begin{align}\label{abc-p2}
\int_{\T^3}a(t,x)dx=0,\ \int_{\T^3}b_i(t,x)dx=0,\ i=1,2,3,
\end{align}
where we have denoted $\Fb=(b_1,b_2,b_3)$.
Next, taking the moments
$$
\sqrt{\mu},\ v_i\sqrt{\mu},\ \frac{1}{6}(|v|^2-3)\sqrt{\mu},\  i=1,2,3
$$
for the equation \eqref{gb-eq}, one has
\begin{align}
\pa_ta+\na_x\cdot \Fb=0,\notag%\label{a-eq}
\end{align}
\begin{align}\label{b-eq}
\pa_tb_i+\pa_i(a+2c)+\sum\limits_{j}((\beta-\la_0) I+\al A)_{ij}b_j+\sum\limits_{j}\pa_j\lag \RA_{ij},\FP_1g\rag=0,
\end{align}
and
\begin{align}\label{c-eq}
\pa_tc+\frac{1}{3}\na_x\cdot\Fb+\beta_1\al^2(2c+a)-\la_0 c+\frac{1}{6}\na_x\cdot\lag (|v|^2-5)v\sqrt{\mu},\FP_1g\rag+\frac{\al}{3}\sum\limits_{i,j}a_{ij}\lag \RA_{ij},\FP_1g\rag=0,
\end{align}
where
$$
\RA_{ij}=(v_iv_j-\frac{\de_{ij}}{3}|v|^2)\sqrt{\mu},\ i,j=1,2,3
$$
with $\de_{ij}$ being the Kronecker delta, and the identity $
\beta_0+\frac{1}{3}{\rm tr}A=0$ was used while deriving \eqref{c-eq}.

Furthermore, taking the higher order moments $\RA_{ij}$ and
$$
\RB_i\eqdef\frac{1}{10}(|v|^2-5)v_i\sqrt{\mu},\ i,j=1,2,3
$$
for the equation \eqref{gb-eq}, respectively, we obtain
\begin{align}\label{2rd-db}
\pa_t&\lag \RA_{ij},\FP_1g\rag+\pa_ib_j+\pa_jb_i-\frac{2}{3}\na_x\cdot\Fb\de_{ij}
+\lag\RA_{ij},v\cdot\na_x \FP_1g\rag
-\beta\lag \mu^{-1/2}\na_v\cdot(v\sqrt{\mu} g),\RA_{ij}\rag
\notag\\&-\al\lag\mu^{-1/2}\na_v\cdot(Av\sqrt{\mu}g),\RA_{ij}\rag
-\la_0\lag \FP_1g,\RA_{ij}\rag=\lag-Lg+\RR,\RA_{ij}\rag,
\end{align}
and
\begin{align}\label{2rd-dc}
\pa_t&\lag \RB_i,\FP_1g\rag+\pa_ic
+\lag \RB_i,v\cdot\na_x\FP_1g\rag
-\beta\lag \mu^{-1/2}\na_v\cdot(v\sqrt{\mu}g),\RB_i\rag
\notag\\&-\al\lag\mu^{-1/2}\na_v\cdot(Av\sqrt{\mu}g),\RB_i\rag
-\la_0\lag \FP_1g,\RB_i\rag=\lag-Lg+\RR,\RB_i\rag.
\end{align}

Choosing $\la_0=\beta_1 \al^2$, we get from the inner product of $(\eqref{c-eq},c)$
that
\begin{align}
\frac{1}{2}\frac{d}{dt}\|c\|^2
&+\beta_1\al^2\|c\|^2
+\frac{\al}{3}\sum\limits_{i,j}a_{ij}(c,\lag\mathscr{A}_{ij},\FP_1g_1\rag)
+\frac{\al}{3}\sum\limits_{i,j}a_{ij}(c_1,\lag\mathscr{A}_{ij},\FP_1g_2\rag)
\notag\\
&+\frac{\al}{3}\sum\limits_{i,j}a_{ij}(c_2,\lag\mathscr{A}_{ij},\FP_1g_2\rag)-\frac{1}{3}(\Fb,\na_xc)+\beta_1\al^2(a,c)
\notag\\&-\frac{1}{6}(\lag (|v|^2-5)v\sqrt{\mu},\FP_1g\rag,\na_xc)=0.\label{dis-c-p1}
\end{align}
Note that the delicate term $\frac{\al}{3}\sum\limits_{i,j}a_{ij}(c_2,\lag\mathscr{A}_{ij},\FP_1g_2\rag)$ will be cancelled later on.

We now derive the $L^2$ estimate on $\FP_1g_2$.
Recall that $g_2$ satisfies
\begin{align}\label{f2-eq-2}
\pa_t g_2&+v\cdot\na_xg_2-\beta\na_v\cdot(vg_2)-\al\na_v\cdot(Avg_2)+\al Av\cdot(\na_v\sqrt{\mu})(|v|^2-3)c_{2}-\la_0g_2+Lg_2\notag\\
=&(1-\chi_M)\mu^{-\frac{1}{2}}\CK g_1+e^{-\la_0t}\Ga(g_2,g_2),
\end{align}
and
\begin{align}%\label{f2-id-2}
g_2(0,x,v)=0.\notag
\end{align}
Taking the inner product of \eqref{f2-eq-2} and $\FP_1g_2$ over $(x,v)\in\T^3\times\R^3$ and applying Cauchy-Schwarz's inequality, one has
\begin{align}\label{g2-l2}
\frac{1}{2}\frac{d}{dt}&\|\FP_1g_2\|^2
-\al (Av \cdot\na_v\{[a_2+\Fb_2\cdot v]\sqrt{\mu}\},\FP_1g_2)-2\al(Av\cdot v\sqrt{\mu}c_2,\FP_1g_2)
+\la\|\FP_1g_2\|^2_\nu\notag\\
\leq& C\|\na_x[a_2,\Fb_2,c_2]\|^2+C\al^2\|w_lg_2\|_{L^\infty}^2
+C\|w_lg_1\|_\infty^2\notag\\
\leq& C\|\na_x[a,\Fb,c]\|^2
+C\al^2\|w_lg_2\|_{L^\infty}^2
+C\|w_lg_1\|_\infty^2,
\end{align}
according to \eqref{abc-p1}, \eqref{abc-p2} and the following estimate
\begin{align}%\label{Ga-tri}
|(\Ga(g_2,g_2),\FP_1g_2)|\leq&\eta\|\FP_1g_2\|^2_\nu+C_\eta \int_{\T^3}\|\nu^{\frac{1}{2}}g_2\|_{L^2_v}^4dx
\leq \eta\|\FP_1g_2\|^2_\nu+C_\eta\al^2 \|w_lg_2\|_{L^\infty}^2\notag
\end{align}
in the case of $l\geq2.$

Notice that $(Av\cdot v\sqrt{\mu}c_2,\FP_1g_2)=\sum\limits_{i,j}a_{ij}(c_2,\lag\mathscr{A}_{ij},\FP_1g_2\rag)$,
we now get from the summation of $\eqref{dis-c-p1}$ and $\frac{1}{6}\eqref{g2-l2}$ that
\begin{align}\label{g2-l2-plus}
\frac{1}{2}\frac{d}{dt}&\|c\|^2
+\frac{1}{12}\frac{d}{dt}\|\FP_1g_2\|^2
+\la\al^2\|c\|^2+\la\|\FP_1g_2\|^2_\nu\notag\\
\leq& C_\eta\sup\limits_{0\leq s\leq t}\|\na_x[a,\Fb,c](s)\|^2
+(C\al^2+\eta)\|w_lg_2\|_{L^\infty}^2
+C\sup\limits_{0\leq s\leq t}\|w_lg_1(s)\|_\infty^2,
\end{align}
where we have used the relations $\|g_1\|\leq\|\FP_0g_2(s)\|+\|\FP_1g_2(s)\|$, and $\|\FP_0g_2(s)\|\leq C\|[a,\Fb,c]\|+\|w_lg_{1}(s)\|_{L^\infty}$ for $l>5/2$ according to \eqref{abc-p1}. Further, \eqref{g2-l2-plus} gives
\begin{align}\label{g2-tt-eng}
\sup\limits_{0\leq s\leq t}\al^2\|g_2(s)\|^2
\leq&C\sup\limits_{0\leq s\leq t}\|\na_x[a,\Fb,c](s)\|^2+(C\al^2+\eta)\sup\limits_{0\leq s\leq t}\|w_lg_2(s)\|_{L^\infty}^2
\notag\\&+C\sup\limits_{0\leq s\leq t}\|w_lg_1(s)\|_\infty^2.
\end{align}
Next, substituting \eqref{g2-tt-eng} into $\eqref{g1-g2-z}_2$, one has
\begin{align}\label{g2-mic-lif-p2}
\al\sup\limits_{0\leq s\leq t}\|w_lg_{2}(s)\|_{L^\infty}
\leq&
C(\al+(\eta+\al)^{\frac{1}{2}})\sup\limits_{0\leq s\leq t}\|w_lg_{1}(s)\|_{L^\infty}
+C\sup\limits_{0\leq s\leq t}\|\na_x[a,\Fb,c](s)\|.
\end{align}
%due to that fact that $\|\FP_1g_2(s)\|\leq C\|w_lg_{2}(s)\|_{L^\infty}$ for $l>5/2.$
Finally, we get from $\eqref{g1-g2-z}_1$ and \eqref{g2-mic-lif-p2} that
\begin{align}\label{g12-lif-z}
\sup\limits_{0\leq s\leq t}\|w_lg_{1}(s)\|_{L^\infty}+\al\sup\limits_{0\leq s\leq t}\|w_lg_{2}(s)\|_{L^\infty}
\leq C\|w_lf_0\|_{L^\infty}+C\sup\limits_{0\leq s\leq t}\|\na_x[a,\Fb,c](s)\|.
\end{align}

\noindent\underline{{\it Step 2.2. Higher order estimates for $[a,\Fb,c]$.}}
We are now in a position to deduce the higher order $L^2$ estimates on $[a,\Fb,c]$. To do this, we first get from \eqref{2rd-db}
that
\begin{align}\label{2rd-db-p2}
\sum\limits_{i,i\neq j}&\pa_t\pa_i\lag \RA_{ij},\FP_1g\rag
+\pa_t\pa_j\lag \RA_{jj},\FP_1g\rag
+\Delta b_j+\frac{1}{3}\pa_j\na\cdot\Fb=\RT,
\end{align}
with
\begin{align}%\label{2rd-db-p2}
\RT=&
\sum\limits_{i,i\neq j}\pa_i\bigg\{-\lag\RA_{ij},v\cdot\na_x \FP_1g\rag
+\beta\lag \mu^{-1/2}\na_v\cdot(v\sqrt{\mu} g),\RA_{ij}\rag
\notag\\&\qquad\qquad+\al\lag\mu^{-1/2}\na_v\cdot(Av\sqrt{\mu}g),\RA_{ij}\rag
+\la_0\lag \FP_1g,\RA_{ij}\rag+\lag-Lg+\RR,\RA_{ij}\rag\bigg\}\notag\\
&+\pa_j\bigg\{-\lag\RA_{ij},v\cdot\na_x \FP_1g\rag
+\beta\lag \mu^{-1/2}\na_v\cdot(v\sqrt{\mu} g),\RA_{ij}\rag
\notag\\&\qquad\qquad+\al\lag\mu^{-1/2}\na_v\cdot(Av\sqrt{\mu}g),\RA_{ij}\rag
+\la_0\lag \FP_1g,\RA_{ij}\rag+\lag-Lg+\RR,\RA_{ij}\rag\bigg\}.\notag
\end{align}
Letting $N-1\geq|\vth|\geq1$,
one has from $\sum\limits_{j}(\pa^\vth\eqref{2rd-db-p2},\pa^\vth b_j)$, $\sum\limits_{i}(\pa^\vth\eqref{2rd-dc},\pa^\vth \pa_ic)$ and
$\sum\limits_{i}(\pa^\vth\eqref{b-eq},\pa^\vth \pa_ia)$ that
\begin{align}\label{2rd-db-p3}
\frac{d}{dt}&\CE^{int}_{\Fb}+\|\na_x\pa^\vth\Fb\|^2+\frac{1}{3}\|\pa^\vth\na_x\cdot\Fb\|^2\notag\\
=&\sum\limits_{j}\bigg(\sum\limits_{i,i\neq j}\pa_i\pa^\vth\lag \RA_{ij},\FP_1g\rag
+\pa_j\pa^\vth\lag \RA_{jj},\FP_1g\rag,\pa^\vth \Big(\pa_j(a+2c)\notag\\&\qquad+\sum\limits_{i}((\beta-\la_0) I+\al A)_{ij}b_i+\sum\limits_{i}\pa_i\lag \RA_{ij},\FP_1g\rag\Big)\bigg)-\sum\limits_{j}(\pa^\vth\RT,\pa^\vth b_j)
\notag\\
\leq&(\eta+\al)\|\pa^\vth[a,\Fb,c]\|^2+C(\eps_0^2+\al^2)\sum\limits_{1\leq |\vth'|\leq N}\|w_l\pa^{\vth'}g_2\|_{L^\infty}^2
\notag\\&+C_\eta\sum\limits_{1\leq|\vth'|\leq N}\{\|\pa^{\vth'}\FP_1g_2\|^2+\|w_l\pa^{\vth'} g_1\|^2_{L^\infty}\},
\end{align}
\begin{align}\label{2rd-dc-p2}
\frac{d}{dt}\CE_c^{int}&+\|\na_x\pa^\vth c\|^2\notag\\
=&-\sum\limits_{i}\bigg(\pa^\vth \Big(\lag \RB_i,v\cdot\na_x\FP_1g\rag
-\beta\lag \mu^{-1/2}\na_v\cdot(v\sqrt{\mu}g),\RB_i\rag
\notag\\&\qquad-\al\lag\mu^{-1/2}\na_v\cdot(Av\sqrt{\mu}g),\RB_i\rag
-\la_0\lag \FP_1g,\RB_i\rag-\lag Lg+\RR,\RB_i\rag\Big),\pa^\vth \pa_ic\bigg)
\notag\\
&+\sum\limits_{i}\bigg(\pa_i\pa^\vth\lag \RB_i,\FP_1g\rag,\pa^\vth \Big(\beta_1\al^2(2c+a)-\la_0 c+\frac{1}{6}\na_x\cdot\lag (|v|^2-5)v,\FP_1g\rag\notag\\&\qquad+\frac{\al}{3}\sum\limits_{i,j}a_{ij}\lag \RA_{ij},\FP_1g\rag\Big)\bigg)\notag\\
\leq&(\eta+\al)\|\pa^\vth[a,\Fb,c]\|^2+C(\eps_0^2+\al^2)\sum\limits_{1\leq|\vth'|\leq N}\|w_l\pa^{\vth'}g_2\|_{L^\infty}^2
\notag\\&+C_\eta\sum\limits_{1\leq|\vth'|\leq N}\{\|\pa^{\vth'}\FP_1g_2\|^2+\|w_l\pa^{\vth'} g_1\|^2_{L^\infty}\},
\end{align}
and
\begin{align}%\label{2rd-da-p2}
\frac{d}{dt}\CE_a^{int}&+\|\na_x\pa^\vth a\|^2\notag\\=&-\sum\limits_{i}(\pa^\vth\pa_ib_i,\pa^\vth \na_x\cdot\Fb)
-2\sum\limits_{i}(\pa^\vth\pa_ic,\pa^\vth \pa_ia)\notag\\&
-\sum\limits_{i,j}\bigg({\pa^\vth
\Big[\big((\beta-\la_0) I+\al A\big)_{ij}b_j+\sum\limits_{j}\pa_j\lag \RA_{ij},\FP_1g\rag\Big]},\pa^\vth \pa_ia\bigg)
\notag\\
\leq&(\eta+\al)\|[\pa^\vth a,\na_x\pa^\vth a]\|^2+C\|\na_x\pa^\vth[\Fb,c]\|^2
+C_\eta\sum\limits_{1\leq|\vth'|\leq N}\{\|\pa^{\vth'}\FP_1g_2\|^2+\|w_l\pa^{\vth'} g_1\|^2_{L^\infty}\},\notag
\end{align}
respectively,
where we have set
\begin{align}\label{abc-in-eng}
\left\{\begin{array}{rll}
\CE_\Fb^{int}=&-\sum\limits_{j}\bigg(\pa^\vth\Big(\sum\limits_{i,i\neq j}\pa_i\lag \RA_{ij},\FP_1g\rag
+\pa_j\lag \RA_{jj},\FP_1g\rag\Big),\pa^\vth b_j\bigg),\\
\CE_c^{int}=&\sum\limits_{i}(\pa^\vth\lag \RB_i,\FP_1g\rag,\pa^\vth \pa_ic),\\
\CE_a^{int}=&\sum\limits_{i}(\pa^\vth\pa_ib_i,\pa^\vth \pa_ia),
\end{array}\right.
\end{align}
and in addition, for $l>4$, the following estimates of the type
\begin{align}%\label{LB-es}
\|\pa^\vth\lag Lg,\RB_{i}\rag\|^2\leq& C\|\nu^{-1}\Ga(\pa^\vth \FP_1g,\sqrt{\mu})\|^2
+C\|\nu^{-1}\Ga(\sqrt{\mu},\pa^\vth \FP_1g)\|^2
\notag\\
\leq&C\|\FP_1\pa^\vth g_2\|^2+C\|w_l\pa^\vth g_2\|^2_{L^\infty}\notag
\end{align}
and
\begin{align}
\|\pa^\vth\lag \Ga(g,g),\RB_i\rag\|^2\leq& C\sum\limits_{\vth'\leq\vth}\int_{\T^3}\|\nu^{-1}w_{l}Q(\sqrt{\mu}\pa^{\vth'}g,\sqrt{\mu}\pa^{\vth-\vth'}g)\|_{L^\infty}^2dx\notag\\
\leq& C\sum\limits_{\vth'\leq\vth}\|w_l[\pa^{\vth'}g_1,\pa^{\vth-\vth'}g_2]\|^4_{L^\infty}\leq C\eps_0^2\sum\limits_{1\leq|\vth'|\leq\vth}\|w_l[\pa^{\vth'}g_1,\pa^{\vth'}g_2]\|^2_{L^\infty}\notag
\end{align}
have been used.

Consequently, letting $\ka_1>0$ be suitably small, we get from the summation of \eqref{2rd-db-p3}, \eqref{2rd-dc-p2} and $\ka_1\times\eqref{2rd-dc-p2}$ that
\begin{align}\label{2rd-dabc}
\frac{d}{dt}[\ka_1\CE_a^{int}&+\CE_\Fb^{int}+\CE_c^{int}]
+\la\|\na_x[a,\Fb,c]\|^2
+\la\sum\limits_{1\leq|\vth|\leq N}\|\na_x\pa^\vth [a,\Fb,c]\|^2\notag\\
\leq& C\al^2\sum\limits_{1\leq|\vth|\leq N}\|w_l\pa^{\vth}g_2\|_{L^\infty}^2
+C\sum\limits_{1\leq|\vth|\leq N}\{\|\pa^{\vth}\FP_1g_2\|^2+\|w_l\pa^{\vth} g_1\|^2_{L^\infty}\},
\end{align}
where the Poincar\'{e}'s inequality $\|\na_x[a,\Fb,c]\| \leq C\|\na^2_x[a,\Fb,c]\|$ has been also used.

\noindent\underline{{\it Step 2.3. Higher order estimates for $\FP_1g_2$.}} With the above estimates in our hands, we then turn to obtain the higher order $L^2$ estimates on $\FP_1g_2$. For this, letting $1\leq|\vth|\leq N$, we take the inner product of $\pa^{\vth}\eqref{f2-eq-2}$ with $\pa^{\vth}g_2$ and apply Lemma \ref{es-L} so as to obtain
\begin{align}\label{g2-l2-p2}
\sum\limits_{1\leq|\vth|\leq N}&\frac{d}{dt}\|\pa^{\vth}g_2\|^2+\de_0\sum\limits_{1\leq|\vth|\leq N}\|\pa^{\vth}\FP_1g_2\|_\nu^2\notag\\
\leq& (C\al^2+\eta+\la_0)\sum\limits_{1\leq|\vth|\leq N}\|\FP_0\pa^{\vth}g_2\|^2
+C\al^2\sum\limits_{1\leq|\vth|\leq N}\|w_l\pa^{\vth}g_2\|_{L^\infty}^2+C_\eta\sum\limits_{1\leq|\vth|\leq N}\|w_l\pa^{\vth}g_1\|_\infty^2,
\end{align}
where according to Lemma \ref{Ga} and the {\it a priori} assumption \ref{ap-as}, the following estimate has been used:
\begin{align}%\label{Ga-tri}
|(\pa^\vth\Ga(g_2,g_2),\pa^{\vth}g_2)|=&|(\pa^\vth\Ga(g_2,g_2),\pa^{\vth}\FP_1g_2)|\notag\\
\leq& \eta\|\pa^{\vth}\FP_1g_2\|^2_\nu
+C_\eta \int_{\T^3}\|\nu^{\frac{1}{2}}\pa^{\vth'}g_2\|_{L^2_v}^2\|\nu^{\frac{1}{2}}\pa^{\vth-\vth'}g_2\|_{L^2_v}^2dx
\notag\\
\leq& \eta\|\pa^{\vth}\FP_1g_2\|^2_\nu
+C_\eta \|w_l\pa^{\vth'}g_2\|_{L^\infty}^2\|w_l\pa^{\vth-\vth'}g_2\|_{L^\infty}^2
\notag\\
\leq& \eta\|\pa^{\vth}\FP_1g_2\|^2_\nu
+C_\eta \al^2\sum\limits_{1\leq |\vth|\leq N}\|w_l\pa^{\vth}g_2\|_{L^\infty}^2,\notag
\end{align}
with $l\geq 2$ required.

On the other hand, from \eqref{abc-p1} and \eqref{abc-p2}, it follows
\begin{equation}\label{abc-12}
\|{[a_2,\Fb_2,c_2]}\|\leq \|{[a,\Fb,c]}\|+C\|w_lg_1\|_\infty,\ \|[a,\Fb]\|\leq C\|\na_x[a,\Fb]\|\leq C\|\na_xg_2\|+C\|w_l\na_xg_1\|_\infty,
\end{equation}
for $l>5/2.$ In addition, by \eqref{abc-in-eng}, we have for $|\vth|\leq N-1$
\begin{align}
|\ka_1\CE_a^{int}+\CE_\Fb^{int}+\CE_c^{int}|
\leq \|\na_x\pa^\vth[a,\Fb,c]\|^2+C\|\na_x\pa^\vth g_2\|^2+C\|w_l\na_x\pa^\vth g_1\|^2_\infty.\notag
\end{align}
Let $\ka_2>0$ be suitably small, then we define
$$
\CE_N^h(t)=\ka_2(\ka_1\CE_a^{int}+\CE_\Fb^{int}+\CE_c^{int})+\sum\limits_{1\leq|\vth|\leq N}\|\pa^{\vth}g_2\|^2,
$$
and hence there exist positive constants $\tilde{C}_1$ and $\tilde{C}_2$  such that
\begin{align}
\sum\limits_{1\leq|\vth|\leq N}&\|\pa^{\vth}g_2\|^2-\tilde{C}_1\sum\limits_{1\leq|\vth|\leq N}\|w_l\pa^\vth g_1\|^2_\infty\notag\\
\leq& \CE_N^h\leq \sum\limits_{1\leq|\vth|\leq N}\|\pa^{\vth}g_2\|^2
+\tilde{C}_2\sum\limits_{1\leq|\vth|\leq N}\|w_l\pa^\vth g_1\|^2_\infty.\notag%\label{eqv-norm}
\end{align}
By this, \eqref{2rd-dabc} and \eqref{g2-l2-p2} lead us to
\begin{align}
\frac{d}{dt}\CE_N^h(t)+\la\CE_N^h(t)
\leq C\al^2\sum\limits_{1\leq|\vth|\leq N}\|w_l\pa^{\vth}g_2\|_{L^\infty}^2+C\sum\limits_{1\leq|\vth|\leq N}\|w_l\pa^{\vth'} g_1\|^2_{L^\infty},\notag
\end{align}
which further gives
\begin{align}\label{g2-l2-p3}
\sup\limits_{0\leq s\leq t}&\|[a_2,b_2](s)\|^2+
\sup\limits_{0\leq s\leq t}\sum\limits_{1\leq|\vth|\leq N}\|\pa^{\vth}g_2(s)\|^2\notag\\
\leq &C\al^2\sup\limits_{0\leq s\leq t}\sum\limits_{1\leq|\vth|\leq N}\|w_l\pa^{\vth}g_2\|_{L^\infty}^2+
C\sup\limits_{0\leq s\leq t}\sum\limits_{1\leq|\vth|\leq N}\|w_l\pa^{\vth} g_1(s)\|^2_{L^\infty},
\end{align}
where \eqref{abc-12} has been used.

As a consequence, \eqref{g1-g2-hs} and \eqref{g2-l2-p3} imply
\begin{align}
\sup\limits_{0\leq s\leq t}\sum\limits_{1\leq |\vth|\leq N}\|w_l\pa^\vth g_{1}(s)\|_{L^\infty}
+\sup\limits_{0\leq s\leq t}\sum\limits_{1\leq |\vth|\leq N}\|w_l\pa^\vth g_{2}(s)\|_{L^\infty}
\leq C\sum\limits_{1\leq |\vth|\leq N}\|w_l\pa^\vth f_0\|_{L^\infty}.\label{g1-g2-hsto}
\end{align}
Thus, we get from \eqref{g12-lif-z} and \eqref{g1-g2-hsto} that
\begin{align}
\sup\limits_{0\leq s\leq t}&\|w_lg_{1}(s)\|_{L^\infty}+\al\sup\limits_{0\leq s\leq t}\|w_lg_{2}(s)\|_{L^\infty}
\notag\\&+\sup\limits_{0\leq s\leq t}\sum\limits_{1\leq |\vth|\leq N}\|w_l\pa^\vth g_{1}(s)\|_{L^\infty}
+\sup\limits_{0\leq s\leq t}\sum\limits_{1\leq |\vth|\leq N}\|w_l\pa^\vth g_{2}(s)\|_{L^\infty}
\leq& C\sum\limits_{|\vth|\leq N}\|w_l\pa^\vth f_0\|_{L^\infty}.\label{g1-g2-spcl}
\end{align}
\noindent\underline{{\it Step 2.4. The estimates for mixture derivatives.}} In this final sub-step, we shall deduce the $L^2$ estimates on $\pa_\ze^\vth g_2$ with $\ze>0$ and $|\ze|+|\vth|\leq N$. %Notice that $\pa_\ze^\vth g_2=\pa_\ze^\vth \FP_0g_2+\pa_\ze^\vth \FP_1g_2$ and $\|\pa_\ze^\vth \FP_0g_2\|\leq C\|\pa^\vth[a_2,\Fb_2,c_2]\|$, it suffices now to show the bound of $\|\pa_\ze^\vth \FP_1g_2\|$.
To see this, we first get from the inner product of $\pa_\ze^\vth\eqref{f2-eq-2}$ and $\pa_\ze^\vth g_2$ over $(x,v)\in\T^3\times\R^3$  that
\begin{align}%\label{g2-md-ip}
&(\pa_t \pa_\ze^\vth g_2,\pa_\ze^\vth g_2)+{(\pa_\ze^\vth(v\cdot\na_xg_2),\pa_\ze^\vth g_2)}-\beta(\pa_\ze^\vth(\na_v\cdot(vg_2)),\pa_\ze^\vth g_2)-\al(\pa_\ze^\vth\na_v\cdot(Avg_2),\pa_\ze^\vth g_2)\notag\\
&\qquad-\la_0(\pa_\ze^\vth g_2,\pa_\ze^\vth g_2)+(\pa_\ze^\vth Lg_2,\pa_\ze^\vth g_2)
=(\pa_\ze^\vth((1-\chi_M)\mu^{-\frac{1}{2}}\CK g_1),\pa_\ze^\vth g_2)+e^{-\la_0t}(\pa_\ze^\vth\Ga(g_2,g_2),\pa_\ze^\vth g_2),\notag
\end{align}
which gives
\begin{align}%\label{g2-md-es1}
\frac{d}{dt}\|\pa_\ze^\vth g_2\|^2+2(\de_1-\la_0)\|\pa_\ze^\vth g_2\|_\nu^2
\leq& C\|\pa^\vth g_2\|^2+
C\sum\limits_{|\ze'|+|\vth'|\leq N\atop{\ze'<\ze}}\|\pa_{\ze'}^{\vth'} g_2\|^2
+C(\al+\al^2)\sum\limits_{|\ze|+|\vth|\leq N}\|w_l\pa_\ze^\vth g_2\|_{L^\infty}^2\notag\\&+\eta\|\pa_\ze^\vth g_2\|^2
+C_\eta\sum\limits_{\ze'\leq\ze}\|w_l\pa_{\ze'}^\vth g_1\|_{L^\infty}^2,\notag
\end{align}
according to Lemma \ref{es-L} and Lemma \ref{Ga}. Thus, it follows by Gronwall's inequality
\begin{align}%\label{g2-md-es1}
\sup\limits_{0\leq s\leq t}\|\pa_\ze^\vth g_2(s)\|^2
\leq& C\sup\limits_{0\leq s\leq t}\|\pa^\vth g_2(s)\|^2+
C\sup\limits_{0\leq s\leq t}\sum\limits_{|\ze'|+|\vth'|\leq N\atop{\ze'<\ze}}\|\pa_{\ze'}^{\vth'} g_2\|^2
\notag\\&+C(\al+\al^2)\sup\limits_{0\leq s\leq t}\sum\limits_{|\ze|+|\vth|\leq N}\|w_l\pa_\ze^\vth g_2\|_{L^\infty}^2
+C\sup\limits_{0\leq s\leq t}\sum\limits_{\ze'\leq\ze}\|w_l\pa_{\ze'}^\vth g_1(s)\|_{L^\infty}^2,\notag
\end{align}
which further implies
\begin{align}\label{g2-md-es1}
\sup\limits_{0\leq s\leq t}\sum\limits_{|\ze|+|\vth|\leq N\atop{\ze>0}}\|\pa_\ze^\vth g_2(s)\|^2
\leq& C\sum\limits_{|\vth|\leq N}\sup\limits_{0\leq s\leq t}\|\pa^\vth g_2(s)\|^2
+C\sup\limits_{0\leq s\leq t}\sum\limits_{|\ze|+|\vth|\leq N}\|w_l\pa_{\ze}^\vth g_1(s)\|_{L^\infty}^2\notag\\&+C(\al+\al^2)\sum\limits_{|\ze|+|\vth|\leq N}\|w_l\pa_\ze^\vth g_2\|_{L^\infty}^2.
\end{align}
On the other hand, from \eqref{g1k-es2} and \eqref{g2k-es2}, it follows
\begin{align}\label{g1k-es3}
\sup\limits_{0\leq s\leq t}\sum\limits_{|\ze|+|\vth|\leq N\atop{\ze>0}}\|w_l\pa_\ze^{\vth}g_{1}(s)\|_{L^\infty}
\leq&\sum\limits_{|\ze|+|\vth|\leq N}\|w_l\pa_\ze^{\vth}f_0\|_{L^\infty}+
C\al\sup\limits_{0\leq s\leq t}\sum\limits_{|\ze|+|\vth|\leq N\atop{\ze>0}} \|w_l\pa_\ze^{\vth}g_{2}(s)\|_{L^\infty}
\notag\\&+C\al\sup\limits_{0\leq s\leq t}\sum\limits_{|\vth|\leq N}\|w_l\pa^\vth g_{2}(s)\|_{L^\infty},
\end{align}
and
\begin{align}\label{g2k-es3}
\sup\limits_{0\leq s\leq t}\sum\limits_{|\ze|+|\vth|\leq N\atop{\ze>0}}\|w_l\pa_\ze^{\vth}g_{2}(s)\|_{L^\infty}\leq&
C\sup\limits_{0\leq s\leq t}\sum\limits_{|\ze|+|\vth|\leq N\atop{\ze>0}}\|w_l\pa_\ze^{\vth}g_{1}(s)\|_{L^\infty}
+C\sup\limits_{0\leq s\leq t}\sum\limits_{|\ze|+|\vth|\leq N\atop{\ze>0}}\|\pa_\ze^{\vth}g_2(s)\|
\notag\\&+C\sup\limits_{0\leq s\leq t}\sum\limits_{|\vth|\leq N}|w_l\pa^\vth g_{1}(s)\|_{L^\infty}+C\sup\limits_{0\leq s\leq t}\sum\limits_{|\vth|\leq N}\|w_l\pa^\vth g_{2}(s)\|_{L^\infty}.
\end{align}
Now \eqref{g1-g2-spcl}, \eqref{g2-md-es1}, \eqref{g1k-es3} and \eqref{g2k-es3} lead us to
\begin{align}
\sup\limits_{0\leq s\leq t}\sum\limits_{|\ze|+|\vth|\leq N}\|w_l\pa_\ze^{\vth}g_{1}(s)\|_{L^\infty}+\al\sup\limits_{0\leq s\leq t}\sum\limits_{|\ze|+|\vth|\leq N}\|w_l\pa_\ze^{\vth}g_{2}(s)\|_{L^\infty}
\leq& C\sum\limits_{|\ze|+|\vth|\leq N}\|w_l\pa_\ze^\vth f_0\|_{L^\infty}.\notag%\label{fin-es}\notag
\end{align}
Thus, \eqref{ap-es} is valid and this also confirms \eqref{decay}.

Finally, by the similar procedure as that of \cite[Step 4, pp.47]{DL-2020}, one can show that the solution of \eqref{G-ust} and \eqref{G-id} is non-negative, and the details of the proof is omitted for brevity. This ends the proof of Theorem \ref{ge.th}.
\end{proof}

\section{Appendix}\label{pre-sec}

In this section, we provide those estimates that have been used in the previous sections. We will first give the basic estimates on the linearized operator $L$ as well as the nonlinear operators $\Ga$ and $Q$, then present a key estimate for the operator $\CK$ in the case of hard potentials, and in the end derive a lower bound for a matrix exponential.

The following lemma is concerned with the integral operator $K$ given by \eqref{sp.L}, and its proof in case of the hard sphere model $(\ga=1)$ has been given by \cite[Lemma 3, pp.727]{Guo-2010}.
\begin{lemma}\label{Kop}
Let $K$ be defined as \eqref{sp.L}, then it holds that
\begin{align}\notag
Kf(v)=\int_{\R^3}\Fk(v,v_\ast)f(v_\ast)\,dv_\ast
\end{align}
with
\begin{equation*}
|\Fk(v,v_\ast)|\leq C\{|v-v_\ast|^\ga+|v-v_\ast|^{-2+\ga}\}e^{-
\frac{1}{8}|v-v_\ast|^{2}-\frac{1}{8}\frac{\left||v|^{2}-|v_\ast|^{2}\right|^{2}}{|v-v_\ast|^{2}}}. %\label{grad}
\end{equation*}
Moreover, let
\begin{align}%\label{kw-def}
\Fk_w(v,v_\ast)=w_{l}(v)\Fk(v,v_\ast)w_{l}^{-1}(v_\ast)\notag
\end{align}
with  $l\geq0$,
then it also holds that
\begin{equation*}
\int_{\R^3} \Fk_w(v,v_\ast)e^{\frac{\varepsilon|v-v_\ast|^2}{8}}dv_\ast\leq \frac{C}{1+|v|},
\end{equation*}
for $\varepsilon=0$ or any $\varepsilon> 0$ small enough.
\end{lemma}

For the velocity weighted derivative estimates on the nonlinear operator $\Ga$, one has
\begin{lemma}\label{Ga}
Let $0\leq \ga\leq 1$ and $\ta\in[0,1]$. For any $p\in[1,+\infty]$ and any $l\geq0$, it holds that
\begin{align}\label{es1.Ga}
\|w_l\nu^{-\ta}\pa_\ze\Ga(f,g)\|_{L_v^p}\leq C\sum\limits_{\ze'+\ze''\leq \ze}
\left\{\|w_l\nu^{1-\ta}\pa_{\ze'}f\|_{L_v^p}\|\pa_{\ze''}g\|_{L_v^p}+\|\pa_{\ze'}f\|_{L_v^p}\|w_l\nu^{1-\ta}\pa_{\ze''}g\|_{L_v^p}\right\}.
\end{align}
\end{lemma}
\begin{proof}
Note that if $l=0$ and $\ze=0$, \eqref{es1.Ga} was given by \cite[Theorem 1.2.3, pp.15]{UY-lect}. Let us now show that \eqref{es1.Ga} can be generalized to $l\geq0$ and $\ze\geq0$. For this, we first have from
definition \eqref{def.Ga} that
\begin{align}
\pa_{\ze}\Ga(f,g)=&\pa_{\ze}\int_{\R^3}\int_{\S^2}B_0(\cos\ta)|v-v_\ast|^\ga\mu^{1/2}(v_\ast)f(v_\ast')g(v')\,d\om dv_\ast
\notag\\&-\pa_{\ze}\int_{\R^3}\int_{\S^2}B_0(\cos\ta)|v-v_\ast|^\ga\mu^{1/2}(v_\ast)f(v_\ast)g(v)\,d\om dv_\ast\notag\\
=&\pa_{\ze}\int_{\R^3}\int_{\S^2}B_0(\cos\ta)|v-v_\ast|^\ga\mu^{1/2}(v_\ast)f(v_\ast')g(v')\,d\om dv_\ast
\notag\\&-c_0\pa_{\ze} \left[g(v)\int_{\R^3}|v-v_\ast|^\ga\mu^{1/2}(v_\ast)f(v_\ast)\, dv_\ast\right],
\notag
\end{align}
where we have used $\int_{\S^2}B_0(\cos\ta)\,d\omega=c_0$ for a constant $c_0>0$.
Then, by a change of variable $v_\ast-v\rightarrow u$, one has
\begin{align}
\pa_{\ze}\Ga(f,g)=&\underbrace{\sum\limits_{\ze''\leq\ze'\leq\ze}C_{\ze}^{\ze'}C_{\ze'}^{\ze''}\int_{\R^3}\int_{\S^2}B_0(\cos\ta)
|u|^\ga(\pa_{\ze-\ze'}\mu^{1/2})(u+v)(\pa_{\ze'-\ze''}f)(v+u_\perp)(\pa_{\ze''}g)(v+u_\parallel)\,d\om du}_{\Ga_1}
\notag\\&\underbrace{-c_0\sum\limits_{\ze'\leq\ze}C_{\ze}^{\ze'}
\left[(\pa_{\ze'}g)(v)\int_{\R^3}|u|^\ga(\pa_{\ze-\ze'}\mu^{1/2})(v+u)(\pa_{\ze'}f)(v+u)\, du\right]}_{\Ga_2},
\label{der-nop}
\end{align}
where $u_\parallel=(u\cdot \om)\om$ and $u_\perp=u-u_\parallel$. As $|(\pa_{\ze-\ze'}\mu^{1/2})(u+v)|\leq C\mu^{1/4}(u+v)$,
one has by changing variable $u$ back to $v_\ast-v$ that
\begin{align}
|\Ga_1|\leq& C \int_{\R^3}\int_{\S^2}B_0(\cos\ta)
|u|^\ga\mu^{1/4}(u+v)|(\pa_{\ze'-\ze''}f)(v+u_\perp)(\pa_{\ze''}g)(v+u_\parallel)|\,d\om du
\notag\\
=&C \int_{\R^3}\int_{\S^2}B_0(\cos\ta)
|v_\ast-v|^\ga\mu^{1/4}(v_\ast)|(\pa_{\ze'-\ze''}f)(v_\ast')(\pa_{\ze''}g)(v')|\,d\om dv_\ast,\notag
\end{align}
which together with the inequality
\begin{align}
(w_l\nu^{-\ta+1})(v)\leq C\left( (w_l\nu^{-\ta+1})(v')+(w_l\nu^{-\ta+1})(v'_\ast)\right),\label{tri-ine}
\end{align}
implies
\begin{align}
w_l\nu^{-\ta}|\Ga_1|\leq& C \left((w_l\nu^{-\ta+1})(v')+(w_l\nu^{-\ta+1})(v'_\ast)\right)\nu^{-1}(v)
\notag\\&\times
\int_{\R^3}\int_{\S^2}B_0(\cos\ta)
|v_\ast-v|^\ga\mu^{1/4}(v_\ast)|(\pa_{\ze'-\ze''}f)(v_\ast')(\pa_{\ze''}g)(v')|\,d\om dv_\ast\notag\\
\leq&C\{\|w_l\nu^{-\ta+1}\pa_{\ze'-\ze''}f\|_{L^\infty}\|\pa_{\ze''}g\|_{L^\infty}
+\|\pa_{\ze'-\ze''}f\|_{L^\infty}\|w_l\nu^{-\ta+1}\pa_{\ze''}g\|_{L^\infty}\}\notag\\&\times
\nu^{-1}(v)\int_{\R^3}\int_{\S^2}B_0(\cos\ta)
|v_\ast-v|^\ga\mu^{1/4}(v_\ast)d\om dv_\ast
\notag\\
\leq&C\{\|w_l\nu^{-\ta+1}\pa_{\ze'-\ze''}f\|_{L^\infty}\|\pa_{\ze''}g\|_{L^\infty}
+\|\pa_{\ze'-\ze''}f\|_{L^\infty}\|w_l\nu^{-\ta+1}\pa_{\ze''}g\|_{L^\infty}\}.\notag
\end{align}
This confirms the $L^\infty$ estimate for $\Ga_1.$
If $p\in[1,\infty)$, by H\"{o}lder's inequality, we get
\begin{align}
w_l\nu^{-\ta}|\Ga_1|
\leq&C_\ze w_l\nu^{-\ta}
\left(\int_{\R^3}\int_{\S^2}B_0(\cos\ta)
|v_\ast-v|^{p'\ga}{\mu^{p'/4}(v_\ast)\,d\om dv_\ast}\right)^{\frac{1}{p'}}
\notag\\
&\times
\left(\int_{\R^3}\int_{\S^2}{B_0(\cos\ta)\left|(\pa_{\ze'-\ze''}f)(v_\ast')(\pa_{\ze''}g)(v')\right|^{p}\,d\om dv_\ast}\right)^{\frac{1}{p}}
\notag\\
\leq&C_\ze w_l\nu^{1-\ta}
\left(\int_{\R^3}|(\pa_{\ze'-\ze''}f)(v_\ast')(\pa_{\ze''}g)(v')|^{p}\,dv_\ast\right)^{\frac{1}{p}},\notag
\end{align}
%$ \left( w_l^p\nu^{-p\ta+p}(v)+w_l^p\nu^{-p\ta+p}(v_\ast)\right)$
where $\frac{1}{p}+\frac{1}{p'}=1.$
Therefore, using \eqref{tri-ine} again and by a change of variable $(v',v_\ast')\rightarrow(v,v_\ast)$, one has
\begin{align}
\|w_l\nu^{-\ta}\Ga_1\|^p_{L^p}\leq& \int_{\R^3}w_l^p\nu^{-p\ta+p}\int_{\R^3}
|(\pa_{\ze'-\ze''}f)(v_\ast')(\pa_{\ze''}g)(v')|^{p}\,dv_\ast dv
\notag\\
=& \int_{\R^3}\int_{\R^3}(w_l^p\nu^{-p\ta+p})(v'_\ast)
|(\pa_{\ze'-\ze''}f)(v_\ast)(\pa_{\ze''}g)(v)|^{p}\,dv_\ast dv\notag\\
\leq& \int_{\R^3}\int_{\R^3}[(w_l^p\nu^{-p\ta+p})(v)+(w_l^p\nu^{-p\ta+p})(v_\ast)]
|(\pa_{\ze'-\ze''}f)(v_\ast)(\pa_{\ze''}g)(v)|^{p}\,dv_\ast dv
\notag\\
\leq&C\left\{ \|w_l\nu^{-\ta}\pa_{\ze'-\ze''}f\|_{L^p_v}^p\|\pa_{\ze''}g\|_{L^p_v}^p
+\|\pa_{\ze'-\ze''}f\|_{L^p_v}^p\|w_l\nu^{-\ta}\pa_{\ze''}g\|_{L^p_v}^p\right\}.\notag
\end{align}
The corresponding estimates for $\Ga_2$ are similar and easier, so we omit them for brevity. This completes the proof of Lemma \ref{Ga}.
\end{proof}

The following lemma is concerned with  coercivity estimates for the linear collision operator $L$.

\begin{lemma}\label{es-L}
Let $0\leq \ga\leq1$, then there is a constant $\de_0>0$ such that
\begin{align}\label{bL}
\lag Lf,f\rag=\lag L\FP_1f,\FP_1f\rag\geq\de_0\|\FP_1f\|_\nu^2,
\end{align}
where $\|\cdot\|_\nu=\|\nu^{\frac{1}{2}}\cdot\|.$
%and for $q>0$ it holds
%\begin{align}\label{w-L}
%\lag w^2_qL f,f\rag\geq\de_0\|w_q f\|^2-C\|f\|^2.
%\end{align}
Moreover, there are constants $\de_1>0$ and $C>0$ such that for $|\ze|>0$
\begin{align}\label{wd-L}
\lag \pa_{\ze} L f,\pa_\ze  f\rag\geq\de_1\|\pa_\ze f\|_\nu^2-C\|f\|^2.
\end{align}
\end{lemma}

\begin{proof}
Note that \eqref{bL} has been already proved in \cite[Lemma 3.2, pp.638]{Guo-2006}. As for \eqref{wd-L}, from \cite[Lemma 3.3, pp.639]{Guo-2006}, we have %provides us
\begin{align}%\label{wd-L-p1}
\lag \pa_{\ze} L f,\pa_\ze  f\rag\geq\de_1\|\pa_\ze f\|_\nu^2-C\|f\|_\nu^2.\notag
\end{align}
We now prove that this can be relaxed to \eqref{wd-L}, which is indeed true for Maxwell molecular case because $\nu\sim c_0$ for some $c_0>0$ in this situation. For $0<\ga\leq1$, we write
\begin{align}\label{wd-L-p1}
\lag \pa_{\ze} L f,\pa_\ze  f\rag=&\lag \pa_{\ze} (\nu f),\pa_\ze  f\rag-\lag \pa_{\ze} (K f),\pa_\ze  f\rag
\notag\\=&\lag L\pa_{\ze} f,\pa_\ze  f\rag+\sum\limits_{0<\ze'\leq\ze}C_\ze^{\ze'}\lag \pa_{\ze'} \nu \pa_{\ze-\ze'}f,\pa_\ze  f\rag
-\sum\limits_{0<\ze'\leq\ze}C_\ze^{\ze'}\lag (\pa_{\ze'} K) \pa_{\ze-\ze'}f,\pa_\ze  f\rag.
\end{align}
From \eqref{bL}, one has
\begin{align}\label{L-coes}
\lag L\pa_{\ze} f,\pa_\ze  f\rag\geq\de_0\|\FP_1\pa_\ze  f\|^2_\nu\geq\de_0\|\pa_\ze  f\|^2_\nu-\de_0\|\FP_0\pa_\ze  f\|^2_\nu
\geq\de_0\|\pa_\ze  f\|^2_\nu-C\|f\|^2.
\end{align}
By definition \eqref{sp.L}, it follows
\begin{align}
{\bf1}_{\ze'>0}|\pa_{\ze'} \nu|\leq C(1+|v|)^{\ga-|\ze'|}\leq C.\notag
\end{align}
Thus, one has by Cauchy-Schwarz's inequality with $\eta>0$ and Sobolev's interpolation inequality that
\begin{align}
|\lag \pa_{\ze'} \nu \pa_{\ze-\ze'}f,\pa_\ze  f\rag|\leq \eta\|\pa_\ze  f\|^2+C_\eta\|f\|^2.\label{nu-es}
\end{align}
Next, in view of \eqref{fk-12}, we have by a change of variable $v_\ast-v\rightarrow u$
\begin{align}
(\pa_{\ze'} K) \pa_{\ze-\ze'}f=&\tilde{c}_1\sum\limits_{0\leq\ze'\leq\ze}C_\ze^{\ze'}\int_{\R^3\times\om^2}B_0(\cos\ta)|u|^\ga \pa_{\ze'}\{e^{-\frac{|u+v|^2+|v|^2}{4}}\}\pa_{\ze-\ze'}f(v+u)du\notag\\
&-\tilde{c}_2\sum\limits_{0\leq\ze'\leq\ze}C_\ze^{\ze'}\int_{\R^3\times\om^2}B_0(\cos\ta)|u|^\ga \pa_{\ze'}\{e^{-
\frac{1}{8}|u|^{2}-\frac{1}{8}\frac{\left|2v\cdot u+|u|^2\right|^{2}}{|u|^{2}}}\}\pa_{\ze-\ze'}f(v+u)du.\notag
\end{align}
Furthermore, direct computations give
\begin{align}
\pa_{\ze'}\{e^{-\frac{|u+v|^2+|v|^2}{4}}\}\leq C(\ze')e^{-\frac{|u+v|^2+|v|^2}{8}},\notag
\end{align}
and
\begin{align}
\pa_{\ze'}\{e^{-
\frac{1}{8}|u|^{2}-\frac{1}{8}\frac{\left|2v\cdot u+|u|^2\right|^{2}}{|u|^{2}}}\}
\leq C(\ze')e^{-
\frac{1}{16}|u|^{2}-\frac{1}{16}\frac{\left|2v\cdot u+|u|^2\right|^{2}}{|u|^{2}}},\notag
\end{align}
which further implies
\begin{align}
|(\pa_{\ze'} K) \pa_{\ze-\ze'}f|\leq C(\ze)\int_{\R^3}\bar{\Fk}(v,v_\ast)|\pa_{\ze-\ze'}f(v_\ast)|d v_\ast\notag
\end{align}
with
\begin{equation}\label{dK-ubd}
\bar{\Fk}(v,v_\ast)\leq C\{|v-v_\ast|^\ga+|v-v_\ast|^{-2+\ga}\}e^{-
\frac{1}{16}|v-v_\ast|^{2}-\frac{1}{16}\frac{\left||v|^{2}-|v_\ast|^{2}\right|^{2}}{|v-v_\ast|^{2}}}. %\label{grad}
\end{equation}
In particular,
\begin{align}
\int_{\R^3}\bar{\Fk}(v,v_\ast)dv\leq \frac{C}{1+|v_\ast|},\quad \int_{\R^3}\bar{\Fk}(v,v_\ast)dv_\ast\leq \frac{C}{1+|v|}.\notag
\end{align}
Therefore, by Cauchy-Schwarz's inequality, {Fubini's theorem} and Sobolev's interpolation inequality, we obtain
\begin{align}
|\lag (\pa_{\ze'} K) \pa_{\ze-\ze'}f,\pa_\ze f\rag|\leq& \eta\|\pa_\ze f\|^2
+C_\eta\int_{\R^3}\left(\int_{\R^3}\bar{\Fk}(v,v_\ast)|\pa_{\ze-\ze'}f(v_\ast)|d v_\ast\right)^2dv\notag\\
\leq& \eta\|\pa_\ze f\|^2
+C_\eta\int_{\R^3}\int_{\R^3}\bar{\Fk}(v,v_\ast)d v_\ast \int_{\R^3}\bar{\Fk}(v,v_\ast)|\pa_{\ze-\ze'}f(v_\ast)|^2d v_\ast dv
\notag\\
\leq& \eta\|\pa_\ze f\|^2
+C_\eta\int_{\R^3}\int_{\R^3}\bar{\Fk}(v,v_\ast)dv|\pa_{\ze-\ze'}f(v_\ast)|^2d v_\ast
\notag\\
\leq& 2\eta\|\pa_\ze f\|^2+C_\eta\|f\|^2.\label{fk-es}
\end{align}
Finally, plugging \eqref{L-coes}, \eqref{nu-es} and \eqref{fk-es} into \eqref{wd-L-p1} gives \eqref{wd-L}.
This ends the proof of Lemma \ref{es-L}.

\end{proof}
\begin{remark}
From \eqref{dK-ubd}, one can justify that $\pa_{\ze'} K$ is a compact operator from $H^{|\ze|}$ to $H^{|\ze|}$, which directly implies \eqref{fk-es}, cf. \cite[Lemma 2.2, pp.1109]{Guo-vpb}.
\end{remark}

Next, the following lemma which was proved in \cite[Proposition 3.1, pp.13]{DL-2020} gives the $L^\infty$ estimates of the solutions in the case of Maxwell molecule model.
\begin{lemma}\label{CK}
Let $\ga=0$ and $\CK$ be given by \eqref{sp.cL}, then for any nonnegative integer $|\ze|\geq 0$, there is $C>0$ such that for any arbitrarily large $l>0$, there is $M=M(l)>0$ such that it holds that
\begin{align}\label{CK1}
\sup_{|v|\geq M} w_{l}|\pa_{\ze}(\CK f)|\leq \frac{C}{l} \sum\limits_{0\leq \ze'\leq \ze}\|w_{l}\pa_{\ze'}f\|_{L^\infty}.
\end{align}
In particular, one can choose $M=l^2$.
%where $C>0$ is independent of $l$.
%for $l$ is positive and suitably large.
\end{lemma}

In the case of $0<\ga\leq1$, the following lemma which can be found in
\cite[Proposition 3.1, pp.397]{AEP-87} enables us to gain the smallness property of $\CK$ at large velocity.
\begin{lemma}\label{g-ck}
Let $0\leq\ga\leq1$ and $l>4$, then there exists a function $\varsigma(l)$ which satisfies $\varsigma(l)\rightarrow0$ as $l\rightarrow+\infty$
such that
\begin{align}
w_l\{|Q_{\rm{loss}}(f,g)|&+|Q_{\rm{gain}}(f,g)|+|Q_{\rm{gain}}(g,f)|\}\notag\\
\leq& \|w_lf\|_{L^\infty}\{C(l)\|w_{l+\ga/2}g\|_{L^\infty}+\varsigma(l)\|w_{3}g\|_{L^\infty}(1+|v|)^\ga\},\label{Q-lf2}
\end{align}
where $Q_{\textrm{loss}}$ denotes the negative part of $Q$ in \eqref{Q-op}.
\end{lemma}

The following result is a direct consequence of Lemma \ref{g-ck}.

\begin{lemma}\label{g-ck-lem}
Let $0<\ga\leq1$, then there is a constant $C>0$ such that for any arbitrarily large $l>0$, there are sufficiently large $M=M(l)>0$ and suitably small $\varsigma=\varsigma(l)>0$ such that it holds that
\begin{align}\label{CK2}
\sup_{|v|\geq M}\nu^{-1} w_{l}|\CK f|\leq C\{(1+M)^{-\ga/2}+\varsigma\}\|w_{l}f\|_{L^\infty}.
\end{align}
\end{lemma}

\begin{proof}
Recall the definition \eqref{sp.cL} for $\CK$. Let $g=\mu$ in \eqref{Q-lf2}, then we obtain
\begin{align}
\nu^{-1}w_{l}|\CK f|\leq C(l)\nu^{-1}\|w_lf\|_{L^\infty}+\varsigma(l)\|w_lf\|_{L^\infty}.\label{CK3}
\end{align}
Noticing that $\varsigma(l)=\frac{1}{l}$ according to the proof in \cite[Proposition 3.1, pp.397]{AEP-87}, we first choose $l$ to be suitably large so that $\varsigma$ is small enough, then we set $M>0$ to be sufficiently large such that $C(l)(1+M)^{-\ga/2}\leq C$ thanks to $\ga>0$.
Then \eqref{CK2} follows from \eqref{CK3}. This concludes the proof of Lemma \ref{g-ck-lem}.
\end{proof}

The following Lemma concerning the polynomial weighted estimates on the collision operator $Q$ can be verified by using a parallel argument as for obtaining \cite[Proposition 3.1, pp.397]{AEP-87}.

\begin{lemma}\label{op.es.lem}For $l>4$ and $\ga\geq0$, then it holds that
\begin{equation}\notag
|w_{l} \nu^{-1}\pa_\ze Q(F_1,F_2)|\leq C\sum\limits_{\ze'+\ze''\leq\ze}\|w_{l} \pa_{\ze'} F_1\|_{L^\infty}\|w_{l} \pa_{\ze''}F_2\|_{L^\infty}.
\end{equation}
\end{lemma}

Finally, we give a technical lemma on the determinant of a matrix exponential and we omit the proof for brevity.

\begin{lemma}\label{expm-lem}
Let $\CM=\al\bar{\CM}$, where $\bar{\CM}=(\bar{a}_{ij})
\in M_{3\times 3}(\R)$ is an invertible constant matrix with $\max\{|\bar{a}_{ij}|\}=C_\CM$, and $\al>0$ is suitably small.
\begin{itemize}
\item[(i)]
If $\frac{1}{3}\geq\eta>0$, then it holds that
\begin{align*}
\left| |\CM^{-1}||e^{\eta\CM}-I|\right|\geq \frac{\eta^3}{8}.
%\label{exma-lbd}
\end{align*}
\item[(ii)] Let $v\in\R^3$ be a vector satisfying $|v|\leq M$ with $M>0$, then for any $\eta>0$, it holds that
\begin{align}
\left| \CM^{-1}\left\{e^{\eta\CM}-I\right\}v\right|\leq \eta Me^{3C_\CM\al\eta}.\label{exma-ubd}
\end{align}
%for some $C>0.$
\end{itemize}
\end{lemma}

\medskip
\noindent {\bf Acknowledgements:}
Renjun Duan's research was partially supported by the General Research Fund (Project No.~14301720) from RGC of Hong Kong and the Direct Grant (4053397) from CUHK. Shuangqian Liu's research was supported by grants from the National Natural Science Foundation of China (contracts:~11971201 and 11731008), and Hong Kong Institute for Advanced Study No.~9360157.

%\newpage

\end{document}